\documentclass[10pt]{amsart}
\usepackage{texmaX,semmaX,semtkzX}
 \usepackage{makecell}

\def\clgBalc{\clusterpicture            
  \Root {1} {first} {r1};
  \Root {} {r1} {r2};
  \Root {} {r2} {r3};
  \Root {} {r3} {r4};
  \Root {} {r4} {r5};
  \Root {} {r5} {r6};
  \Cluster c1 = (r1)(r2)(r3)(r4)(r5)(r6);
\endclusterpicture}

\def\clgCc{\clusterpicture            
  \Root {1} {first} {r1};
  \Root {1} {r1} {r2};
  \Root {} {r2} {r3};
  \Root {} {r3} {r4};
  \Root {} {r4} {r5};
  \Root {} {r5} {r6};
  \Cluster c1 = (r2)(r3)(r4)(r5)(r6);
  \Cluster c2 = (r1)(c1);
\endclusterpicture}

\def\clnBalce{\clusterpicture            
  \Root {1} {first} {r1};
  \Root {} {r1} {r2};
  \Root {} {r2} {r3};
  \Root {} {r3} {r4};
  \Root {1} {r4} {r5};
  \Root {} {r5} {r6};
  \ClusterLD c1[{\ee}][{\nn}] = (r5)(r6);
  \Cluster c2 = (r1)(r2)(r3)(r4)(c1);
\endclusterpicture}

\def\clnCce{\clusterpicture            
  \Root {1} {first} {r1};
  \Root {} {r1} {r2};
  \Root {1} {r2} {r3};
  \Root {} {r3} {r4};
  \Root {} {r4} {r5};
  \Root {} {r5} {r6};
  \ClusterLD c1[{\ee}][{\nn}] = (r3)(r4)(r5)(r6);
  \Cluster c2 = (r1)(r2)(c1);
\endclusterpicture}

\def\clnDce{\clusterpicture            
  \Root {2} {first} {r1};
  \Root {} {r1} {r2};
  \ClusterD c1[{\nn\!-\!t}] = (r1)(r2);
  \Root {1} {c1} {r3};
  \Root {} {r3} {r4};
  \Root {} {r4} {r5};
  \Root {} {r5} {r6};
  \ClusterD c2[t] = (r3)(r4)(r5)(r6);
  \ClusterL c3[{\ee}] = (c1)(c2);
\endclusterpicture}

\def\clnEce{\clusterpicture            
  \Root {1} {first} {r1};
  \Root {1} {r1} {r2};
  \Root {} {r2} {r3};
  \Root {} {r3} {r4};
  \Root {1} {r4} {r5};
  \Root {} {r5} {r6};
  \ClusterLD c1[{\ee}][{\nn}] = (r5)(r6);
  \Cluster c2 = (r2)(r3)(r4)(c1);
  \Cluster c3 = (r1)(c2);
\endclusterpicture}

\def\clnGce{\clusterpicture            
  \Root {1} {first} {r1};
  \Root {1} {r1} {r2};
  \Root {1} {r2} {r3};
  \Root {} {r3} {r4};
  \Root {} {r4} {r5};
  \Root {} {r5} {r6};
  \ClusterLD c1[{\ee}][{\nn}] = (r3)(r4)(r5)(r6);
  \Cluster c2 = (r2)(c1);
  \Cluster c3 = (r1)(c2);
\endclusterpicture}

\def\clnmBalce{\clusterpicture            
  \Root {1} {first} {r1};
  \Root {} {r1} {r2};
  \Root {1} {r2} {r3};
  \Root {} {r3} {r4};
  \ClusterLD c1[{\ee}][{\nn}] = (r3)(r4);
  \Root {1} {c1} {r5};
  \Root {} {r5} {r6};
  \ClusterLD c2[{\dd}][{\mm}] = (r5)(r6);
  \Cluster c3 = (r1)(r2)(c1)(c2);
\endclusterpicture}

\def\clnmCce{\clusterpicture            
  \Root {1} {first} {r1};
  \Root {} {r1} {r2};
  \Root {1} {r2} {r3};
  \Root {} {r3} {r4};
  \Root {1} {r4} {r5};
  \Root {} {r5} {r6};
  \ClusterLD c1[{\ee}][{\nn}] = (r5)(r6);
  \ClusterLD c2[{\dd}][{\mm}] = (r3)(r4)(c1);
  \Cluster c3 = (r1)(r2)(c2);
\endclusterpicture}

\def\clnmDce{\clusterpicture            
  \Root {2} {first} {r1};
  \Root {} {r1} {r2};
  \ClusterD c1[{\mm\!-\!t}] = (r1)(r2);
  \Root {1} {c1} {r3};
  \Root {} {r3} {r4};
  \Root {1} {r4} {r5};
  \Root {} {r5} {r6};
  \ClusterLD c2[{\ee}][{\nn}] = (r5)(r6);
  \ClusterD c3[t] = (r3)(r4)(c2);
  \ClusterL c4[{\dd}] = (c1)(c3);
\endclusterpicture}

\def\clnmEce{\clusterpicture            
  \Root {1} {first} {r1};
  \Root {1} {r1} {r2};
  \Root {1} {r2} {r3};
  \Root {} {r3} {r4};
  \ClusterLD c1[{\ee}][{\nn}] = (r3)(r4);
  \Root {1} {c1} {r5};
  \Root {} {r5} {r6};
  \ClusterLD c2[{\dd}][{\mm}] = (r5)(r6);
  \Cluster c3 = (r2)(c1)(c2);
  \Cluster c4 = (r1)(c3);
\endclusterpicture}

\def\clnmGce{\clusterpicture            
  \Root {1} {first} {r1};
  \Root {1} {r1} {r2};
  \Root {1} {r2} {r3};
  \Root {} {r3} {r4};
  \Root {1} {r4} {r5};
  \Root {} {r5} {r6};
  \ClusterLD c1[{\ee}][{\nn}] = (r5)(r6);
  \ClusterLD c2[{\dd}][{\mm}] = (r3)(r4)(c1);
  \Cluster c3 = (r2)(c2);
  \Cluster c4 = (r1)(c3);
\endclusterpicture}

\def\clnnBalce{\clusterpicture            
  \Root {1} {first} {r1};
  \Root {} {r1} {r2};
  \Root {1} {r2} {r3};
  \Root {} {r3} {r4};
  \ClusterL c1[{\eta}] = (r3)(r4);
  \Root {1} {c1} {r5};
  \Root {} {r5} {r6};
  \ClusterLD c2[{\ee\eta}][{\nn}] = (r5)(r6);
  \Cluster c3 = (r1)(r2)(c1)(c2);
  \frob(c1t)(c2);
\endclusterpicture}

\def\clnnCce{\clusterpicture            
  \Root {1} {first} {r1};
  \Root {1} {r1} {r2};
  \Root {1} {r2} {r3};
  \Root {} {r3} {r4};
  \ClusterL c1[{\eta}] = (r3)(r4);
  \Root {1} {c1} {r5};
  \Root {} {r5} {r6};
  \ClusterLD c2[{\ee\eta}][{\nn}] = (r5)(r6);
  \Cluster c3 = (r2)(c1)(c2);
  \Cluster c4 = (r1)(c3);
  \frob(c1t)(c2);
\endclusterpicture}

\def\clUnmkBalce{\clusterpicture            
  \Root {2} {first} {r1};
  \Root {} {r1} {r2};
  \ClusterD c1[{\nn}] = (r1)(r2);
  \Root {1} {c1} {r3};
  \Root {} {r3} {r4};
  \ClusterD c2[{\mm}] = (r3)(r4);
  \Root {1} {c2} {r5};
  \Root {} {r5} {r6};
  \ClusterD c3[{\kk}] = (r5)(r6);
  \ClusterL c4[{\ee}] = (c1)(c2)(c3);
\endclusterpicture}

\def\clUnmkCce{\clusterpicture            
  \Root {1} {first} {r1};
  \Root {} {r1} {r2};
  \Root {2} {r2} {r3};
  \Root {} {r3} {r4};
  \ClusterD c1[{\nn}] = (r3)(r4);
  \Root {1} {c1} {r5};
  \Root {} {r5} {r6};
  \ClusterD c2[{\mm}] = (r5)(r6);
  \ClusterLD c3[{\ee}][{\kk}] = (c1)(c2);
  \Cluster c4 = (r1)(r2)(c3);
\endclusterpicture}

\def\clUnmkDce{\clusterpicture            
  \Root {2} {first} {r1};
  \Root {} {r1} {r2};
  \ClusterD c1[{\kk\!-\!t}] = (r1)(r2);
  \Root {2} {c1} {r3};
  \Root {} {r3} {r4};
  \ClusterD c2[{\nn}] = (r3)(r4);
  \Root {1} {c2} {r5};
  \Root {} {r5} {r6};
  \ClusterD c3[{\mm}] = (r5)(r6);
  \ClusterD c4[t] = (c2)(c3);
  \ClusterL c5[{\ee}] = (c1)(c4);
\endclusterpicture}

\def\clUnmkEce{\clusterpicture            
  \Root {1} {first} {r1};
  \Root {1} {r1} {r2};
  \Root {2} {r2} {r3};
  \Root {} {r3} {r4};
  \ClusterD c1[{\nn}] = (r3)(r4);
  \Root {1} {c1} {r5};
  \Root {} {r5} {r6};
  \ClusterD c2[{\mm}] = (r5)(r6);
  \ClusterLD c3[{\ee}][{\kk}] = (c1)(c2);
  \Cluster c4 = (r2)(c3);
  \Cluster c5 = (r1)(c4);
\endclusterpicture}

\def\clUnnkBalce{\clusterpicture            
  \Root {2} {first} {r1};
  \Root {} {r1} {r2};
  \Cluster c1 = (r1)(r2);
  \Root {1} {c1} {r3};
  \Root {} {r3} {r4};
  \ClusterD c2[{\nn}] = (r3)(r4);
  \Root {1} {c2} {r5};
  \Root {} {r5} {r6};
  \ClusterD c3[{\kk}] = (r5)(r6);
  \ClusterL c4[{\ee}] = (c1)(c2)(c3);
  \frob(c1)(c2);
\endclusterpicture}

\def\clUnnkCce{\clusterpicture            
  \Root {1} {first} {r1};
  \Root {} {r1} {r2};
  \Root {2} {r2} {r3};
  \Root {} {r3} {r4};
  \Cluster c1 = (r3)(r4);
  \Root {1} {c1} {r5};
  \Root {} {r5} {r6};
  \ClusterD c2[{\nn}] = (r5)(r6);
  \ClusterLD c3[{\ee}][{\kk}] = (c1)(c2);
  \Cluster c4 = (r1)(r2)(c3);
  \frob(c1)(c2);
\endclusterpicture}

\def\clUnnkDce{\clusterpicture            
  \Root {2} {first} {r1};
  \Root {} {r1} {r2};
  \ClusterD c1[{\kk\!-\!t}] = (r1)(r2);
  \Root {2} {c1} {r3};
  \Root {} {r3} {r4};
  \Cluster c2 = (r3)(r4);
  \Root {1} {c2} {r5};
  \Root {} {r5} {r6};
  \ClusterD c3[{\nn}] = (r5)(r6);
  \ClusterD c4[t] = (c2)(c3);
  \ClusterL c5[{\ee}] = (c1)(c4);
  \frob(c2)(c3);
\endclusterpicture}

\def\clUnnkEce{\clusterpicture            
  \Root {1} {first} {r1};
  \Root {1} {r1} {r2};
  \Root {2} {r2} {r3};
  \Root {} {r3} {r4};
  \Cluster c1 = (r3)(r4);
  \Root {1} {c1} {r5};
  \Root {} {r5} {r6};
  \ClusterD c2[{\nn}] = (r5)(r6);
  \ClusterLD c3[{\ee}][{\kk}] = (c1)(c2);
  \Cluster c4 = (r2)(c3);
  \Cluster c5 = (r1)(c4);
  \frob(c1)(c2);
\endclusterpicture}

\def\clUnnnBalce{\clusterpicture            
  \Root {2} {first} {r1};
  \Root {} {r1} {r2};
  \Cluster c1 = (r1)(r2);
  \Root {1} {c1} {r3};
  \Root {} {r3} {r4};
  \Cluster c2 = (r3)(r4);
  \Root {1} {c2} {r5};
  \Root {} {r5} {r6};
  \ClusterD c3[{\nn}] = (r5)(r6);
  \ClusterL c4[{\ee}] = (c1)(c2)(c3);
  \frob(c1)(c2);
  \frob(c2)(c3);
\endclusterpicture}

\def\clInImBalce{\clusterpicture            
  \Root {2} {first} {r1};
  \Root {1} {r1} {r2};
  \Root {} {r2} {r3};
  \ClusterLD c1[{\ee}][{\nn}] = (r2)(r3);
  \ClusterD c2[t] = (r1)(c1);
  \Root {1} {c2} {r4};
  \Root {1} {r4} {r5};
  \Root {} {r5} {r6};
  \ClusterLD c3[{\dd}][{\mm}] = (r5)(r6);
  \ClusterD c4[{2r\!-\!t}] = (r4)(c3);
  \Cluster c5 = (c2)(c4);
\endclusterpicture}

\def\clInImCce{\clusterpicture            
  \Root {1} {first} {r1};
  \Root {} {r1} {r2};
  \Root {1} {r2} {r3};
  \Root {1} {r3} {r4};
  \Root {1} {r4} {r5};
  \Root {} {r5} {r6};
  \ClusterLD c1[{\ee}][{\nn}] = (r5)(r6);
  \ClusterD c2[{2r}] = (r4)(c1);
  \ClusterLD c3[{\dd}][{\mm}] = (r3)(c2);
  \Cluster c4 = (r1)(r2)(c3);
\endclusterpicture}

\def\clInImDce{\clusterpicture            
  \Root {2} {first} {r1};
  \Root {} {r1} {r2};
  \ClusterD c1[{\nn\!-\!t}] = (r1)(r2);
  \Root {1} {c1} {r3};
  \Root {1} {r3} {r4};
  \Root {1} {r4} {r5};
  \Root {} {r5} {r6};
  \ClusterLD c2[{\dd}][{\mm}] = (r5)(r6);
  \ClusterD c3[{2r}] = (r4)(c2);
  \ClusterD c4[t] = (r3)(c3);
  \ClusterL c5[{\ee}] = (c1)(c4);
\endclusterpicture}

\def\clInImEce{\clusterpicture            
  \Root {1} {first} {r1};
  \Root {2} {r1} {r2};
  \Root {} {r2} {r3};
  \ClusterLD c1[{\ee}][{\nn}] = (r2)(r3);
  \Root {1} {c1} {r4};
  \Root {1} {r4} {r5};
  \Root {} {r5} {r6};
  \ClusterLD c2[{\dd}][{\mm}] = (r5)(r6);
  \ClusterD c3[{2r}] = (r4)(c2);
  \Cluster c4 = (c1)(c3);
  \Cluster c5 = (r1)(c4);
\endclusterpicture}

\def\clInImGce{\clusterpicture            
  \Root {1} {first} {r1};
  \Root {1} {r1} {r2};
  \Root {} {r2} {r3};
  \ClusterLD c1[{\ee}][{\nn}] = (r2)(r3);
  \Root {1} {c1} {r4};
  \Root {1} {r4} {r5};
  \Root {} {r5} {r6};
  \ClusterLD c2[{\dd}][{\mm}] = (r5)(r6);
  \ClusterD c3[{2r}] = (r4)(c2);
  \Cluster c4 = (r1)(c1)(c3);
\endclusterpicture}

\def\clInImHce{\clusterpicture            
  \Root {1} {first} {r1};
  \Root {1} {r1} {r2};
  \Root {1} {r2} {r3};
  \Root {1} {r3} {r4};
  \Root {1} {r4} {r5};
  \Root {} {r5} {r6};
  \ClusterLD c1[{\ee}][{\nn}] = (r5)(r6);
  \ClusterD c2[{2r}] = (r4)(c1);
  \ClusterLD c3[{\dd}][{\mm}] = (r3)(c2);
  \Cluster c4 = (r2)(c3);
  \Cluster c5 = (r1)(c4);
\endclusterpicture}

\def\clInInBalce{\clusterpicture            
  \Root {2} {first} {r1};
  \Root {1} {r1} {r2};
  \Root {} {r2} {r3};
  \ClusterL c1[{\eta}] = (r2)(r3);
  \Cluster c2 = (r1)(c1);
  \Root {1} {c2} {r4};
  \Root {1} {r4} {r5};
  \Root {} {r5} {r6};
  \ClusterLD c3[{\ee\eta}][{\nn}] = (r5)(r6);
  \ClusterD c4[r] = (r4)(c3);
  \Cluster c5 = (c2)(c4);
  \frob(c2)(c4);
\endclusterpicture}

\def\cloInBalce{\clusterpicture            
  \Root {2} {first} {r1};
  \Root {} {r1} {r2};
  \Root {} {r2} {r3};
  \ClusterD c1[t] = (r1)(r2)(r3);
  \Root {1} {c1} {r4};
  \Root {1} {r4} {r5};
  \Root {} {r5} {r6};
  \ClusterLD c2[{\ee}][{\nn}] = (r5)(r6);
  \ClusterD c3[{2r\!-\!t}] = (r4)(c2);
  \Cluster c4 = (c1)(c3);
\endclusterpicture}

\def\cloInCce{\clusterpicture            
  \Root {1} {first} {r1};
  \Root {} {r1} {r2};
  \Root {1} {r2} {r3};
  \Root {1} {r3} {r4};
  \Root {} {r4} {r5};
  \Root {} {r5} {r6};
  \ClusterD c1[{2r}] = (r4)(r5)(r6);
  \ClusterLD c2[{\ee}][{\nn}] = (r3)(c1);
  \Cluster c3 = (r1)(r2)(c2);
\endclusterpicture}

\def\cloInDce{\clusterpicture            
  \Root {2} {first} {r1};
  \Root {} {r1} {r2};
  \ClusterD c1[{\nn\!-\!t}] = (r1)(r2);
  \Root {1} {c1} {r3};
  \Root {1} {r3} {r4};
  \Root {} {r4} {r5};
  \Root {} {r5} {r6};
  \ClusterD c2[{2r}] = (r4)(r5)(r6);
  \ClusterD c3[t] = (r3)(c2);
  \ClusterL c4[{\ee}] = (c1)(c3);
\endclusterpicture}

\def\cloInEce{\clusterpicture            
  \Root {1} {first} {r1};
  \Root {1} {r1} {r2};
  \Root {} {r2} {r3};
  \Root {1} {r3} {r4};
  \Root {1} {r4} {r5};
  \Root {} {r5} {r6};
  \ClusterLD c1[{\ee}][{\nn}] = (r5)(r6);
  \ClusterD c2[{2r}] = (r4)(c1);
  \Cluster c3 = (r2)(r3)(c2);
  \Cluster c4 = (r1)(c3);
\endclusterpicture}

\def\cloInGce{\clusterpicture            
  \Root {1} {first} {r1};
  \Root {1} {r1} {r2};
  \Root {} {r2} {r3};
  \ClusterLD c1[{\ee}][{\nn}] = (r2)(r3);
  \Root {1} {c1} {r4};
  \Root {} {r4} {r5};
  \Root {} {r5} {r6};
  \ClusterD c2[{2r}] = (r4)(r5)(r6);
  \Cluster c3 = (r1)(c1)(c2);
\endclusterpicture}

\def\cloInHce{\clusterpicture            
  \Root {1} {first} {r1};
  \Root {2} {r1} {r2};
  \Root {} {r2} {r3};
  \ClusterLD c1[{\ee}][{\nn}] = (r2)(r3);
  \Root {1} {c1} {r4};
  \Root {} {r4} {r5};
  \Root {} {r5} {r6};
  \ClusterD c2[{2r}] = (r4)(r5)(r6);
  \Cluster c3 = (c1)(c2);
  \Cluster c4 = (r1)(c3);
\endclusterpicture}

\def\cloInJce{\clusterpicture            
  \Root {1} {first} {r1};
  \Root {} {r1} {r2};
  \Root {} {r2} {r3};
  \Root {1} {r3} {r4};
  \Root {1} {r4} {r5};
  \Root {} {r5} {r6};
  \ClusterLD c1[{\ee}][{\nn}] = (r5)(r6);
  \ClusterD c2[{2r}] = (r4)(c1);
  \Cluster c3 = (r1)(r2)(r3)(c2);
\endclusterpicture}

\def\cloInKce{\clusterpicture            
  \Root {1} {first} {r1};
  \Root {1} {r1} {r2};
  \Root {1} {r2} {r3};
  \Root {1} {r3} {r4};
  \Root {} {r4} {r5};
  \Root {} {r5} {r6};
  \ClusterD c1[{2r}] = (r4)(r5)(r6);
  \ClusterLD c2[{\ee}][{\nn}] = (r3)(c1);
  \Cluster c3 = (r2)(c2);
  \Cluster c4 = (r1)(c3);
\endclusterpicture}

\def\clooBalc{\clusterpicture            
  \Root {2} {first} {r1};
  \Root {} {r1} {r2};
  \Root {} {r2} {r3};
  \ClusterD c1[t] = (r1)(r2)(r3);
  \Root {1} {c1} {r4};
  \Root {} {r4} {r5};
  \Root {} {r5} {r6};
  \ClusterD c2[{2r\!-\!t}] = (r4)(r5)(r6);
  \Cluster c3 = (c1)(c2);
\endclusterpicture}

\def\clooCc{\clusterpicture            
  \Root {1} {first} {r1};
  \Root {1} {r1} {r2};
  \Root {} {r2} {r3};
  \Root {1} {r3} {r4};
  \Root {} {r4} {r5};
  \Root {} {r5} {r6};
  \ClusterD c1[{2r}] = (r4)(r5)(r6);
  \Cluster c2 = (r2)(r3)(c1);
  \Cluster c3 = (r1)(c2);
\endclusterpicture}

\def\clooDc{\clusterpicture            
  \Root {1} {first} {r1};
  \Root {} {r1} {r2};
  \Root {} {r2} {r3};
  \Root {1} {r3} {r4};
  \Root {} {r4} {r5};
  \Root {} {r5} {r6};
  \ClusterD c1[{2r}] = (r4)(r5)(r6);
  \Cluster c2 = (r1)(r2)(r3)(c1);
\endclusterpicture}

\def\cloofc{\clusterpicture            
  \Root {2} {first} {r1};
  \Root {} {r1} {r2};
  \Root {} {r2} {r3};
  \Cluster c1 = (r1)(r2)(r3);
  \Root {1} {c1} {r4};
  \Root {} {r4} {r5};
  \Root {} {r5} {r6};
  \ClusterD c2[{r}] = (r4)(r5)(r6);
  \Cluster c3 = (c1)(c2);
  \frob(c1)(c2);
\endclusterpicture}

\advance\oddsidemargin -1cm
\advance\evensidemargin -1cm
\advance\textwidth 2cm
\advance\textheight 0.3cm 

\begin{document}

\title{On the parity conjecture for abelian surfaces}
\author{Vladimir Dokchitser, C\'eline Maistret}

\address{Department of Mathematics, University College London, London WC1H 0AY, UK}
\email{v.dokchitser@ucl.ac.uk}

\address{School of Mathematics, University of Bristol, Bristol BS8 1UG, UK}
\email{tim.dokchitser@bristol.ac.uk}
\email{celine.maistret@bristol.ac.uk}

\address{School of Mathematics, University of Glasgow, Glasgow G12 8QQ}
\email{adam.morgan@glasgow.ac.uk}

%

\begin{abstract}
Assuming finiteness of the Tate--Shafarevich group, we prove that the Birch--Swinnerton-Dyer conjecture correctly predicts the parity of the rank of semistable principally polarised abelian surfaces. If the surface in question is the Jacobian of a curve, we require that the curve has good ordinary reduction at 2-adic places.
\end{abstract}

\maketitle


\setcounter{tocdepth}{1}

\vspace{-0.9cm}

\tableofcontents


\vspace{-1.1cm}

\section{Introduction}

\subsection{Main results}

The Birch–-Swinnerton-Dyer conjecture predicts that the Mordell–-Weil rank of an abelian variety $A$ over a number field $K$ is given by the order of vanishing of the $L$-function $L(A/K, s)$ at $s\!=\!1$. Despite being more than half a century old, there has been little theoretical evidence for the conjecture beyond the case of elliptic curves. The aim of the present article is to show that it correctly predicts the parity of the rank of abelian surfaces, at least if one is willing to assume the finiteness of Tate--Shafarevich groups.

Recall that the functional equation for $L(A/K, s)$ says that this function is essentially either symmetric or antisymmetric around the central point $s\!=\!1$, and, consequently, the sign in the functional equation determines the parity of the order of the zero there. Of course, neither the analytic continuation of the $L$-function nor its functional equation have been proved. However, part of the conjectural framework specifies that the sign is given by the global root number $w_{A/K}\in \{\pm 1\}$, an invariant that is defined independently of any conjectures. One thus expects that the root number controls the parity of the rank of $A/K$:

\begin{conjecture}[Parity conjecture]\label{conj:parity}
For every abelian variety $A$ over a number field $K$,
$$
 (-1)^{\rk (A/K)} = w_{A/K}. 
$$
\end{conjecture}

Our main result is the following:

\begin{theorem}[=Theorem \ref{thm:paritymain}]\label{thm:introparitymain}
The parity conjecture holds for principally polarised abelian~surfaces over number fields $A/K$ such that $\sha_{A/K(A[2])}$ has finite 2-, 3- and 5-primary part~that~are 
\begin{itemize}[leftmargin=*]
\item Jacobians of semistable genus 2 curves with good ordinary reduction at primes above 2, or
\item semistable, and not isomorphic to the Jacobian of a genus 2 curve.
\end{itemize}
\end{theorem}
\noindent We note that the hypothesis at primes above $2$ requires the underlying curve, and not merely the Jacobian itself, to have good reduction. By a curve with ``ordinary'' reduction we mean one whose Jacobian has ordinary reduction.

There is a range of results on the parity conjecture in the context of elliptic curves, but the progress for higher dimensional abelian varieties has been rather limited. Previous results only apply to sparse families, e.g. \cite{CFKS} requires the abelian variety to admit a suitable $K$-rational isogeny of degree $p^{\dim A}$, and \cite{AdamParity} addresses Jacobians of hyperelliptic curves that have been base-changed from a subfield of index 2.

The proof of the above theorem has two main ingredients. The first is the following reduction step that applies to abelian varieties of arbitrary dimension. Its proof is based on the method of regulator constants of \cite{tamroot, squarity} and will be explained in Appendix B.

\begin{theorem}\label{intro:appmain1}
Let $F/K$ be a Galois extension of number fields with Galois group $G$. Let $A/K$ be a semistable principally polarised abelian variety such that $\sha_{A/F}$ is finite. If the parity conjecture holds for $A/F^H$ for all 2-groups $H\le G$, then it holds for $A/K$.
\end{theorem}

The second ingredient is a proof of Theorem \ref{thm:introparitymain} under the assumption that the degree of the field extension generated by $A[2]$ is a power of 2. More precisely, we establish the ``2-parity conjecture'' in this case. Without some assumption on the Tate--Shafarevich group the parity conjecture currently appears to be completely out of reach --- indeed, it would give an elementary criterion for predicting the existence of points of infinite order, something that seems to be impossibly difficult already for elliptic curves. However, the version for Selmer groups is more tractable. We will write $\rk_p(A/K)$ for the $p^\infty$-Selmer rank of $A$, that is $\rk_p (A/K)\!=\!\rk(A/K)+\delta_p$, where $\delta_p$ is the multiplicity of $\Q_p/\Z_p$ in the decomposition $\sha_{A/K}[p^\infty]\simeq (\Q_p/\Z_p)^{\delta_p} \times$(finite) and is conjecturally always 0.

\begin{conjecture}[$p$-Parity conjecture]\label{conj:pparity}
For every abelian variety $A$ over a number field $K$ and for a prime number $p$,
$$
 (-1)^{\rk_p (A/K)} = w_{A/K}. 
$$
\end{conjecture}

\begin{theorem}[=Theorem \ref{thm:2paritymain}]\label{thm:introparityC2D4}
The 2-parity conjecture holds for all principally polarised abelian surfaces over number fields $A/K$ such that $\Gal(K(A[2])/K)$ is a 2-group that are 
\begin{itemize}[leftmargin=*]
\item Jacobians of semistable genus 2 curves with good ordinary reduction at primes above 2, or
\item not isomorphic to the Jacobian of a genus 2 curve.
\end{itemize}
\end{theorem}

Assuming the finiteness of Tate--Shafarevich groups, the $p$-parity conjecture clearly implies the parity conjecture. In particular, Theorem \ref{thm:introparitymain} is a direct consequence of Theorems \ref{thm:introparityC2D4} and \ref{intro:appmain1} applied to $F\!=\! K(A[2])$.

The proof of Theorem \ref{thm:introparityC2D4} consists of two parts, outlined in more detail in \S\ref{ss:introC2D4} and \S\ref{ss:intro2parity} below. The first expresses the parity of the $2^{\infty}$-Selmer rank of $A/K$ as a product of some local terms $\lambda_{A/K_v}$, 
$$
 (-1)^{\rk_2 A/K} = \prod_v \lambda_{A/K_v},
$$
analogously to the formula for the global root number as a product of local root numbers $w_{A/K}=\prod w_{A/K_v}$, the products taken over all the places of $K$. This makes crucial use of a Richelot isogeny on $A$ whose existence is guaranteed by the restriction on $\Gal(K(A[2])/K)$.

The second part is the proof that this expression for the parity of the rank is compatible with root numbers. In other words, that $\lambda_{A/K_v}w_{A/K_v}$ satisfies the product formula
$$
  \prod_v \lambda_{A/K_v}w_{A/K_v} = 1,
$$
which leads to the desired expression $(-1)^{\rk_2 A/K} \!=\! w_{A/K}$. This product formula is more delicate than one might expect, because one often has $\lambda_{A/K_v}\!\neq\! w_{A/K_v}$. However, rather miraculously, $\lambda_{A/K_v}\in\{\pm 1\}$ always differs from $w_{A/K_v}\in\{\pm 1\}$ at an even number of places $v$. Conjecture~\ref{conj:local} below gives an explanation for this phenomenon, by describing an explicit relation between the two local invariants. The key point of the conjecture is that it reduces the global problem of controlling the parity of $2$-Selmer ranks to the purely local one of proving an identity between various invariants of genus 2 curves defined over local fields. We prove this conjecture for all semistable curves with good ordinary reduction at primes above $2$ (see Theorem 1.16), which lets us deduce the $2$-parity result of Theorem \ref{thm:introparityC2D4} and hence Theorem \ref{thm:introparitymain} (see Theorem \ref{thm:introlocalglobal}). The proof relies on explicit formulas and the study of genus 2 curves over local fields, and occupies a substantial part of the present paper.

We note that recently Docking has managed to prove an analogue of the parity formula $(-1)^{\rk_2 A/K}=\prod \lambda_{A/K_v}$ for Jacobians of curves of genus 3 with $\Gal(K(A[2])/K)$ a 2-group, see \cite{Docking} Thm.\ 1.5. A proof of the product formula $\prod_v \lambda_{A/K_v} w_{A/K_v}=1$ in his setting, combined with Theorem \ref{intro:appmain1} above, would give an analogue of Theorem \ref{thm:introparitymain} for Jacobians of curves of genus 3.

\subsection{Parity of $2^\infty$-Selmer ranks of Jacobians of C2D4 curves}\label{ss:introC2D4}

The main part of the paper deals with the $2$-parity conjecture for principally polarised abelian surfaces $A$ that have $G=\Gal(K(A[2])/K)$ a 2-group. Generically, these surfaces are Jacobians of genus 2 curves $C:y^2\!=\!f(x)$ for polynomials $f(x)$ of degree 6. Moreover, $G$ is precisely $\Gal(f(x))$, so the 2-group condition can be phrased as $\Gal(f(x))\le C_2\!\times\! D_4$, where $D_4$ denotes the dihedral group of order 8. We will refer to these as C2D4 curves:

\begin{definition}
A {\em C2D4 curve} $C$ over a field $K$ is a genus 2 curve $C:y^2\!=\!cf(x)$ with $c\in K^\times$ and $f(x)$ monic of degree 6, together with an embedding $\Gal(f(x))\subseteq C_2\!\times\! D_4$ as a permutation group on 6 roots (where $C_2$ and $D_4$ act separately on 2 and on 4 roots  in their natural ways).
\end{definition}

We can control the parity of the $2^\infty$-Selmer rank of Jacobians of C2D4 curves using purely local data as follows.
The Jacobian $J=\Jac C/K$ of a C2D4 curve $C$ admits a canonical Richelot isogeny $\phi:\! J \!\to\! \widehat{J}$ to the Jacobian $\widehat{J}$ of another C2D4 curve $\widehat{C}$ (Richelot dual curve), at least if we ignore an exceptional case when a certain invariant $\Delta=\Delta(C)$ vanishes (Definitions~\ref{def:richelotdual},~\ref{def:invariants}).

\begin{definition}\label{de:lambda}
For a C2D4 curve $C$ over a local field $K$ let
$$
  \lambda_{C/K} = {\df_{C/K}}{\df_{\widehat{C}/K}}\cdot(-1)^{\dim_{\F_2}\ker\phi|_K - \dim_{\F_2}\coker\phi|_K}.
$$
Here $\phi$ is the associated Richelot isogeny, $\widehat{C}$ is the Richelot dual curve, 
$\ker\phi|_K=J(K)[\phi]$ and $\coker\phi|_K = \widehat{J}(K)/\phi(J(K))$ for the two Jacobians $J,\widehat{J}$,
and $\mu_{C/K}$ is $\pm1$ depending on whether $C/K$ is deficient ($-1$) or not ($+1$); see Definition \ref{de:deficiency}.
\end{definition}

\begin{theorem}[see Theorem \ref{thm:localtoglobal}]\label{thm:intro2parity}
Let $K$ be a  number field. For every C2D4 curve $C/K$ with $\Delta(C) \neq 0$,
$$
 (-1)^{\rk_2 \Jac C/K} = \prod_{v} \lambda_{C/K_v},
$$
the product taken over all places of $K$.
\end{theorem}

Assuming analytic continuation and the functional equation of $L(J/K,s)$, the decomposition of the global root number as a product of local root numbers,  $w_{J/K}=\prod_v w_{J/K_v}$, shows that the analytic rank of the Jacobian satisfies
$$
 (-1)^{\rk_{\text{an}} J/K} = \prod_{v} w_{J/K_v}.
$$
Thus the above result can be viewed as a $2^\infty$-Selmer group analogue of the root number formula for parity of the analytic rank given by the functional equation.

The key point is that $\lambda_{C/K_v}$ is a purely local invariant. 
It is usually computable for any given curve (see Remark \ref{rmk:computelambda}). 
By Theorem \ref{thm:intro2parity} it can be used to determine the parity of the rank assuming finiteness of $\sha[2^{\infty}]$, without worrying about its compatibility with the Birch--Swinnerton-Dyer conjecture.
Such expressions also have direct consequences for arithmetic, such as:

\begin{example}
For every C2D4 curve $C/\Q$, the $2^\infty$-Selmer rank of $\Jac C/F$ is even over the field $F\!=\!\Q(i,\sqrt{17})$. Indeed, every rational prime splits in $F/\Q$, so each term in the product in Theorem~\ref{thm:intro2parity} appears an even number of times. Here $F/\Q$ can, of course, be replaced by any Galois extension of number fields in which every prime splits into an even number of primes.
\end{example}

\begin{remark}
Theorem \ref{thm:intro2parity} in fact holds for all genus 2 curves whose Jacobians admit a Richelot isogeny and, more generally, for Jacobians of curves that admit an isogeny $\phi$ to another Jacobian that satisfies $\phi\phi^t=[2]$, see Theorem \ref{thm:localtoglobal}.
\end{remark}

\subsection{2-parity conjecture for C2D4 curves}\label{ss:intro2parity}

In view of the parity conjecture, Theorem \ref{thm:intro2parity} and the root number formula $w_{J/K}=\prod_v w_{J/K_v}$ for the Jacobian $J$, it is tempting to hope that $\lambda_{C/K_v}=w_{J/K_v}$. This is false! However, whenever $C$ is a C2D4 curve over a {\em number field} 
one finds that $\lambda_{C/K_v}$ always differs from $w_{J/K_v}$ at an even number of places. Finding a purely local explanation for this phenomenon was the most difficult task of the present project.

\begin{definition}\label{de:C2D4Curve}
Let $C:y^2=cf(x)$ be a C2D4 curve over a field $K$ of characteristic 0.
The embedding $\Gal(f(x))\subset C_2\!\times\! D_4$ gives a factorisation of $f(x)$ into monic quadratics
$$
f(x) = r(x)s(x)t(x), 
$$
where $r(x) \in K[x]$ and $\Gal(\bar{K}/K)$ preserves $\{s(x),t(x)\}$. 
We write the roots as $\alpha_i, \beta_i \in \bar{K}$, with
$$
r(x)=(x-\alpha_1)(x-\beta_1), \qquad
s(x)=(x-\alpha_2)(x-\beta_2), \qquad
t(x)=(x-\alpha_3)(x-\beta_3).
$$
We will refer to this data or, equivalently, to the choice of embedding $\Gal(f(x))\subset C_2\!\times\! D_4$, as a {\em C2D4 structure} on a genus 2 curve $y^2=cf(x)$.
A C2D4 curve is {\em centered} if $\beta_1=-\alpha_1$. Note that any C2D4 curve is isomorphic to a centered one by the substitution $x\mapsto x-\frac{\alpha_1+\beta_1}{2}$.
\end{definition}

\begin{definition}
\label{def:invariants}
To a centered C2D4 curve we assign the following quantities: 
$$
\begin{array}{ccl}
\DG &=& c(-\a_1^2(\a_2\!+\!\b_2\!-\!\a_3\!-\!\b_3)+\a_2\b_2(\a_3\!+\!\b_3)-\a_3\b_3(\a_2\!+\!\b_2)) \\
\xi &=& 2((\a_2\!+\!\a_1)(\b_2\!+\!\a_1)(\a_3\!+\!\a_1)(\b_3\!+\!\a_1)+ (\a_2\!-\!\a_1)(\b_2\!-\!\a_1)(\a_3\!-\!\a_1)(\b_3\!-\!\a_1))
\end{array}
$$
$$
\begin{array}{cclcrcl}
\ell_1 &=& \a_2\!+\!\b_2\!-\!\a_3\!-\!\b_3 &&\phantom{ha}\eta_1&=&(\a_2\!-\!\a_3)(\b_2\!-\!\b_3)+(\b_2\!-\!\a_3)(\a_2\!-\!\b_3) \\
\ell_2 &=& \a_3\!+\!\b_3 && \eta_2&=&(\a_2\!-\!\a_1)(\a_2\!+\!\a_1)+(\b_2\!-\!\a_1)(\b_2\!+\!\a_1)\\
\ell_3 &=& \a_2\!+\!\b_2 && \eta_3&=&(\a_3\!-\!\a_1)(\a_3\!+\!\a_1)+(\b_3\!-\!\a_1)(\b_3\!+\!\a_1)\\
\dgu &=& \a_1^2 && \dlu &=& \frac{1}{\DG^2}(\a_2\!-\!\b_3)(\a_2\!-\!\a_3)(\b_2\!-\!\a_3)(\b_2\!-\!\b_3) \\
\dgd &=& (\a_2\!-\!\b_2)^2 && \dld &=& 4(\a_3\!+\!\a_1)(\a_3\!-\!\a_1)(\b_3\!+\!\a_1)(\b_3\!-\!\a_1) \\
\dgt &=& (\a_3\!-\!\b_3)^2 && \dlt &=& 4(\a_2\!+\!\a_1)(\a_2\!-\!\a_1)(\b_2\!+\!\a_1)(\b_2\!-\!\a_1) \\
\end{array}
$$
If $C$ is not centered and $C'$ is the centered curve corresponding to it by shifting the $x$-coordinate, we define these quantities for $C$ as being those for $C'$, that is $\Delta(C)=\Delta(C')$, etc.

We will also use $\cP=\ell_1\ell_2\ell_3\eta_2\eta_3\xi(\delta_2\!+\!\delta_3)(\delta_2\eta_2\!+\!\delta_3\eta_3)(\hat{\delta}_2\eta_3\!+\!\hat{\delta}_3\eta_2)$, for the purposes of the shorthand expression ``$\cP\neq 0$''; note that $\delta_i$ and $\hat\delta_i$ are always non-zero. Note also that in general the invariants $\DG, \ell_1, \ell_2, \ell_3, \dgd, \dgt, \eta_2, \eta_3, \dld$ and $\dlt$ are not necessarily $K$-rational.
\end{definition}

\begin{definition}\label{de:errorterm}
For a C2D4 curve $C$ over a local field $K$ of characteristic 0 with $\cP, \Delta\neq 0$ we define $E_{C/K}$ as the following product of Hilbert symbols
$$
\begin{array}{ll}
 E_{C/K}=& 
  (\delta_2\!+\!\delta_3,-\ell_1^2 \delta_2 \delta_3)
  (\delta_2\eta_2\!+\!\delta_3\eta_3 ,-\ell_1^2 \eta_2\eta_3 \delta_2 \delta_3)  
   (\hat{\delta}_2\eta_3\!+\!\hat{\delta}_3\eta_2 ,-\ell_1^2 \eta_2\eta_3 \hat\delta_2 \hat\delta_3) \cr
  &(\xi  ,-\delta_1 \hat\delta_2 \hat\delta_3)
  (\eta_2\eta_3 ,- \delta_2 \delta_3 \hat\delta_2 \hat\delta_3)
  (c,\delta_1 \delta_2 \delta_3 \hat\delta_2 \hat\delta_3)\cdot
  {\leftchoice{(\eta_1  ,- \delta_2 \delta_3\Delta^2\hat\delta_1)}{\text{if } \eta_1 \neq 0}{1}{\text{if } \eta_1 = 0}} \cr
  &( \hat\delta_1,-\frac{\ell_1}{\Delta})( \ell_1^2, -\ell_2 \ell_3)  (2,-\ell_1^2)(\hat{\delta}_2\hat{\delta}_3,-2).\cr
\end{array}
$$
We remark that all individual terms $\dgu$, $\dlu$, $\DG^2$, $\frac{\ell_1}{\Delta}$, $\l_1^2$, $\eta_1$, $\ell_2\ell_3$, $\eta_2\eta_3$, $\xi$, $\delta_2\!+\!\delta_3$, $\delta_2\eta_2\!+\!\delta_3\eta_3$, $\hat{\delta}_2\eta_3\!+\!\hat{\delta}_3\eta_2$, $\dgd\dgt$ and $\dld\dlt$ lie in $K$ as they are preserved by $C_2 \times D_4$. 
\end{definition}

\begin{conjecture}\label{conj:local}
For a C2D4 curve $C/K$ over a local field of characteristic 0 with~$\cP, \Delta \neq 0$,
$$
  w_{\Jac C/K} = \lambda_{C/K} \cdot E_{C/K}.
$$
\end{conjecture}

The essential property of this description is that it explains why $w$ and $\lambda$ always differ at an even number of places, since, by the product formula for Hilbert symbols, $\prod_v E_{C/K_v}\!=\!1$.

\begin{theorem}\label{thm:introlocalglobal}
Let $K$ be a number field. 
The 2-parity conjecture holds for all C2D4 curves $C/K$ with $\cP, \Delta \neq 0$ for which Conjecture \ref{conj:local} holds for $C/K_v$ for all places $v$ of $K$.
\end{theorem}
\begin{proof}
Take the product over all places $v$ of the formula in Conjecture \ref{conj:local}. The result follows from the root number formula $w_{\Jac C/K}=\prod_v w_{\Jac C/K_v}$, Theorem \ref{thm:intro2parity} and the product formula. 
\end{proof}

We will prove this local formula in a large number of cases:

\begin{theorem}[=Theorem \ref{thm:mainlocal}]\label{thm:introlocal}
Conjecture \ref{conj:local} holds for all C2D4 curves with \hbox{$\cP, \Delta\neq 0$} over archimedean local fields, all semistable C2D4 curves with $\cP, \Delta\neq 0$ over finite extensions of $\Q_p$ for odd primes $p$, and all C2D4 curves with $\cP, \Delta\neq 0$ and good ordinary reduction over finite extensions of $\Q_2$.
\end{theorem}

Note that Theorem \ref{thm:introparityC2D4} now follows from Theorems \ref{thm:introlocalglobal} and \ref{thm:introlocal}, at least provided $\cP,\Delta\neq 0$. We will also show that the strange formula of Conjecture \ref{conj:local} must hold in the non-semistable case too, in that a counterexample would lead to a counterexample to the parity conjecture:

\begin{theorem}[=Theorem \ref{thm:converse}]\label{intro:converse}
If the 2-parity conjecture is true for all Jacobians of C2D4 curves over number fields, then Conjecture \ref{conj:local} holds for all C2D4 curves with $\cP, \Delta\neq 0$ over local fields of characteristic 0.
\end{theorem}

\begin{remark}
It would be very interesting to have a conceptual interpretation for $E_{C/K}$. The analogous problem appears to be difficult even in the significantly simpler setting of elliptic curves, see \cite{isogroot} Thm.\ 4, \cite{kurast} Thm.\ 5.8. One could probably extend the definition of $E_{C/K}$ to $\cP=0$. However, this seems to require a lengthy case by case analysis which we chose to avoid by picking a good model for $C$ that has $\cP\ne 0$ (Lemma \ref{thm:mobiusPnonzero}). The case $\eta_1=0$ has to be treated separately as this condition is model independent.

Manipulating formal algebraic expressions such as $E_{\hat{C}/K}$ with a computer is not practical: these expressions are enormous and computers cannot simplify Hilbert symbols. However we made extensive use of computational \emph{data} to find the expression for $E_{C/K}$. Once one finds the right list of invariants $I_i$ it is not difficult to produce the product expression of $E_{C/K}$: one compiles a large list of C2D4 curves and for each curve one computes $w_{\Jac C/K}$, $\lambda_{C/K}$ and all possible Hilbert symbols $(I_i,I_j)$. One then uses linear algebra to find an expression for $w_{\Jac C/K}\lambda_{C/K}$ in terms of these Hilbert symbols. The difficulty is then to find this list of invariants in the first place, the main issue being that Hilbert symbols do not behave sensibly under addition. Classical invariants such as Igusa invariants are not sensitive to Richelot isogenies and some of the local data that determine $w_{\Jac C/K}\lambda_{C/K}$. Our invariants carry this information, for example see proof of Theorem \ref{th:localconjectureinfinite}, and \S \ref{se:lambdaw}. 

In principle $w_{\Jac C/K}\lambda_{C/K}$  only depends on the Richelot isogeny; in terms of Definition \ref{de:C2D4Curve} it means that it is symmetric in $r,s$ and $t$. However there appears to be a barrier to finding a Hilbert symbol expression for $w_{\Jac C/K}\lambda_{C/K}$ without breaking this symmetry or the symmetry between $C$ and $\Ch$. 

We have numerically verified Conjecture \ref{conj:local} on all 40441 genus 2 curves currently in the LMFDB whose simplified model is given by a degree 6 polynomial, for all odd primes of tame reduction and for each possible C2D4 structure for which $\cP,\Delta \ne 0$ (excluding the small number of cases when Magma failed to return a regular model for $\Ch$). In theory, one might be able to prove this conjecture over a specific local field by numerically checking a finite list of curves in the vein of Halberstadt's work on root numbers (\cite{halberstadt}). However the length of the list is likely to be unreasonable. 
\end{remark}

\subsection{Overview}\label{ss:overview}

In \S \ref{s:notation} we review background material, including the construction of the Richelot dual curve 
and the theory of clusters of \cite{m2d2}, which will allow us to control local invariants of genus 2 curves over completions $K_v$ for primes $v$ of odd residue characteristic.

In \S \ref{s:isogenyparity} we explain how to control the parity of the $2^\infty$-Selmer rank for Jacobians of curves that admit a suitable isogeny, and prove a general version of Theorem \ref{thm:intro2parity} (see Theorem \ref{thm:localtoglobal}). We also prove a formula for $\lambda_{C/K}$, which converts the kernel-cokernel into Tamagawa numbers and other standard quantities (Theorem \ref{thm:localtoglobal}); for example for curves over finite extension of $\Q_p$ with $p$ odd, it reads $\lambda_{C/K}=\mu_C \mu_{\widehat{C}} (-1)^{\ord_2 c_{J/K}/c_{\widehat{J}/K}}$, where $J$ and $\widehat{J}$ are the two Jacobians.

Sections \S \ref{s:locconj1}--\ref{s:locconj2} focus on C2D4 curves and form the technical heart of the proof of Theorem~\ref{thm:introparityC2D4} on the 2-parity conjecture and Theorem \ref{thm:introlocal} on Conjecture \ref{conj:local}, which compares local root numbers to the $\lambda$-terms. Roughly, the idea is the following.

First of all, we can work out certain cases by making all the terms in Conjecture \ref{conj:local} totally explicit. For example, suppose $C/\Q_p$ is a C2D4 curve for $p\neq 2$, given by $y^2=f(x)$ with $f(x)\in\Z_p[x]$ monic, and that $f(x)\bmod p$ has four simple roots $\bar{\alpha}_2,\bar{\beta}_2,\bar{\alpha}_3,\bar{\beta}_3$, and a double root $\bar{\alpha}_1=\bar{\beta}_1$. 
The reduced curve has a node, and, analogously to multiplicative reduction on an elliptic curve, the Jacobian has local Tamagawa number $v(\alpha_1-\beta_1)^2=v(\delta_1)$ if the node is split, and 1 or 2 (depending on whether $v(\delta_1)$ is odd or even) if the node is non-split. Whether the node is split or non-split turns out to be precisely measured by whether or not $\xi$ is a square in~$\Q_p$.
An explicit computation of the Richelot dual curve shows that, generically (if $v(\Delta)\!=\!0$), its reduction also has a node and its Jacobian's Tamagawa number is $2v(\delta_1)$ or 2 depending again on whether $\xi$ is a square (split node) or not (non-split node).
Neither curve here is deficient, so we obtain $\lambda_{C/K}=-1$ unless $\xi$ is a non-square and $v(\delta_1)$ is even, in which case it is $+1$.
As for multiplicative reduction on elliptic curves, the local root number in this case is $w_{\Jac C/\Q_p}=\pm 1$ depending on whether the node is split ($-1$) or non-split ($+1$). Finally, generically (!) all the terms apart from $\delta_1$ in the expression for $E_{C/\Q_p}$ are units, so that all the Hilbert symbols are (unit,unit)$=1$, except for one remaining term $(\xi,\delta_1)$. The latter is $-1$ precisely when $\xi$ is a non-square (non-split node!) and $v(\delta_1)$ is odd. This magically combines to $w=\lambda\cdot E$, as~required.

We will work out a number of cases by a similar brute force approach (\S \ref{s:oddisogeny}--\ref{s:2adic}); this is often rather more delicate than described above, as we have brushed the non-generic cases (when certain quantities become non-units) under the rug. Unfortunately, there is a myriad of possible reduction types that one would need to address to prove the formula $w=\lambda \cdot E$ in general. Instead, we will use a global-to-local trick to cut down the number of cases to a manageable list (from 938 to 48, in the description used in Theorem \ref{th:localconjectureoddprime}). This is based on the following lemma, which follows directly from Theorem~\ref{thm:intro2parity} and the product formula for Hilbert symbols.

\begin{lemma}\label{lem:lastplace}
Let $K$ be a number field and $C/K$ a C2D4 curve with $\cP, \Delta\neq 0$ for which the 2-parity conjecture holds. If Conjecture 
 \ref{conj:local} holds for $C/K_v$ for all places $v$ of $K$ except possibly one place $w$, then it also holds for $C/K_w$.
\end{lemma}

Thus to prove the formula $w=\lambda\cdot E$ of Conjecture \ref{conj:local} for $C$ over a local field, we can try to deform it to a suitable curve over a number field. The main difficulty, of course, is that we do not a priori have a supply of C2D4 curves over number fields for which we know the 2-parity conjecture!
However, we can bootstrap ourselves by making use of the cases for which we have worked out Conjecture \ref{conj:local} using the brute force approach outlined above, and which give us a supply of C2D4 curves over number fields for which the 2-parity conjecture holds. Observe also that the truth of the 2-parity conjecture for a C2D4 curve $C$ is
\begin{itemize}
\item independent of the choice of model for $C$, and 
\item independent of the choice of the C2D4 structure.
\end{itemize}
This will let us show that Conjecture \ref{conj:local} is also independent of the choice of model and the choice of C2D4 structure for curves over local fields. 

To make this method work we need to understand how various quantities behave under a change of model (\S \ref{s:mobius}), and to have a way to approximate C2D4 curves over local fields by C2D4 curves over number fields that, moreover, behave well at all other places (\S \ref{s:deformation}). In \S \ref{s:locconj2} we justify that these tools are enough to prove Conjecture \ref{conj:local} in all the cases we claim in~Theorem~\ref{thm:introlocal}.

In \S \ref{s:final} we tie these results together, deal with the exceptional cases when $\cP=0$, $\Delta=0$ or the abelian surface is not a Jacobian, and prove Theorems \ref{thm:introparitymain} and \ref{thm:introparityC2D4}. 

Appendix A (by A. Morgan) provides a formula for $\lambda$ for curves with good ordinary reduction over 2-adic fields.
Appendix B (by T. Dokchitser and V. Dokchitser) deals with regulator constants and Theorem~\ref{intro:appmain1}.

\begin{acknowledgements}
We would like to warmly thank Tim Dokchitser and Adam Morgan for contributing the appendices, numerous helpful discussions and for undertaking our joint project~\cite{m2d2}, which was essential for the present article. We are very grateful to the anonymous referee for their exceptionally detailed report and their suggestions for improvements. 
We would like to thank Andrew Sutherland for providing useful data on genus 2 curves over 2-adic fields.
We also thank the University of Warwick, King's College London, Boston University, Max Planck Institute and the University of Bristol, where parts of this work were carried out. The first author was supported by a Royal Society University Research Fellowship. 
This research was partially supported by EPSRC grants EP/M016838/1 and EP/M016846/1 `Arithmetic of hyperelliptic curves'
and the Simons Collaboration grant 550023 `Arithmetic Geometry, Number Theory, and Computation'.
\end{acknowledgements}


\section{Notation and background}\label{s:notation}

\subsection{General notation}
Throughout the paper, $\rk_p(A/K)$ will denote the $p^{\infty}$-Selmer rank of $A/K$ (see Conjecture \ref{conj:pparity}) and $\phi^t$ the dual of a given isogeny $\phi$.
 
For a local field $K$ with residue field $k$, a curve $C/K$ and an abelian variety $A/K$ we write
 
\begin{tabular}{lll}
$\pi_K$ & uniformiser of a local non-archimedean field $K$ \cr
$v(x)$ & valuation of $x\in\bar{K}$, normalised so that $v(\pi_K)=1$ \cr
$|x|_K$ & normalised absolute value of $x$ so that $|\pi_K|_K = \frac{1}{|k|}$\cr
$I_K$ & inertia subgroup of $\Gal(\bar{K}/K)$ \cr
$\Frob_K$ & a Frobenius element in $\Gal(\bar{K}/K)$ \cr
$K^{nr}$ & maximal unramified extension of $K$\cr
\end{tabular}

\begin{tabular}{lll}
$c_{A/K}$ & Tamagawa number for $A/K$\cr
$ \omega_{A/K}^\circ$ & N\'eron exterior form for $A/K$\cr
$\frac{\omega}{\omega'}$& scalar $\kappa\in K$ with $\omega=\kappa\omega'$ for exterior forms $\omega,\omega'$\cr
$w_{A/K}$, $w_{C/K}$ & local root number of $A/K$ and of $\Jac{C}/K$\cr
$n_{A/\R}$, $n_{C/\R}$ & number of components of $A(\R)$ and $C(\R)$\cr
$\mu_{C/K}$ & $1$ if $C$ is not deficient, $-1$ if $C$ is deficient (see Definition \ref{de:deficiency})\cr
$\square$& a non-zero square element in $K$\cr
\end{tabular}

\medskip

We will almost always deal with genus 2 curves $C:y^2=c f(x)$, where $f(x)$ is monic. For convenience of the reader, we  list where some of the definitions associated to $C$ may be found (numbering refers to Definitions and Notations):

\begin{tabular}{lll}
C2D4 curve & \ref{de:C2D4Curve}\cr
 $\lambda_{C/K}$ &  \ref{de:lambda} \cr
 $E_{C/K}$ &  \ref{de:errorterm}\cr
$\mathcal{F}, \mathcal{F}_{C2D4}$ & explicit 2-adic families of genus 2 curves, \ref{def:familyf} \cr
$\widehat{C}$ & Richelot dual curve of $C$, \ref{def:richelotdual}\cr
$r(x)$, $s(x)$, $t(x)$ & factorisation of $f(x)$,  \ref{de:C2D4Curve}\cr
$\alpha_i, \beta_i$ & roots of $f(x)$,  \ref{de:C2D4Curve}\cr
 $\hat{\a}_i, \hat{\b}_i$ & roots of defining polynomial of $\widehat{C}$,  \ref{def:richelotdual}\cr
$\delta_i, \hat{\delta}_i$, $\eta_i$, $\xi$, $\Delta$, $\ell_i$, $\cP$ &  terms entering $E_{C/K}$,  \ref{def:invariants}\cr
\smash{\raise4pt\hbox{\clusterpicture\Root[A]{1}{first}{r1};\Root[A]{}{r1}{r2};\endclusterpicture}} 
\smash{\raise4pt\hbox{\clusterpicture\Root[B]{1}{first}{r1};\Root[B]{}{r1}{r2};\endclusterpicture}}
\smash{\raise4pt\hbox{\clusterpicture\Root[C]{1}{first}{r1};\Root[C]{}{r1}{r2};\endclusterpicture}} 
\smash{\raise4pt\hbox{\clusterpicture\Root[D]{1}{first}{r1};\endclusterpicture}} 
 & 
$\a_1, \b_1, \a_2, \b_2, \a_3, \b_3$ and a general root,  \ref{not:rst}, \ref{not:realroots}, \ref{not:drawcluster} \cr
$C_m, C_t$ & models for $C$ for $m\in\GL_2(K)$ and $t\in K$,  \ref{def:Cm}, \ref{def:Mt} \cr
$M_t$ & matrix and M\"obius map to obtain $C_t$,  \ref{def:Mt}\cr
$C$ close to $C'$ & \ref{def:close}
\end{tabular}

\noindent We will write $\alpha_i(C)$, $\Delta(C)$ etc, if we wish to stress which curve we are referring to.

We write $D_4$ for the dihedral group of order 8 and $V_4$ for the Klein subgroup of $S_4$.

Every point on $\Jac C$ can be represented by the divisor $[P,Q]=P+Q-
\infty^+-\infty^-$, where $P,Q \in C(\bar{K})$ and $\infty^{\pm}$ are the points at infinity of $C$. 2-torsion points are of the form $[(\a_i,0), (\b_i,0)]$.

%
%
%

\subsection{Richelot isogenies}

\begin{definition}[Richelot dual curve, see \cite{Flynn} Ch. 10, \cite{Bruin} \S 4, \cite{Smith} Ch. 8]
\label{def:richelotdual}
For a C2D4 curve $C:y^2=cr(x)s(x)t(x)$ as in Definition \ref{de:C2D4Curve}, with $\ell_1, \ell_2, \ell_3, \Delta\neq 0$, its {\em Richelot dual curve} $\widehat{C}$ is given by 
$$
\widehat C : y^2 = \frac{\ell_1\ell_2\ell_3}{\Delta} \hat{r}(x)\hat{s}(x)\hat{t}(x),
$$
where (writing $r'(x)$ for the derivative of $r(x)$ etc)
$$
\hat r (x)= \frac{t(x)s'(x)-s(x)t'(x)}{\ell_1}, \qquad
 \hat s (x)= \frac{r(x)t'(x)-t(x)r'(x)}{ \ell_2}, \qquad
 \hat t (x)= \frac{r(x)s'(x)-s(x)r'(x)}{\ell_3}.
$$
The curve $\widehat{C}$ is a C2D4 curve; $\hat{r}(x), \hat{s}(x), \hat{t}(x)$ are monic quadratics with discriminants $4\Delta^2 \hat{\delta}_1/\ell_1^2$, $\hat{\delta}_2/{\ell_2^2}$ and ${\hat{\delta}_3}/{\ell_3^2}$.
We denote the roots of $\hat{r}(x)$ (respectively $\hat{s}(x), \hat{t}(x)$) by $\hat{\a}_1, \hat{\b}_1$ (respectively $\hat{\a}_2, \hat{\b}_2$ and~$\hat{\a}_3, \hat{\b}_3$).

There is an isogeny $\phi: \Jac C  \to \Jac \widehat{C}$ of degree 4 whose kernel is totally isotropic with respect to the Weil pairing and consists of $0$ and the 2-torsion points $[(\a_i,0), (\b_i,0)]$. The isogeny satisfies $\phi \phi^{t} =[2]$. 
\end{definition}

\begin{remark}
When $\ell_1, \ell_2$ or $\ell_3=0$, one can define the Richelot dual curve by the same construction by cancelling the offending terms in the equation for $\widehat{C}$ and the expressions for~$\hat{r}, \hat{s}, \hat{t}$. 
\end{remark}

\subsection{Deficiency}

\begin{definition}[Deficiency, \cite{PS} Corollary 12]\label{de:deficiency}
A curve $C$ of genus $g$ over a local field $K$ is {\em deficient} if it has no $K$-rational divisor of degree $g\!-\!1$. 
For a genus 2 curve $C/K$, being deficient is equivalent to $C$ not having any $L$-rational points over all extensions $L/F$ of odd degree. 
\end{definition}

\subsection{Pictorial representation of roots}\label{ss:pictorial}

\begin{notation}\label{not:rst}
For a C2D4 curve $C:y^2\!=\!cr(x)s(x)t(x)$, we pictorially represent the roots $\alpha_1, \beta_1$ of $r(x)$ as {\bf r}uby circles
($\smash{\raise4pt\hbox{\clusterpicture\Root[A]{1}{first}{r1};\endclusterpicture}}$), roots $\alpha_2, \beta_2$ of $s(x)$ as {\bf{s}}apphire hexagons ($\smash{\raise4pt\hbox{\clusterpicture\Root[B]{1}{first}{r1};\endclusterpicture}}$), and roots $\alpha_3, \beta_3$ of $t(x)$ as {\bf{t}}urquoise diamonds ($\smash{\raise4pt\hbox{\clusterpicture\Root[C]{1}{first}{r1};\endclusterpicture}}$). We will sometimes refer to them as ruby, sapphire and turquoise roots, respectively.

Note that the Galois group $\le C_2\!\times\! D_4$ preserves the set of ruby roots and either preserves the set of sapphire roots and the set of turquoise roots, or swaps these two sets around. 
\end{notation}

\begin{notation}\label{not:realroots}
For a C2D4 curve $C/\R$, it will turn out that most of the local data that we are interested in is encoded in the arrangement of the real roots of the defining polynomial on the real line. We will depict this information by drawing the real roots in the order that they appear in $\R$ and connect two roots $r, r'$ if the points $(r,0)$ and $(r',0)$ are on the same connected component of $C(\R)$. Thus, for example, a curve with $\alpha_1<\beta_1<\alpha_2<\beta_2<\alpha_3<\beta_3$ and $c<0$ will be depicted by
$\smash{\raise3pt\hbox{\CPRNF[A][A][B][B][C][C]}}$.
\end{notation}

\subsection{Clusters: curves over local fields with odd residue characteristic}\label{ss:clusterpics}

To keep track of the arithmetic of genus 2 curves over $p$-adic fields with $p$ odd we will use the machinery of ``clusters'' of \cite{m2d2}:

\begin{definition}[Clusters]
\label{def:cluster}
Let $K$ be a finite extension of $\Q_p$ 
and $C\!:\!y^2\!=\!cf(x)$ a genus 2 curve, where $f(x)\in K[x]$ is monic of degree 6 with set of roots $\cR$. A {\em cluster} is a non-empty subset $\c\subset\cR$ of the form $\c = D \cap \cR$ for some disc $D=\{x\!\in\! \Kbar \mid v(x-z)\!\geq\! d\}$ for some $z\in \Kbar$~and~$d\in \Q$. 

For a cluster $\s$ of size $>1$, its {\em depth} $d_\s$ is the maximal $d$ for which $\s$ is cut out by such a disc, that is $d_\s\! =\! \min_{r,r' \in \mathfrak{s}} v(r\!-\!r')$. If moreover $\s\neq \cR$, then its {\em relative depth} is $\delta_\s\! =\! d_\s\! -\!d_{P(\s)}$, where $P(\s)$ is the smallest cluster with $\s\subsetneq P(s)$ (the ``parent'' cluster).

We refer to this data as the {\em cluster picture} of $C$.

For C2D4 curves we will often specify which roots are ruby, sapphire and turquoise: we will refer to this data as the {\em colouring} of the cluster picture.
\end{definition}

\begin{notation}\label{not:drawcluster}
We draw cluster pictures by drawing roots $r \in \cR$ as \smash{\raise4pt\hbox{\clusterpicture\Root[D]{1}{first}{r1};\endclusterpicture}}, or as in Notation \ref{not:rst} if we wish to specify which root is which,
and draw ovals around roots to represent clusters (of size $>1$), such as:
$$
\scalebox{1.4}{\clusterpicture            
  \Root[D] {1} {first} {r1};
  \Root[D] {} {r1} {r2};
  \Root[D] {} {r2} {r3};
  \Root[D] {3} {r3} {r4};
  \Root[D] {1} {r4} {r5};
  \Root[D] {} {r5} {r6};
  \ClusterD c1[{2}] = (r1)(r2)(r3);
  \ClusterD c2[{2}] = (r5)(r6);
  \ClusterD c3[{1}] = (r4)(c2);
  \ClusterD c4[{0}] = (c1)(c2)(c3);
\endclusterpicture}
\qquad
\text{or}
\qquad
\scalebox{1.4}{\clusterpicture            
  \Root[A] {1} {first} {r1};
  \Root[B] {} {r1} {r2};
  \Root[C] {} {r2} {r3};
  \Root[C] {3} {r3} {r4};
  \Root[B] {1} {r4} {r5};
  \Root[A] {} {r5} {r6};
  \ClusterD c1[{2}] = (r1)(r2)(r3);
  \ClusterD c2[{2}] = (r5)(r6);
  \ClusterD c3[{1}] = (r4)(c2);
  \ClusterD c4[{0}] = (c1)(c2)(c3);
\endclusterpicture}
$$
The subscript on the largest cluster $\cR$ is its depth; on the other clusters it is their relative depth.
\end{notation}

\begin{definition}[Twins]\label{def:tangenttwin}
A {\em twin} is a cluster of size 2. 
\end{definition}

\begin{definition}[Balanced]
A cluster picture of a genus 2 curve is {\em balanced} if $|\cR|=6$, $d_\cR=0$, there are no clusters of size 4 or 5, and there are either no clusters of size 3 or there are two of them, in which case they have equal depth. A curve $C$ is \emph{balanced} if its cluster picture is. 
\end{definition}

\begin{theorem}\label{th:existbalance}
Let $K$ be a finite extension of $\Q_p$ for an odd prime $p$ with residue field $k$ of size $|k|>5$. Then every semistable $C/K$ of genus 2 admits a model whose cluster picture is balanced. 
\end{theorem}

\begin{proof}
\cite{m2d2} Cor. 15.3.
\end{proof}
\begin{lemma}\label{lem:integralroots}
Let $K$ be a finite extension of $\Q_p$ for an odd prime $p$. If $C/K$ is a balanced centered C2D4 curve, then $v(\alpha_i), v(\beta_i)\ge 0$ for $i=1, 2, 3$. Moreover, if $v(\ell_i)=0$ for some $i$, then $v(\A_i), v(\B_i)\ge 0$.
\end{lemma}
\begin{proof}
Since the curve is balanced and centered, $\alpha_1=-\beta_1$ and the depth of the top cluster is~0. Thus $v(\alpha_1)= v(\frac{1}{2}(\alpha_1-\beta_1)) \ge 0$. The first claim follows as $v(\alpha_1-r)\ge 0$ for each root~$r$. The second claim follows directly from Definition \ref{def:richelotdual}, as $\A_i, \B_i$ are roots of a monic quadratic polynomial with integral coefficients.
\end{proof}

Roughly speaking, the proof of the formula $\lambda=wE$ of Conjecture \ref{conj:local} will require a separate computation for each balanced cluster picture.

\subsection{Local invariants of semistable curves of genus 2}

Let $C/K$ be a curve of genus 2 over a finite extension of $\Q_p$ for an odd prime $p$. We record some results of \cite{m2d2} that will let us control the arithmetic invariants of $C/K$ in terms of its cluster picture.

\begin{theorem}[Semistability criterion, \cite{m2d2} Thm.\  1.8]\label{th:ss}
Let $C/K$ be a curve of genus~2 over a finite extension of $\Q_p$ for an odd prime $p$, given by $C:y^2=cf(x)$ for some $c\in K^\times$ and monic $f(x)\in K[x]$ of degree 6. 
Then $C/K$ is semistable if and only if the following conditions hold:
\begin{enumerate}
\item
The extension $K(\cR)/K$ has ramification degree at most 2,
\item
Every cluster $\s$ with $|\s|\neq 1$ is $I_K$-invariant, 
\item
Every principal cluster $\s$ has $d_\c\! \in\!\Z$ and 
$v(c)\!+\! |\s|d_\s\!+\!\sum_{r\notin\c} v(r-r_\s)\!\in\! 2\Z$,
for any (equivalently every) root $r_\s\in\s$. 
Here a cluster $\s$ is {\em principal} if $|\s|\ge 3$, $\s$ does not properly contain a cluster of size $4$, and $\s$ is not a disjoint union of two clusters of size 3.
\end{enumerate}
\end{theorem}

We will need to keep track of the analogue of the split/non-split dichotomy for elliptic curves with multiplicative reduction. This is done by keeping track of the Galois action on clusters and associating signs $\pm$ to certain clusters of even size, see  \cite{m2d2} Definition 1.13. We will only need their explicit expressions for balanced pictures:

\begin{definition}[Sign and $\theta_\s$]\label{no:signandfrob}
Suppose that $C$ is semistable and balanced.
\begin{enumerate}[leftmargin=*]
\item
If $\cR$ is a union of three twins, the sign of $\cR$ is $+$ if $c\in K^{\times 2}$ and $-$ if $c\not\in K^{\times 2}$;
\item
Otherwise,
for each twin $\s = \{r_1, r_2\}$ pick a square root $\theta_\c = \sqrt{c \prod_{r \notin \c}(\frac{r_1+r_2}{2}-r)}$ and define its sign $\pm$ through the formula
$\frac{\Frob_K(\theta_\c)}{\theta_{\Frob_K(\c)}} \equiv \pm 1$ in the residue field.
\end{enumerate}
Note that the signs depend on the choices of square roots in (2).
\end{definition}

\begin{remark}\label{theta}
When $C$ is semistable and $\s$ is a twin as in Definition \ref{no:signandfrob}(2), $v(\tanc[\s])$ is always even, equivalently $K(\theta_\s)/K$ is unramified (by Theorem \ref{th:ss}(3) for $P(\s)$). 
\end{remark}

\begin{notation}
If we wish to keep track of the signs of clusters in our pictures of Notation~\ref{not:drawcluster}, these will be written as superscripts to the ovals.
If we wish to keep track of the Frobenius action, then lines joining clusters (of size $>1$) will indicate that Frobenius permutes them. We refer to cluster pictures with this extra data as \emph{cluster picture with Frobenius action}. 
\end{notation}

\begin{example}
Let $f(x) = x(x\!-\!1)(x\!-\!i)(x\!-\!i\!+\!3)(x\!+\!i)(x\!+\!i\!-\!3)$ over $\Q_3$, where $i$ is a square root of $-1$. Then there are two clusters of size 2, $\c=\{i, i\!+\!3\}$ and $\c'=\{-i, -i\!+\!3\}$, which are clearly swapped by Frobenius. Here $\theta_\c = \frac{5}{4}\sqrt{1-8i}$ and $\theta_{\c'} = \frac{5}{4}\sqrt{1+8i}$. Since $\frac{25}{16}(1-8i)$ is not a square in the quadratic unramified extension of $\Q_3$, the two signs are $+$ and $-$ in some order. The cluster picture with Frobenius action is \CPTNPMO[D][D][D][D][D][D].
\end{example}

\begin{theorem}[\cite{m2d2} \S 18, Thms. 18.7, 18.8]
\label{th:genus2bible}
Let $K/\Q_p$ be a finite extension for an odd prime~$p$, $C/K$ a semistable genus 2 curve and $J$ its Jacobian. Then the cluster picture with Frobenius action of $C$ is one of the ones given in the table below for some integers $n,m,k,r\!>\!0$. The Tamagawa number $c_{J/K}$, deficiency $\mu_{C/K}$ and local root number $w_{C/K}$ are as given in the table. Isomorphic curves have the same type.
\def\cdepthscale{0.6}  
\def\ee{\epsilon}
\def\nn{\frac{n}{2}}
\def\mm{\frac{m}{2}}
\def\kk{\frac{k}{2}}
\def\dd{\delta}

\def\tgrNA{\hetype{2}}
\def\tgrNB{\hetype{1_n^+}}
\def\tgrNC{\hetype{1_n^-}}
\def\tgrNBe{\hetype{1_n^{\epsilon}}}
\def\tgrND{\hetype{I_{n,m}^{+,+}}}
\def\tgrNE{\hetype{I_{n,m}^{+,-}}}
\def\tgrNF{\hetype{I_{n,m}^{-,-}}}
\def\tgrNDed{\hetype{I_{n,m}^{\epsilon,\delta}}}
\def\tgrNG{\hetype{I_{n\FrobL n}^+}}
\def\tgrNH{\hetype{I_{n\FrobL n}^-}}
\def\tgrNGe{\hetype{I_{n\FrobL n}^{\epsilon}}}
\def\tgrNI{\hetype{U_{n,m,k}^+}}
\def\tgrNJ{\hetype{U_{n,m,k}^-}}
\def\tgrNIe{\hetype{U_{n,m,k}^{\epsilon}}}
\def\tgrNK{\hetype{U_{n\FrobL n,k}^+}}
\def\tgrNL{\hetype{U_{n\FrobL n,k}^-}}
\def\tgrNKe{\hetype{U_{n\FrobL n,k}^{\epsilon}}}
\def\tgrNM{\hetype{U_{n\FrobL n\FrobL n}^+}}
\def\tgrNN{\hetype{U_{n\FrobL n\FrobL n}^-}}
\def\tgrNMe{\hetype{U_{n\FrobL n\FrobL n}^{\epsilon}}}
\def\tgrNO{\hetype{I_n^+x_{\protect\scalebox{0.6}{$\!\!\scriptstyle r\>$}}I_m^+}}
\def\tgrNP{\hetype{I_n^+x_{\protect\scalebox{0.6}{$\!\!\scriptstyle r\>$}}I_m^-}}
\def\tgrNQ{\hetype{I_n^-x_{\protect\scalebox{0.6}{$\!\!\scriptstyle r\>$}}I_m^-}}
\def\tgrNOed{\hetype{I_n^{\epsilon}x_{\protect\scalebox{0.6}{$\!\!\scriptstyle r\>$}}I_m^{\delta}}}
\def\tgrNR{\hetype{I_n^+\FrobX_{\protect\scalebox{0.6}{$\!\scriptstyle r\>$}} I_n}}
\def\tgrNS{\hetype{I_n^-\FrobX_{\protect\scalebox{0.6}{$\!\scriptstyle r\>$}} I_n}}
\def\tgrNRe{\hetype{I_n^{\epsilon}\FrobX_{\protect\scalebox{0.6}{$\!\scriptstyle r\>$}} I_n}}
\def\tgrNT{\hetype{1x_{\protect\scalebox{0.6}{$\!\!\scriptstyle r\>$}}I_n^+}}
\def\tgrNU{\hetype{1x_{\protect\scalebox{0.6}{$\!\!\scriptstyle r\>$}}I_n^-}}
\def\tgrNTe{\hetype{1x_{\protect\scalebox{0.6}{$\!\!\scriptstyle r\>$}}I_n^{\epsilon}}}
\def\tgrNV{\hetype{1x_{\protect\scalebox{0.6}{$\!\!\scriptstyle r\>$}}1}}
\def\tgrNW{\hetype{1\FrobX_{\protect\scalebox{0.6}{$\!\scriptstyle r\>$}} 1}}

\def\dsh{\hbox{--}}

\def\cltopskip{5pt}              
\def\clbottomskip{5pt}           

\begin{table}[H]
$$
\hskip -1cm
\scalebox{0.8}{$
\begin{array}{|c|l|c|c|c|c|}
\hline
\text{Type} & &
\begin{array}{ll}
\ee&\dd
\end{array}
&c_{J/K}&\mu_{C/K}&w_{C/K}\\
\hline
 \tgrNA&
\begin{array}{l@{\>}l@{\>}l@{\>}l@{\>}l@{\>}l}
  \clgBalc & \clgCc &&&
    \end{array}
 &
&1&1&1 \\
\hline
\tgrNV & 
\begin{array}{l@{\>}l@{\>}l@{\>}l@{\>}l@{\>}l}
 \clooDc&\clooBalc &\clooCc &&
    \end{array}
 &
&1&1&1 \\
\hline
\tgrNW &
\begin{array}{l@{\>}l@{\>}l@{\>}l@{\>}l@{\>}l}
 \cloofc& & & &
    \end{array}
&
&1&(-1)^r&1  \\
\hline
  \tgrNBe &  
\begin{array}{l@{\>}l@{\>}l@{\>}l@{\>}l@{\>}l}
 \clnBalce &  \clnCce& \clnDce&\clnEce&\clnGce\\
    \end{array}
&
\begin{array}{ll}
+ & \phantom{+}\\
-  &\\
\end{array}
&
\begin{array}{c}
n\\
\tilde n\\
\end{array}
&
\begin{array}{c}
1\\
1\\
\end{array}
&
\begin{array}{c}
-1\\
1\\
\end{array}
  \\
\hline
 \tgrNTe &  
\begin{array}{l@{\>}l@{\>}l@{\>}l@{\>}l@{\>}l}
\cloInBalce  &\cloInCce&\cloInDce& \cloInEce &\\
\cloInGce&\cloInJce &\cloInKce &  \cloInHce &\\
    \end{array}
&
\begin{array}{ll}
+ &   \phantom{+}\\
-  & \\
\end{array}
&
\begin{array}{c}
n\\
\tilde n\\
\end{array}
&
\begin{array}{c}
1\\
1\\
\end{array}
&
\begin{array}{c}
-1\\
1\\
\end{array}
  \\
\hline
 \tgrNDed &  
\begin{array}{l@{\>}l@{\>}l@{\>}l@{\>}l@{\>}l}
\clnmBalce & \clnmCce &\clnmDce &\clnmEce&\clnmGce
    \end{array}
&
\begin{array}{ll}
+ & +\\
+ & -\\
- & -\\
\end{array}
&
\begin{array}{c}
nm\\
n\tilde m\\
\tilde n \tilde m \\
\end{array}
&
\begin{array}{c}
1\\
1\\
1\\
\end{array}
&
\begin{array}{c}
1\\
-1\\
1\\
\end{array}
  \\
\hline
 \tgrNGe &  
\begin{array}{l@{\>}l@{\>}l@{\>}l@{\>}l@{\>}l}
\clnnBalce &  \clnnCce &&&\\
    \end{array}
&
\begin{array}{ll}
+ &  \phantom{+} \\
- & \\
\end{array}
&
\begin{array}{c}
n\\
\tilde n\\
\end{array}
&
\begin{array}{c}
1\\
1\\
\end{array}
&
\begin{array}{c}
-1\\
1\\
\end{array}
  \\
\hline
 \tgrNIe &  
\begin{array}{l@{\>}l@{\>}l@{\>}l@{\>}l@{\>}l}
\clUnmkBalce & \clUnmkCce &\clUnmkDce & \clUnmkEce  \\
    \end{array}
&
\begin{array}{ll}
+ &  \phantom{+} \\
- & \\
\end{array}
&
\begin{array}{c}
N\\
\widetilde{\frac{N}{M}}\cdot\widetilde M\\
\end{array}
&
\begin{array}{c}
1\\
(-1)^{nmk}\\
\end{array}
&
\begin{array}{c}
1
\end{array}
  \\
\hline
 \tgrNKe
&  
\begin{array}{l@{\>}l@{\>}l@{\>}l@{\>}l@{\>}l}
\clUnnkBalce & \clUnnkCce &\clUnnkDce & \clUnnkEce & \\   
 \end{array}
&
\begin{array}{ll}
+ &  \phantom{+} \\
- & \\
\end{array}
&
\begin{array}{c}
n+2k\\
n\\
\end{array}
&
\begin{array}{c}
1\\
(-1)^k\\
\end{array}
&
\begin{array}{c}
-1\\
-1\\
\end{array}
  \\
\hline
 \tgrNMe
&  
\begin{array}{l@{\>}l@{\>}l@{\>}l@{\>}l@{\>}l}
\clUnnnBalce & & & &   \\
 \end{array}
&
\begin{array}{ll}
+ &  \phantom{+}\\
- & \\
\end{array}
&
\begin{array}{c}
3\\
1\\
\end{array}
&
\begin{array}{c}
1\\
(-1)^n\\
\end{array}
&
\begin{array}{c}
1\\
1\\
\end{array}
  \\
\hline
 \tgrNOed
&  
\begin{array}{l@{\>}l@{\>}l@{\>}l@{\>}l@{\>}l}
 \clInImBalce &\clInImCce &\clInImDce  &\clInImEce&\\
 \clInImGce&\clInImHce&&&\\
 \end{array}
&
\begin{array}{ll}
+ & +\\
+ & -\\
- & -\\
\end{array}
&
\begin{array}{c}
nm\\
n\tilde m\\
\tilde n \tilde m\\
\end{array}
&
\begin{array}{c}
1\\
1\\
1\\
\end{array}
&
\begin{array}{c}
1\\
-1\\
1\\
\end{array}
  \\
\hline
 \tgrNRe
&  
\begin{array}{l@{\>}l@{\>}l@{\>}l@{\>}l@{\>}l}
  \clInInBalce&&&&\\
 \end{array}
&
\begin{array}{ll}
+ & \phantom{+}\\
- & \\
\end{array}
&
\begin{array}{c}
n\\
\tilde n\\
\end{array}
&
\begin{array}{c}
(-1)^r\\
(-1)^r\\
\end{array}
&
\begin{array}{c}
-1\\
1\\
\end{array}
  \\
\hline
 \end{array}$}
$$
\hbox{\footnotesize{Notation: Clusters of size 5 and 6 have arbitrary integer depths, $\eta \in \{\pm1\}$ and $t \in \Z$ are arbitrary.}}
\hbox{\footnotesize{$N=nm+nk+km$, $M=$ gcd$(n,m,k)$, $\tilde x=2$ if $2|x$ and $\tilde x=1$ if $2\nmid x$. }}

\label{tb:genus2bible}
 \end{table}

 \def\cltopskip{1pt}              
\def\clbottomskip{1pt}           

\end{theorem}

We will mostly use this table for balanced curves, i.e. the first column of the cluster pictures. Note that type $\I_{n,m}^{\epsilon, \delta}$ is the same as $\I_{m,n}^{\delta,\epsilon}$. Similarly $\I_{n}^{\epsilon} \times_r \I_{m}^{\delta}$ is the same as $ \I_{m}^{\delta}\times_r \I_{n}^{\epsilon}$, and $\UU^{\epsilon}_{n,m,k}$ is unchanged by any permutation of the indices.

\begin{remark}\label{rmk:Inmcomponentgroup}
We will use a little more information about the types $\I_{n,m}$ and $\I_n\times \I_m$.
Suppose $C$ has type $\I_{n,m}^{\epsilon, \delta}$ or $\I_n^\epsilon\times \I_m^\delta$.
Write $C_k^+$ and $C_k^-$ for the cyclic group $C_k$ on which Frobenius acts trivially and by multiplication by $-1$, respectively.
By \cite{m2d2} Thm.\ 1.15 and Lemma 2.22, the N\'eron component group of $\Jac C/K^{nr}$ is $\Phi_C= C_n^\epsilon\times C_m^\delta$.
Note also that if $C$ has type  $I_{n,m}^{+, -}$ and  $\Phi_C= C_{k}^{+}\times C_{l}^{-}$, for some even $k$ and $l$, then necessarily $n=k$ and $m=l$.
Indeed, if $C_{k}^+\times C_{l}^- \simeq C_{k'}^+\times C_{l'}^-$ then $k=k'$ and $l=l'$, since the groups have $2k=2k'$ Frobenius-invariant elements, and $2l=2l'$ elements on which Frobenius acts by $-1$.
\end{remark}

\section{Parity of $2^\infty$-Selmer rank of Jacobians with a $2^g$-isogeny}\label{s:isogenyparity}

In this section we discuss how to control the parity of the $2^\infty$-Selmer ranks for Jacobians of curves that admit a suitable isogeny.

\begin{definition}
Let $C$ and $C'$ be curves over a local field $K$ whose Jacobians admit an isogeny $\phi: \Jac C \to \Jac C'$ with $\phi\phi^t = [2]$ (equivalently, an isogeny whose kernel is a maximal isotropic subspace of $\Jac C[2]$ with respect to the Weil pairing). We write
$$
\lambda_{C/K,\phi} = \mu_{C/K}\mu_{C'/K}\cdot (-1)^{\dim_{\F_2}\ker\phi|_K - \dim_{\F_2}\coker\phi|_K},
$$
where $\ker\phi|_K = \Jac C(K)[\phi]$ and $\coker\phi|_K = \Jac C'(K)/\phi(\Jac C(K))$.

For a C2D4 curve $C/K$ this is $\lambda_{C/K}$ of Definition \ref{de:lambda}.
\end{definition}

\subsection{Parity theorem}

\begin{theorem}\label{thm:localtoglobal}
Let $C$ and $C'$ be curves over a number field $K$ whose Jacobians admit an isogeny $\phi: \Jac C\to \Jac C'$ 
with $\phi\phi^t=[2]$. Then
$$
  (-1)^{\rk_2 \Jac C/K}= \prod_v \lambda_{C/K_v,\phi},
$$
the product taken over all places of $K$.
\end{theorem}
\vspace{-1mm}
\begin{proof}
Write $A=\Jac C$ and $A'=\Jac C'$.
As in the proof of Thm.\ 4.3 in \cite{squarity}
$$
 2^{\rk_2 A/K} = \square \cdot \frac{|\sha^{\text{nd}}_{A/K}[2^\infty]|}{|\sha^{\text{nd}}_{A'/K}[2^\infty]|} \cdot \prod_v \frac{|\coker\phi|_{K_v}|}{|\ker\phi|_{K_v}|},
$$
where $\sha^{\text{nd}}_{A/K}$ denotes $\sha_{A/K}$ modulo its divisible part.
By a result of Poonen and Stoll (\cite{PS} Thm.\ 8, Cor.\ 12), the order of $|\sha^{\text{nd}}_{A/K}|$ is a square if and only if $C$ is deficient at an even number of places, and is twice a square otherwise (and similarly for $A'$). Hence
\vspace{-2mm}
$$ 
 2^{\rk_2 A/K} = \square\cdot \prod_v \frac{\mu_v}{\mu'_v}\cdot \frac{|\coker\phi|_{K_v}|}{|\ker\phi|_{K_v}|},
$$
where $\mu_v\!=\!2$ if $C/K_v$ is deficient and $\mu_v\!=\!1$ otherwise; and similarly for $\mu_v'$. 
By definition of $\lambda$, the 2-adic valuation of the term at $v$ is even if and only if $\lambda_{C/K_v,\phi}=1$. The result follows.
\end{proof}

\subsection{Kernel/Cokernel on local points}

\begin{notation}
For a curve $C/\R$ with Jacobian $A$, we write $n_{C/\R}$ for the number of connected components of $C(\R)$. We write $A(\R)^\circ$ for the connected component of the identity of $A$ and $n_{A/\R}=|A(\R)/A(\R)^\circ|$ for the number of connected components. 
\end{notation}

Recall from \S\ref{s:notation} that when $K/\Q_p$ is a finite extension, $c_{A/K}$ and $ \omega_{A/K}^\circ$ denote the Tamagawa number and N\'eron exterior form for $A/K$.

\begin{lemma}\label{lem:kercoker}
Let $C$ and $C'$ be curves of genus $g$ over a local field $K$ of characteristic 0, whose Jacobians $A$ and $A'$ admit an isogeny $\phi: A\to A'$ with $\phi\phi^t=[2]$. Then
$$
\frac{|\ker\phi|_K|}{|\coker\phi|_K|} = \left\{
\begin{array}{cl}
  2^g & \text{if }K\simeq\C,\cr
 \bigl|A(K)^\circ[\phi]\bigr|\cdot  {n_{A/K}}/{n_{A'/K}}& \text{if }K\simeq\R, \cr
 c_{A/K}/c_{A'/K} & \text{if }K/\Q_p \text{ finite, $p$ odd},\cr
 {c_{A/K}}/{c_{A'/K}}\cdot \Bigl|\frac{\phi^*\omega^\circ_{A'/K}}{\omega^\circ_{A/K}}\Bigr|_K & \text{if } K/\Q_2 \text{ finite}.
\end{array}
\right.
$$
\end{lemma}
\begin{proof}
The result if clear for $K\simeq \C$.
For $K\simeq \R$, consider the commutative diagram
$$
\begin{CD}
0@>>>A(K)^\circ @>>> A(K) @>>> A(K)/A(K)^\circ @>>>0\\
@.@VV\phi V@VV\phi V@VV\phi V@.\\
0@>>>A'(K)^\circ @>>> A'(K) @>>> A'(K)/A'(K)^\circ @>>>0\\
\end{CD}
$$
The kernels and cokernels of the vertical maps are finite, so, by the snake lemma,
$$
 \frac{|\ker \phi|_K|}{|\coker \phi|_K|} = \frac{|\ker \phi|_{A(K)^\circ}|}{|A'(K)^\circ/\phi(A(K)^\circ)|}\cdot \frac{|\ker \phi|_{A(K)/A(K)^\circ}|}{|(A'(K)/A'(K)^\circ)/\phi(A(K)/A(K)^\circ)|}.
$$
The map on the connected component of the identity is surjective (as $K\simeq\R$), and the groups of connected components are both finite, so this simplifies to the expression claimed.

The case of non-archimedean $K$ is similar, with $A(K)^\circ$ replaced by $A_1(K)$, the kernel of the reduction on the N\'eron model of $A$, see e.g. Lemma 3.8 in \cite{Schaefer}.
\end{proof}

\begin{lemma}[\cite{Gross} Prop. 3.2.2 and 3.3]
\label{pr:n}
For a smooth projective curve $C$ over $\R$
$$
 n_{\Jac C/\R} = \left\{
   \begin{array}{ll}
       2^{n_{C/\R}-1} & {\text{if }} n_{C/\R} > 0,\cr
       1 & {\text{if }} n_{C/\R} = 0 {\text{ and }} C {\text{ has even genus}},\cr
       2 & {\text{if }} n_{C/\R} = 0 {\text{ and }} C {\text{ has odd genus}}.
   \end{array}
\right.
$$
\end{lemma}

\begin{remark}
\label{rmk:computelambda}
For Jacobians of genus 2 curves over $\Q$, one can compute Tamagawa numbers using van Bommel's algorithm (\cite{RvB} \S 4.4) and $\frac{|\ker\phi|_\RR|}{|\coker\phi|_\RR|}$ is obtained using Lemmata \ref{lem:kercoker} and \ref{pr:n}. 
 Let $\omega^{\circ}$ and $\hat{\omega}^{\circ}$ be N\'eron exterior forms at $p\!=\!2$ for $J\!=\!\Jac C$ and $\widehat{J}\!=\!\Jac \widehat{C}$.  The absolute value $ |\frac{\phi^*\hat{\omega}^{\circ}}{\omega^{\circ}}|_2 $ can be computed as follows. 
Let $\omega = \frac{dx}{y}\wedge \frac{xdx}{y}$ be an exterior form on $J$ and similarly for $\widehat{J}$.
One has $ |\frac{\phi^*\hat{\omega}}{\omega}| = \frac{\Omega_{\widehat{J},\hat{\omega}}}{\Omega_{J,\omega}}\frac{|\ker\phi|_{\RR}|}{|\coker\phi|_{\RR}|}$ (as in \cite{LocalInvariantsIso} Lemma 7.4), where the periods $\Omega$  can be computed using Magma \cite{Magma} and $\frac{|\ker|}{|\coker|}$ using Lemma \ref{lem:kercoker}. 
Using Magma, we compute $\{\omega_1, \omega_2\}$ a basis of integral differentials at $p=2$. Then a minimal N\'eron exterior form for $J$ is given by $\omega^{\circ}=\omega_1\wedge\omega_2$ (\cite{RvB} \S 3.2), which lets us determine $\frac{\omega}{\omega^{\circ}}$, and similarly for $\widehat{J}$. 
Now use 
$$
\frac{\phi^*\hat{\omega}^{\circ}}{\omega^{\circ}} = 
\frac{\omega}{\omega^{\circ}} \cdot \frac{\phi^*\hat{\omega}}{\omega} \cdot \frac{\phi^*\hat{\omega}^{\circ}}{\phi^*\hat{\omega}} =\frac{\omega}{\omega^{\circ}} \cdot \frac{\phi^*\hat{\omega}}{\omega} \cdot \frac{\hat{\omega}^{\circ}}{\hat{\omega}}.
$$
\end{remark}

\subsection{Odd degree base change}
Finally, we record a basic observation regarding the behaviour of Conjecture \ref{conj:local} in odd degree unramified extensions.

\begin{lemma}\label{lem:odddegext}
Let $K$ be a finite extension of $\Q_p$ and $F/K$ an unramified extension of odd degree.
Let $C/K$ be a C2D4 curve with $\mathcal{P}, \Delta \ne 0$ and let $A=\Jac C$. Then
\begin{enumerate}
\item $c_{A/F}= n^2 c_{A/K}$ for some $n\in\Z$, with $n\!=\!1$ if $[F\!:\!K]$ is a sufficiently large prime;
\item $w_{A/F}=w_{A/K}$;
\item $C/K$ is deficient if and only if $C/F$ is deficient;
\item $E_{C/F}=E_{C/K}$;
\item Conjecture \ref{conj:local} holds for $C/K$ if and only if it holds for $C/F$.
\end{enumerate}
\end{lemma}
\begin{proof}
(1) Let $\Phi$ be the group of connected components of the special fibre of the N\'eron model of $A/\cO_K$ with its $\Gal(K^{nr}/K)$-action, so that $c_{A/K}=|\Phi^{\Frob_{K}}|$ and $c_{A/F}=|\Phi^{\Frob_{F}}|$. The group $\Phi$ carries a perfect symmetric $\Frob_K$-invariant pairing $\Phi\times\Phi\to\Q/\Z$ (see \cite{SGA7}, \cite{BosLor} Thm.\  2.3).
Since $\Frob_{F}$ is an odd power of $\Frob_K$, this forces $\Phi^{\Frob_{K}}$ to have a square index in $\Phi^{\Frob_F}$ (see e.g. \cite{DB} Thm.\  2.4.1(ii) with $f=[F:K]$ and $C_k$ the finite quotient of $\Gal(K^{nr}/K)$ through which the action on $\Phi$ factors).

(2) See e.g. \cite{Ces} Prop. 4.3. (3) Clear from the definition (see Definition \ref{de:deficiency}).

(4) 
Hilbert symbols are clearly unchanged, and hence so is $E_{C/K}$.

(5)
By (1, 2, 3, 4) the root number, deficiency, $E$ and the parity of the 2-adic valuation of the Tamagawa numbers are unchanged. 
Finally, minimal exterior forms are unchanged in unramified extensions, so for 2-adic primes, $\left|\frac{\phi^*\omega^\circ_{A'/K}}{\omega^\circ_{A/K}}\right|_F=\left|\frac{\phi^*\omega^\circ_{A'/K}}{\omega^\circ_{A/K}}\right|_K^{[F:K]}$ is also unchanged up to squares. The result thus follows from Lemma \ref{lem:kercoker}.
\end{proof}

\section{Main local theorem: base cases}\label{s:locconj1}

We now turn to the proof of Conjecture \ref{conj:local}, which relates local root numbers to the local term~$\lambda_{C/K}$. As outlined in \S\ref{ss:overview}, we begin by proving a number of cases through explicit computation, summarised in Theorem \ref{thm:localconjprelim}. The proof will occupy \S\ref{s:locconj1}--\S\ref{s:2adic}. In \S \ref{s:deformation}--\ref{s:locconj2} we will deduce the conjecture for the general class of C2D4 curves in Theorem \ref{thm:introlocal} by deforming them to number fields and using a global-to-local trick. Recall from \S\ref{ss:pictorial}, \ref{ss:clusterpics} that we draw pictures to indicate the distribution of the roots of $f(x)$ in $\RR$ or $\overline{\Q}_p$.

\begin{theorem}\label{thm:localconjprelim}
Let $C/K$ be a C2D4 curve over a local field of characteristic 0, with $\cP,\Delta,\eta_1\!\ne\!0$. 
Conjecture \ref{conj:local} holds for $C/K$ if either
\begin{enumerate}[leftmargin=*]
\item $K\iso\C$, or
\item $K\iso\R$, and the picture of $C$ is either
\begin{itemize}[leftmargin=*]
\item~ \raise3pt\hbox{\CPRNF[A][A][B][B][C][C]} 
 or
\raise3pt\hbox{\CPRNF[B][B][A][A][C][C]} 
 or
\raise3pt\hbox{\CPRNF[C][C][B][B][A][A]} 
or
\raise3pt\hbox{\CPRNF[A][A][B][C][B][C]} 
or
\raise3pt\hbox{\CPRNF[A][A][B][C][C][B]}
, or
\item~   
\raise3pt\hbox{\clusterpicture            
  \Root[B] {} {first} {r1};
  \Root[B] {} {r1} {r2};
  \frob (r1)(r2);
   \Root[C] {} {r2} {r3};
  \Root[C] {} {r3} {r4};
  \frob (r3)(r4);
 \endclusterpicture}
and $\a_1\!=\!\bar{\b}_1$, or
\item~ 
\raise3pt\hbox{\clusterpicture       
  \Root[B] {} {first} {r1};
  \Root[B] {} {r1} {r2};
  \frob (r1)(r2);
 \endclusterpicture}
and $\a_1\!=\!\bar{\b}_1$, $\a_3\!=\!\bar{\b}_3$,  or
\item $\a_i\!=\!\bar{\b}_i$ for $i\!=\! 1,2,3$, or
\end{itemize} 
\item $K/\Q_2$ is a finite extension and either
\begin{itemize}[leftmargin=*]
\item $C\!\in\! \mathcal{F}_{C2D4}$ (see Notation \ref{def:familyf}), or 
\item $C/K$ has good ordinary reduction, cluster picture $\CPUZERO[A][A][B][B][C][C]$ with depth of each twin equal to $v(4)$, $f(x)\in\cO_K[x]$ with roots in $K^{nr}$, and $\frac{\delta_2+\delta_3}{16}, \frac{\delta_2\eta_2+\delta_3\eta_3}{32}, \frac{\hat\delta_2\eta_3+\hat\delta_3\eta_2}{8}\in \cO_K^\times$, or
\end{itemize}
\item $K/\Q_p$ is a finite extension for odd $p$ and $C$ is semistable with cluster picture either
\begin{itemize}[leftmargin=*]
\item \GRND[D][D][D][D][D][D], \GRUNND[D][D][D][D][D][D], \CPTCNDZERO[A][B][B][A][C][C], \CPTCTNNDZERO[B][B][A][C][C][A], \CPUZERO[A][A][B][B][C][C], or
\item \CPTCTC[A][B][C][A][B][C] with $v(\l_1)=t$, or
\item \CPONNDZERO[D][D][D][D][D][D] or \CPTCONZERO[A][A][B][B][C][C]  
with $v(\ell_1)= v(\ell_2)=v(\ell_3)=v( \eta_2)=v(\eta_3)=0$, or 
\item \CPTNNDZERO[B][B][C][C][A][A], with $v(\ell_1)=v( \ell_2)=v(\ell_3)=0 $, or 
\item  \CPTNNMZERO[B][C][B][C][A][A] 
with $v(\l_1) = \min(n,m)$, $v(\ell_2)=v(\ell_3)=v(\eta_2)=v(\eta_3)=0$.
\end{itemize}
\end{enumerate}
\end{theorem}

\begin{proof}
(1,2) This follows from Theorems \ref{th:localconjecturecomplex} and \ref{th:localconjectureinfinite} below.

(3) This follows from Theorems \ref{thm:errorterm2} and \ref{thm:2orderrorterm}.

(4) A substitution $x\mapsto \pi_K^a x$, $y\mapsto \pi_K^{3a}y$ scales the roots by $\pi_K^a$ without changing the cluster picture or the leading term $c$. This does not change any of the Hilbert symbols in $E_{C/K}$ ($\delta_i, \eta_i\ldots$ all have even degree) nor $\lambda_{C/K}$. Thus we may assume that the depth of the maximal cluster~is~0. 
We may also assume that the C2D4 curve is centered, as a shift in the $x$-coordinate does not change any of the invariants. 

Theorem \ref{th:localconjectureoddprime} exhausts all possible Frobenius actions on these cluster pictures (after possibly recolouring $\smash{\raise4pt\hbox{\clusterpicture\Root[B]{1}{first}{r1};\endclusterpicture}} \leftrightarrow \smash{\raise4pt\hbox{\clusterpicture\Root[C]{1}{first}{r1};\endclusterpicture}}$). By the semistability criterion (Theorem \ref{th:ss}) $d_\s \in \Z$ for every cluster $\s$ of size $\ge 3$ and $d_\t \in \frac 12 \Z$ for all twins $\t$. Moreover for the cluster pictures \CPONND[B][C][B][C][A][A] and \CPONND[A][B][A][B][C][C] the depth of the twin lies in $\Z$ because $r(x), s(x), t(x) \in K^{nr}[x]$. The result follows from Theorem \ref{th:localconjectureoddprime}. 
\end{proof}


\section{Archimedean places}

\begin{theorem}\label{th:localconjecturecomplex}
Conjecture \ref{conj:local} holds for all C2D4 curves over $\C$ with $\cP, \Delta\neq 0$.
\end{theorem}

\begin{proof}
Here $w_{ C/\C}=1$ as $C$ has genus 2, and clearly $E_{C/K}\!=\! 1$ and $\lambda_{C/\C}\!=\!(-1)^2\!=\!1$.
\end{proof}


For curves over $\R$ we shall, for the moment, only prove Conjecture \ref{conj:local} in a restricted number of cases. 
The direct proof below can be extended to all cases (cf \cite{CelineThesis}), but we will obtain the remaining ones for free using our methods in \S \ref{s:deformation}--\ref{s:locconj2} (see Theorem \ref{th:localtheoremI22nonzero1}).

\begin{theorem}\label{th:localconjectureinfinite}
Let $C/\R$ be a C2D4 curve with $\mathcal{P}, \Delta, \eta_1 \neq 0$, with Richelot dual curve $\widehat{C}$, Richelot isogeny $\phi$  and $J$ and $\widehat{J}$ the Jacobians of $C$ and $\widehat{C}$.
The table below gives the values of $n_{J/\R}$, $n_{\widehat{J}/\R}$, $|J(\R)^\circ[\phi]|$, $\mu_{C/\R}$, $\mu_{\widehat{C}/\R}$, $\lambda_{C/\R}$, $w_{C/\R}$ and $E_{C/\R}$, depending on the roots associated to $C$ and the sign of $c$.
Conjecture \ref{conj:local} holds for all curves described in the table.

\noindent
\resizebox{\textwidth}{!}{
\begin{tabular}{| c | c || c | c | c | c | c |c|c|c|c|}
\hline
& \phantom{$X^{X^X}$}Roots\phantom{$X^{X^X}$} & $n_{J\!/\R}$& $n_{\widehat{J}\!/\R}$ & $|J(\R)^\circ[\phi]|$ &$\mu_{C\!/\!\R}$&$\mu_{\widehat{C}\!/\!\R}$ & $\lambda_{C\!/\R}$& $w_{C\!/\R} $& $E_{C\!/\R}$\cr
\hline
$1$$\vphantom{X^{X^X}}$ &$\a_1\!=\!\bar{\b}_1, \a_2\!=\!\bar{\b}_2, \a_3\!=\!\bar{\b}_3, c>0$ &$1$ & $4$ & $4$ & $1$ & $1$ & $1$ & $1$&$1$ \cr
\hline
$2$$\vphantom{X^{X^X}}$ & $\a_1\!=\!\bar{\b}_1, \a_2\!=\!\bar{\b}_2, \a_3\!=\!\bar{\b}_3, c<0$ & $1$ &$ 4$ &$ 4$ &$ -1$ & $1$ & $ -1$&$1$&$-1$ \cr
\hline
$3$$\vphantom{X^{X^X}}$&
 $\a_1\!=\!\bar{\b}_1, \a_3\!=\!\bar{\b}_3\>$
\smash{\raise4pt\hbox{\clusterpicture          
  \Root[B] {2} {first} {r1};
  \Root[B] {5} {r1} {r2};
  \frob (r1)(r2);
 \endclusterpicture}} 
& $1$ & $4$ & $4$ & $1$ & $1$ & $1$ & $1$&$1$ \\
\hline
$4$$\vphantom{X^{X^X}}$ &
 $\a_1\!=\!\bar{\b}_1$ \smash{\raise4pt\hbox{\clusterpicture       
  \Root[B] {2} {first} {r1};
  \Root[B] {5} {r1} {r2};
  \frob (r1)(r2);
   \Root[C] {2} {r2} {r3};
  \Root[C] {5} {r3} {r4};
  \frob (r3)(r4);
 \endclusterpicture}} 
&  $2$ & $4$ & $4$ & $1$ & $1$ & $ -1$&$1$&$-1$ \\
\hline
$5$&
\smash{\raise4pt\hbox{\CPR[A][A][B][B][C][C]}}
&$4$ & $4$ & $4$ & $1$ & $1$ & $1$ & $1$&$1$ \\
\hline
$6$&
\smash{\raise4pt\hbox{\CPR[B][B][A][A][C][C]}}
&$4$ & $4$ & $4$ & $1$ & $1$ & $1$ & $1$&$1$ \\
\hline
$7$&
\smash{\raise4pt\hbox{\CPR[C][C][B][B][A][A]}}
&$4$ & $4$ & $4$ & $1$ & $1$ & $1$ & $1$&$1$ \\
\hline
$8$ &
\smash{\raise4pt\hbox{\CPR[A][A][B][C][B][C]}}
&  $4$ & $2$ & $2$ & $1$ & $1$ & $1$ & $1$&$1$ \\
\hline
$9$ &
\smash{\raise4pt\hbox{\CPR[A][A][B][C][C][B]}}
&  $4$ & $4$ & $2$ & $1$ & $1$ & $ -1$&$1$&$-1$ \\
\hline
\end{tabular}}
\end{theorem}

\begin{proof}
Write $C$ as $y^2=cr(x)s(x)t(x)$ as in Definition \ref{de:C2D4Curve}. Note that in cases (3)--(9), the picture indicates that $c<0$. 
We find that the number of 
components of $C(\R)$ is 0 in case (2), 1 in cases (1, 3), 2 in case (4) and 3 in cases (5)--(9).  This gives the values of $n_{J/\R}$ (Lemma \ref{pr:n}) and $\mu_{C/\R}$  (a curve $C$ of genus 2 is deficient over $\R$ if and only if $C(\R) = \emptyset$, see Definition \ref{de:deficiency}).

Recall that $\widehat{C}$ is also a C2D4 curve with equation $y^2 = \frac{\ell_1\ell_2\ell_3}{\Delta}\hat{r}(x)\hat{s}(x)\hat{t}(x)$.
From Definitions \ref{def:invariants}, \ref{def:richelotdual} we see that as $r(x), s(x), t(x)$ all have $\R$-coefficients in all the above cases, so do $\hat{r}(x), \hat{s}(x), \hat{t}(x)$. In particular, whether $\hat{\a}_i, \hat{\b}_i$ are real is determined by the sign of the discriminant of the corresponding quadratic. 
The discriminants of $\hat{r}(x), \hat{s}(x), \hat{t}(x)$ are $\frac{4\Delta^2}{\ell_1^2}\hat{\delta_1}$, $\frac{1}{\ell_2^2}\hat{\delta_2}$ and $\frac{1}{\ell_3^2}\hat{\delta_3}$ (Definition \ref{def:richelotdual}). Explicitly,
$$
 (\hat{\a}_1-\hat{\b}_1)^2= \frac{4(\a_2-\a_3)(\a_2-\b_3)(\b_2-\a_3)(\b_2-\b_3)}{(\a_2+\b_2-\a_3-\b_3)^2},
$$
with similar expressions for $(\hat{\a}_2-\hat{\b}_2)^2$ and $(\hat{\a}_3-\hat{\b}_3)^2$, obtained by permuting the indices~1,~2,~3.

If $\a_2, \b_2, \a_3, \b_3\in \R$, the above discriminant is positive if and only if the roots of the two quadratics are not interlaced (they are ``interlaced'' if $\bullet_2\!<\!\bullet_3\!<\!\bullet_2\!<\!\bullet_3$ or vice versa). If either $\a_2=\bar{\b}_2$ or $\a_3=\bar{\b}_3$ the discriminant is always positive, being of the form $\frac{z\bar{z}}{w^2}$ for $w\in\R$. An identical analysis applies to $(\hat{\a}_2-\hat{\b}_2)^2$ and $(\hat{\a}_3-\hat{\b}_3)^2$.

Putting this information together and considering the sign of the leading term, we deduce that $\Ch$ has 3 real components in all cases above, except for case (8), when it has 2 real components. This gives the values for $n_{\widehat{J}/\R}$ and $\mu_{\widehat{C}/\R}$ as above for $C$.

Now consider $\phi$-torsion on the connected component of the identity of the Jacobian, $J(\R)^\circ[\phi]$. The nontrivial $\phi$-torsion points are represented by the divisors $[(\alpha_i,0), (\beta_i,0)]$, and the identity by all pairs of the form $[(x,y), (x,-y)]$. If $\beta_i=\bar{\a}_i$, then there is a path on $J(\R)$ of the form $[(w,z),(\bar{w},\bar{z})]$ from  $[(\alpha_i,0), (\beta_i,0)]$ to the identity (take a path to any $w\in \R$). If $\a_i, \b_i\in\R$ and $(\a_i,0)$ and $(\b_i,0)$ lie on the same component of $C(\R)$, then  moving $(\beta_i,0)$ to $(\alpha_i,0)$ along $C(\R)$ gives a path from $[(\alpha_i,0), (\beta_i,0)]$ to the identity. However, if $\a_i, \b_i\in\R$ and $(\a_i,0)$ and $(\b_i,0)$ do not lie on the same component of $C(\R)$, then no such path exist: both points have to remain in $C(\R)$ on the path as the $x$-coordinates will never have the same real part and hence will never be complex conjugate. This fully determines the order of $J(\R)^\circ[\phi]$.

The formula for $\lambda_{C/K}$ now follows from Lemma \ref{lem:kercoker}.

As $C$ has genus 2, the Jacobian is 2-dimensional and $w_{J/\R}=(-1)^{\dim J}=1$. Conjecture \ref{conj:local} for all the cases in the table will thus follow once we justify the formula for $E_{C/\R}$.

We finally turn to $E_{C/\R}$. This will be done by a case-by-case analysis of Hilbert symbols. For convenience, we may assume that the curve is centered, that is $\a_1=-\b_1$, as (by definition) this does not affect any of the Hilbert symbols defining $E_{C/\R}$.

\noindent \underline{\textbf{Cases 1,2}}. From their definitions $\ell_1^2, \eta_1, \dlu, \dld, \dlt>0$ and $\dgu, \dgd, \dgt <0$. Thus it follows that
$
E_{C/\RR}=-(\dgd \eta_2+\dgt \eta_3, -\eta_2\eta_3)(\dld\eta_3+\dlt\eta_2, -\eta_2\eta_3) \cdot(\eta_2\eta_3, -1)(c,-1).
$
In this case, $\eta_2, \eta_3 \in \RR$. 
If $\eta_2\eta_3<0$ then  $E_{C/\RR}= (c,-1)$. 
If $\eta_2,\eta_3>0$ then $\dgd \eta_2+\dgt \eta_3<0$ and $\dld\eta_3+\dlt\eta_2>0$ which yields $E_{C/\RR}= (c,-1)$. Finally, if $\eta_2,\eta_3<0$ then $\dgd \eta_2+\dgt \eta_3>0$ and $\dld\eta_3+\dlt\eta_2<0$ and $E_{C/\RR}= (c,-1)$. Thus we always have $E_{C/\R}=$ sign($c$), as claimed.

\noindent \underline{\textbf{Case 3}}. Here $\ell_1^2, \DG^2,\eta_2,\dgd,\dlu, \dld,\dlt>0$ and $c,\dgu,\dgt<0$. Hence
$
E_{C/\RR} =(\dgd \eta_2+\dgt \eta_3, \eta_2\eta_3)\cdot$ $\cdot(\dld\eta_3+\dlt\eta_2, -\eta_2\eta_3).
$
Either $\eta_3>0$ and $\dld\eta_3\!+\!\dlt\eta_2\!>\!0$ so that $E_{C/\RR}=1$, or $\eta_3<0$ and $\dgd \eta_2\!+\!\dgt \eta_3>0$, so that $E_{C/\RR}=1$. 

\noindent \underline{\textbf{Case 4}}. Here $\ell_1^2, \DG^2, \dgd, \dgt, \eta_1,\eta_2,\eta_3, \dlu,\dld, \dlt >0$ and $c, \dgu<0$,
so that $E_{C/\RR} \!=\! (\dgu,c)\!=\!-1.$

In the remaining cases $c<0$ and all roots are real, so $\dgu, \dgd, \dgt, \ell_1^2, \DG^2\!>\!0$, and 
$
 E_{C/\RR}  = (\dgd \eta_2+\dgt \eta_3, -\eta_2\eta_3)\cdot(\dld\eta_3+\dlt\eta_2, -\eta_2\eta_3\dld\dlt)(\xi\eta_2\eta_3, -\dld\dlt)
 \cdot(-c, \dld\dlt)(\eta_1, -\dlu)(\dlu, -\frac{\ell_1}{\DG}).
 $

\noindent \underline{\textbf{Cases  5,6,7}}. Here $\eta_1,\eta_2, \eta_3, \xi, \dlu, \dld, \dlt>0$, so that $E_{C/\RR} =1$.  

\noindent \underline{\textbf{Case  8}}. Here  $\eta_2, \eta_3, \xi, \dld, \dlt\!>\!0$ and  $\dlu, \l_1\!<\!0$, so that $E_{C/\RR}\!=\!(-1, -\frac{\ell_1}{\DG})$. Noting that $\DG = -c((\a_1\! +\! \b_2)(\a_1 \!-\! \b_2)(\a_2 \!-\! \a_3) + (\b_3 \!-\! \b_2)((\b_2 \!-\! \a_1)(\a_3 \!-\! \a_2) + (\a_2 \!+\! \a_1)(\a_3 \!-\! \a_1)))$, it follows that $\DG\!>\!0$ so that $\frac{\l_1}{\DG}\!<\!0$ and $E_{C/\R}\!=\!1$.

\noindent \underline{\textbf{Case 9}}. Here $\eta_2, \eta_3, \xi, \dlu, \dld, \dlt>0$  and $\eta_1<0$, so that $E_{C/\RR} =-1$. 
\end{proof}

\section{Changing the model by M\"obius transformations}\label{s:mobius}

\def\aa{{\mathbf a}}
\def\bb{{\mathbf b}}
\def\cc{{\mathbf c}}
\def\dd{{\mathbf d}}

For the proof of our main results on Conjecture \ref{conj:local} 
it will often be useful to be able to change the model of a C2D4 curve. This does not change the classical arithmetic invariants, but it does affect the terms $\Delta, \xi, \ldots$ that enter $E_{C/K}$ and hence Conjecture \ref{conj:local}. In this section we discuss possible changes of model and their effect on these terms.

\subsection{$\GL_2$ action on models}

\begin{definition}\label{def:Cm}
Let $C$ be a C2D4 curve over a field $K$ of characteristic 0,
$$
C:y^2 = c \prod_{i=1}^3 (x-\alpha_i)(x-\beta_i),
$$
and $m\!=\!\smallmatrix{\aa}{\bb}{\cc}{\dd}\!\in \!\GL_2(K)$ such that the associated M\"obius map $m(z)\!=\!\frac{\aa z+\bb}{\cc z+\dd}$ has $m(\alpha_i), m(\beta_i)\!\neq\!\infty$.
We define the model $C_m$ of $C$ by
$$
C_m: y_m^2 = c_m\prod_{i=1}^3 (x_m-\alpha_i')(x_m-\beta_i'), 
$$
with
$\alpha_i'\!=\!m(\alpha_i)$, $\beta_i'\!=\!m(\beta_i)$,  $c_m\!=\!c\prod_{i=1}^3 (\cc\alpha_i\!+\!\dd)(\cc\beta_i\!+\!\dd)$, via the transformation $x_m = \frac{\aa x+\bb}{\cc x+\dd}$ and $y_m = y\cdot\frac{(\aa\dd-\bb\cc)^3}{(\cc x+\dd)^3}$.
One can check that this construction satisfies 
$(C_{m})_{m'}=C_{m'm}$.
\end{definition}

\begin{remark}[see also \cite{LiuM} \S2]\label{rmk:models}
If a genus 2 curve over $K$ admits two hyperelliptic models $C:y^2=c f(x)$ and $C':y_2^2=c_2 f_2(x_2)$, then the $x$-coordinates are always related by a M\"obius map $x_2 = m(x)= \frac{\aa x+\bb}{\cc x+\dd}$ for some $\aa, \bb, \cc, \dd\in K$ (because these are the only transformations on $\mathbb{P}^1$ that is the quotient of the curve by the hyperelliptic involution). If both equations have degree 6, the model $C'$ then agrees with $C_m$ up to scaling the $y$-coordinate by a suitable constant, $y_2=\lambda y_m$ for some $\lambda\in K$.
\end{remark}

It will be particularly convenient to have a 1-parameter family of models:

\begin{definition}\label{def:Mt}
Let $K$ be a field  of characteristic 0 and $\alpha_1^2\in K$.
For $t\in K\setminus\{\frac{1}{\a_1}, -\frac{1}{\a_1}\}$, define $M_t=\smallmatrix{1}{t\alpha_1^2}{t}{1}$ with the corresponding M\"obius map over $K$ given by $M_t(z)=\frac{z+t\alpha_1^2}{tz+1}$. Write also $M_\infty=\smallmatrix{0}{\alpha_1^2}{1}{0}$ and $M_\infty(z)=\frac{\alpha_1^2}{z}$.

We will use the shorthand notation $C_t\!=\!C_m$ for centered curves $C\!:\!y^2\!=c\! \prod_{i=1}^3 (x-\alpha_i)(x-\beta_i)$.
Note that  $M_t(\alpha_1)=\alpha_1$ and $M_t(-\alpha_1)=-\alpha_1$, so that $C_t$ is also centered.
\end{definition}

\begin{lemma}\label{le:Mobius}
Let $M$ be a M\"obius map defined over a field $K$. Suppose that $\a_1^2 \in K$ and that $M(\a_1)=-M(-\a_1)$. Then $M= r \circ M_t$ for some $t \in K\cup\{\infty\}$, where $r(z) = \frac{M(\a_1)}{\a_1}z$.
\end{lemma}
\begin{proof}
First note that $\frac{M(\a_1)}{\a_1} \in K$. Hence $r^{-1} \circ M$ is defined over $K$ and fixes $\a_1$ and $-\a_1$. Now write $r^{-1} \circ M(z) = \frac{\aa z+\bb}{\cc z+\dd}$. Since $r^{-1} \circ M(\a_1) = \a_1$ and $r^{-1} \circ M(-\a_1) = -\a_1$ we have $\aa=\dd$ and $\bb=\cc\a_1^2$. If $\aa\ne0$ the map is of the required form. If $\aa=0$ then $r^{-1} \circ M(z) = \frac{\a_1^2}{z}$, so $M(z)=r\circ M_\infty(z)$.
\end{proof}

\subsection{Rebalancing}

\begin{theorem}\label{thm:mobiusbalance}
Let $K$ be a finite extension of $\Q_p$ for an odd prime $p$, with residue field $k$ of size $|k|>5$, and $C/K$ a semistable C2D4 curve.
There is $m\in \GL_2(K)$ such that $C_m$ is balanced.
\end{theorem}

\begin{proof}
Theorem \ref{th:existbalance} and Remark \ref{rmk:models}.
\end{proof}

\begin{theorem}\label{thm:mobiusoverbalance}
Let $K$ be a finite extension of $\Q_p$ for an odd prime $p$, with residue field $k$ of size $|k|\ge 23$.
Let $C/K$ be a centered balanced semistable C2D4 curve. 
Then there is a $t_0 \in K$ such that for all $t\in K$ with $v(t-t_0)\!>\!0$ the cluster picture of $C_t$ with signs and Frobenius action on proper clusters, its colouring, $v(c)$ and $v(\Delta)$ are the same as that of $C$ and
\begin{itemize}[leftmargin=*]
\item if the cluster picture of $C_t$ is 
\GRND[D][D][D][D][D][D], 
\CPONND[D][D][D][D][D][D], 
\CPTNND[B][B][C][C][A][A], 
\CPU[B][B][C][C][A][A] or 
\CPTCON[A][A][B][B][C][C] 
then $C_t$ has $v(\ell_1)=v(\ell_2)=v(\ell_3)=v(\eta_2)=v(\eta_3)=0$;
\item if the cluster picture of $C_t$ is 
\CPTCND[B][B][A][C][C][A] or 
\CPTCTN[B][B][A][C][C][A] 
then $C_t$ has  $v(\ell_1)=v(\ell_2)=v(\ell_3)=0$;
\item if the cluster picture of $C_t$ is  
\CPTNNM[B][C][B][C][A][A] 
then $C_t$ has  $v(\ell_2)=v(\ell_3)=v(\eta_2)=v(\eta_3)=0$ and $v(\ell_1)=\min(n,m)$.
\item if the cluster picture of $C_t$ is  
\CPTCJ[R][S][T][R][S][T] 
then $C_t$ has $v(\ell_1)=v(\ell_2)=v(\ell_3)=j$.
\end{itemize}
\end{theorem}

\begin{proof}
Since $C$ is centered and balanced, all the roots are necessarily integral (Lemma \ref{lem:integralroots}).

One readily checks that
$$
  M_t(r_1)-M_t(r_2) = (r_1-r_2)\cdot \frac{(1-t^2\alpha_1^2)}{(tr_1+1)(tr_2+1)}.
$$
In particular, so long as $t\not\equiv -1/r_1, -1/r_2, 1/\alpha_1, -1/\alpha_1$ in $k$, one necessarily has $v(r_1-r_2)=v(M_t(r_1)-M_t(r_2))$. Thus, if $t\not\equiv -1/r$ in $k$ for any root $r$, then $C_t$ has the same cluster picture as $C$, with the same colouring. The Galois action on proper clusters and the signs of clusters is the same by Lemma \ref{lem:mobiussign} below. Moreover, the same condition on $t$ ensures that the valuation of the leading term $c$ and of $\Delta/c$ remain unchanged (cf. Definition \ref{def:Cm}, Lemma \ref{lem:RebalanceDelta0}(i) below).

Recall that (for $C$) $\ell_3=\alpha_2+\beta_2$. Now
$$
M_t(\a_2)+M_t(\b_2) = \frac{t^2\a_1^2(\a_2+\b_2)+2t(\a_1^2+\a_2\b_2)+(\a_2+\b_2)}{(t\a_2+1)(t\b_2+1)}.
$$
Observe that the numerator is the zero polynomial in $k(t)$ if and only if $\a_2 \equiv -\b_2$ and $\a_2^2 \equiv \a_1^2$ in $k$. 
This is equivalent to $\alpha_2\equiv \pm \alpha_1$ and $\beta_2\equiv \mp \alpha_1$, which would mean that there is a cluster of depth $>0$ containing $\alpha_2$ and $\pm\alpha_1$ and one containing $\beta_2$ and $\mp\alpha_1$. This is not the case for the listed cluster pictures, except for \CPTCJ[R][S][T][R][S][T]
, so the numerator is not the zero polynomial in $k(t)$ for these. It follows that, so long as $t$ avoids the roots of the polynomial in $k$ and the residues of $-1/\alpha_2, -1/\beta_2$, the expression $M_t(\alpha_2)+M_t(\beta_2)$ will have valuation 0 in $\bar{K}$. Repeating a similar argument for $\ell_2$ 
shows that $C_t$ has $v(\ell_2)=v(\ell_3)=0$ so long as $t$ avoids a specific list of residue classes of $k$.
For the exceptional cluster picture \CPTCJ[R][S][T][R][S][T]
, the coefficients of the numerator all have valuation $\ge j$, and one similarly checks that at least one has valuation exactly $j$, so that a suitable choice of $t$ makes $v(\ell_2)=v(\ell_3)=j$.

The arguments for $\ell_1$ and $\eta_2$, $\eta_3$ are similar. Recall that (for $C$), $\ell_1=\alpha_2+\beta_2-\alpha_3-\beta_3$. 
Writing $\aa=\a_2\!+\!\b_2\!-\!\a_3\!-\!\b_3$, $\bb=\a_2\b_2\!-\!\a_3\b_3$ and $\cc=\a_2\b_2\a_3\!-\!\a_2\a_3\b_3\!+\!\a_2\b_2\b_3\!-\!\a_3\b_2\b_3$,
one checks that 
$$
\ell_1(C_t) =
 \frac{\aa+ 2\bb t+
 (-\alpha_1^2\aa\!+\!\cc)t^2-2\a_1^2\bb t^3 -\a_1^2\cc t^4
 }{(t\a_2\!+\!1)(t\b_2\!+\!1)(t\a_3\!+\!1)(t\b_3\!+\!1)},
$$
and that if the numerator reduces to the zero polynomial in $k(t)$ then $\{\bar{\a}_2,\bar{\b}_2\} = \{\bar{\a}_3,\bar{\b}_3\}$. This is not the case in the listed cluster pictures, except for \CPTNNM[B][C][B][C][A][A]
and \CPTCJ[R][S][T][R][S][T]
,
and picking $t$ that avoids the residue classes that make the numerator or denominator 0 in $k$ makes $v(\ell_1(C_t))=0$.
For the two exceptional cluster pictures each coefficient has valuation $\ge \min(n,m)$ (respectively $\ge j$), and one easily checks that either $\aa$ or $\bb$ must have precisely this valuation. Picking $t$ similarly gives $v(\ell_1(C_t))=\min(n,m)$ (respectively $=j$).

For $\eta_2$ one finds that 
$$
\eta_2(C_t)=\frac{(\a_2^2\!+\!\b_2^2\!-\!2\a_1^2)
 +2t\bb
 +2t^2\cc
 -2t^3\a_1^2\bb
 +t^4\a_1^2(\a_1^2\a_2^2\!+\!\a_1^2\b_2^2\!-\!2\a_2^2\b_2^2) 
 }{(t\a_2+1)^2(t\b_2+1)^2},
$$
where $\bb=(\a_2\b_2\!-\!\a_1^2)(\a_2\!+\!\b_2)$ and $\cc=(\a_1^2\!-\!\a_2^2)(\a_1^2\!-\!\b_2^2)$. Here the numerator reduces to zero in $k(t)$ only if $\a_1^2\equiv \a_2^2 \equiv \b_2^2$ in $k$, equivalently only if $\alpha_2\equiv \pm \alpha_1$ and $\beta_2\equiv \pm \alpha_1$ (and similarly for $\eta_3$).
This only happens for \CPTCND[B][B][A][C][C][A] and \CPTCTN[B][B][A][C][C][A]
of the listed cluster pictures, which make no claim for $\eta_2, \eta_3$. 
Thus $C_t$ will have $v(\eta_2)=v(\eta_3)=0$, so long as $t$ avoids the residue classes that make either the numerators or denominators of $\eta_2(C_t), \eta_3(C_t)$ reduce to 0 in~$k$.

The total number of residue classes $t$ has to avoid is at most $6$ (of the form $-1/r$ for a root $r$, that account for all the denominators) plus $2+2$ (for $\ell_2, \ell_3$) plus $4$ (for $\ell_1$) plus $4+4$ (for $\eta_2, \eta_3$), that is 22.
\end{proof}

\begin{lemma}\label{lem:mobiussign}
Let $K$ be a finite extension of $\Q_p$ for an odd prime $p$ and 
$m\!=\!\smallmatrix{\aa}{\bb}{\cc}{\dd}\!\in \!\GL_2(K)$.
Suppose that $C$ and $C_m$ are semistable, balanced C2D4 curves over $K$, and that $r\mapsto m(r)$ induces a bijection between the sets of twins and preserves their relative depths.
Then $r\mapsto m(r)$ also commutes with the Galois action on twins and preserves the signs of clusters of even size (after possibly suitably choosing signs of $\theta_{m(\s)}$ for twins $\s$ of $C$).
\end{lemma}
\begin{proof}
Since $\aa, \bb, \cc, \dd \in K$, the Galois action on the roots for $C$ is the same as on the roots on $C_m$, and so the map respects the Galois action on twins. It remains to check that it respects signs.

Suppose $C$ has exactly one twin, $\s_1$.
As the two curves are isomorphic, Theorem \ref{th:genus2bible} tells us that the sign must be the same for $\s_1$ and $m(\s_1)$ (see types $\II_n^\epsilon$ and $1\times\II_n^\epsilon$).

Suppose $C$ has two twins, $\s_1$ and $\s_2$.
If these are swapped by Frobenius then choosing the signs of $\theta_{m(\s_i)}$ appropriately guarantees that the sign of $\s_1$ agrees with that of $m(\s_1)$; by Theorem~\ref{th:genus2bible} the signs of $\s_2$ and $m(\s_2)$ must then also agree (see types $\I_{n\sim n}^\epsilon$ and $\I_n^\epsilon\tilde\times \I_n$).
If the twins are not swapped but have the same sign, then the result again follows by Theorem \ref{th:genus2bible} (types $\I_{n, m}^{+, +}$, $\I_{n, m}^{-, -}$, $\I_n^+\times \I_m^+$, $\I_n^-\times \I_m^-$).
If the twins are not swapped, have different signs and different relative depths, the result follows from the structure of the N\'eron component group by Remark~\ref{rmk:Inmcomponentgroup}, after possibly passing to a quadratic ramified extension (types  $\I_{n, m}^{+, -}$, $\I_n^+\times \I_m^-$ with $n\neq m$).
If the twins are not swapped by Frobenius, have different signs (say, + for $\s_1$ and - for $\s_2$) and equal relative depths (say $\delta_{\s_i}=n$), we unfortunately need to use the explicit description of the minimal regular models and the reduction map to the special fibre (see \cite{m2d2} Thm.\  8.5 and \S 5.6): passing to a quadratic ramified extension if necessary so that $n$ is even, the special fibre of the minimal regular model of $C$ is
\begin{center}\includegraphics[scale=0.15]{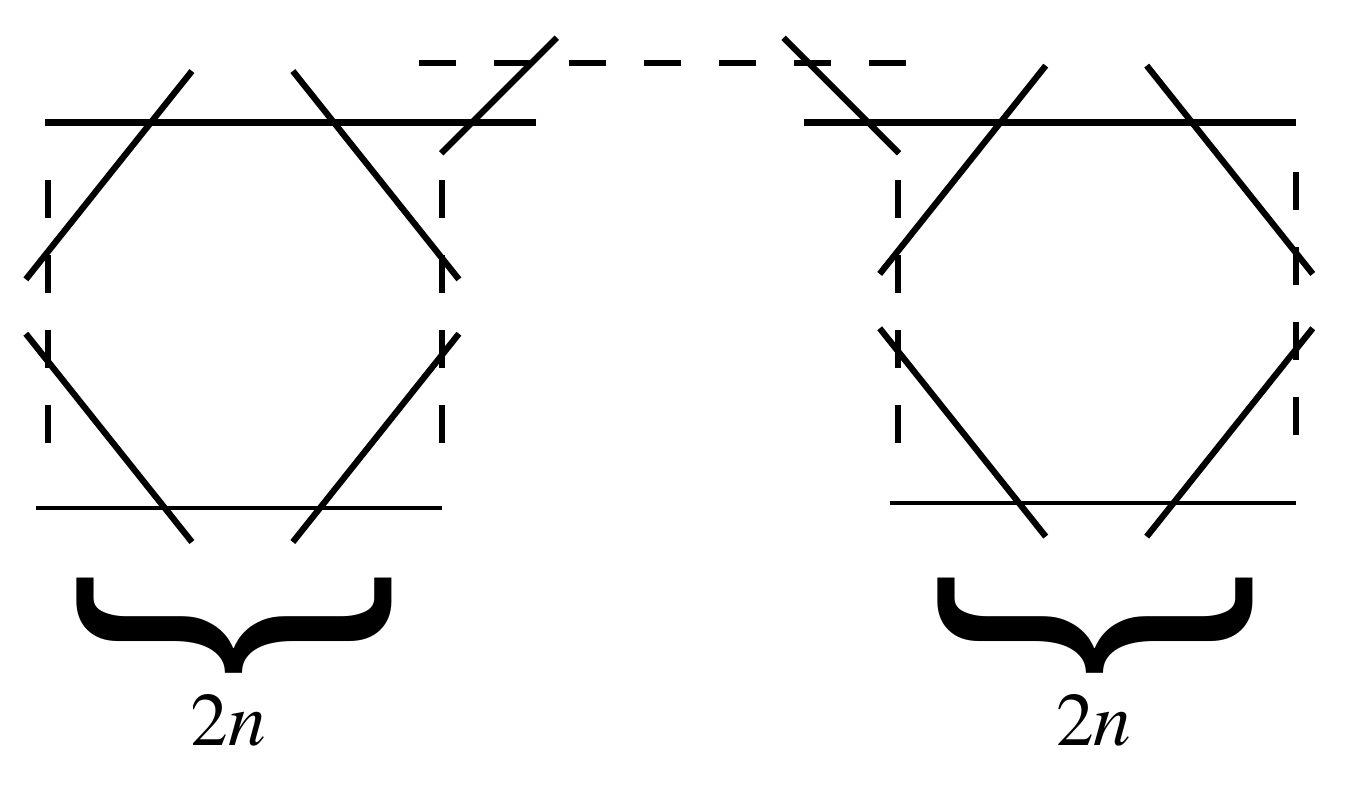}\end{center}
\vskip -.3cm
and Frobenius fixes the components on the left $2n$-gon and acts as a reflection on the right $2n$-gon. The two Weierstrass points of $C$ that correspond to the roots in the twin $\s_1$ reduce to the component in the left $2n$-gon that is furthest away from the central chain, corresponding to the fact that the sign of $\s_1$ is $+$, and the two Weierstrass points for $\s_2$ reduce to the corresponding component in the right $2n$-gon. As this description is model-independent, it follows that the M\"obius transformation must preserve signs of the twins.

Finally, when $C$ and $C_m$ have three twins each, the only signs are those of the full sets of roots. These agree by Theorem \ref{th:genus2bible} (types $\UU_{*}^\epsilon$).
\end{proof}

\subsection{Change of invariants}

\begin{lemma}\label{lem:RebalanceDelta0}
Let $C$ be a C2D4 curve over a field $K$ of characteristic 0. \\
(i) For $t\in K\setminus\{\frac{1}{\alpha_1},-\frac{1}{\alpha_1}\}$, 
$$
{\Delta(C_t)}/{c(C_t)}=  
\frac{(1+\alpha_1 t)^2(1-\alpha_1 t)^2}{(1+\alpha_2 t)(1+\beta_2 t)(1+\alpha_3 t)(1+\beta_3 t)} 
\cdot {\Delta(C)}/{c(C)}.
$$
(ii) For $m(z)=1/z$,
$$
{\Delta(C_m)}/{c(C_m)}=  \frac{1}{\alpha_1^2\alpha_2\beta_2\alpha_3\beta_3}\cdot{\Delta(C)}/{c(C)}.
$$
(ii) For $m(z)=\lambda z$,
$$
{\Delta(C_m)}/{c(C_m)}= \lambda^3\cdot {\Delta(C)}/{c(C)}.
$$
\end{lemma}
\begin{proof}
Direct computation.
\end{proof}

\begin{lemma}\label{lem:RebalanceDelta}
Let $K/\Q_p$ be a finite extension with $p$ an odd prime, and let $C/K$ be a 
C2D4 curve with cluster picture \GRUNZERO[D][D][D][D][D][D]
.
Then
$$
 v\left(\frac{\Delta(C)}{c(C)}\right)=2n+v\left(\frac{\Delta(C_m)}{c(C_m)}\right),
$$
for any $m\in\GL_2(K)$ such that $C_m/K$ is balanced.
\end{lemma}

\begin{proof}
Enlarging 
$K$ if necessary, we may pick $z\in K$ which has $v(z-r)=n$ for the roots inside the cluster of size 5 of $C$ and $v(z-r)=0$ for the remaining root. One checks that applying the following M\"obius transformation yields a model $C_m$ with a balanced cluster picture: $m:x\mapsto x-z \mapsto \frac{1}{x-z}\mapsto \frac{\pi_K^n}{x-z}$. By Lemma \ref{lem:RebalanceDelta0}, $v(\Delta(C_m))-v(c(C_m))=v(\Delta)-v(c)-5n+3n$. 

It remains to show that if $C$ and $C_{m}$ are both balanced models 
then $v(\Delta(C))=v(\Delta(C_{m}))$. As $\Delta$ is invariant under shifts of the $x$-coordinate, we may assume that both $C$ and $C_m$ are centered; in particular $\alpha_1(C)$ and $\alpha_1(C_m)$ are both units. By Lemma \ref{le:Mobius}, the associate M\"obius transformation is of the form $m= r \circ M_t$ for some 
$t$, where $r(z) = \lambda z$ and $\lambda=\frac{\a_1(C_m)}{\a_1(C)}\in K^{nr}$. As the roots are integral with distinct images in the residue field for both curves, we find that $t\!\not\equiv\! \pm\alpha_1^{-1},\alpha_2^{-1},-\beta_2^{-1},\alpha_3^{-1},-\beta_3^{-1}$ in the residue field. The result now follows from Lemma~\ref{lem:RebalanceDelta0}(i).
\end{proof}

\begin{lemma}\label{lem:d1hsquare}
For a C2D4 curve $C$ over a field $K$ of characteristic 0 and $m\in \GL_2(K)$, 
$$
\hat\delta_1(C)=\hat\delta_1(C_m)\cdot\square.
$$
\end{lemma}
\begin{proof}
As $\dlu$ is invariant under shifts of the $x$-coordinate, we may assume that both $C$ and $C_m$ are centered. By Lemma \ref{le:Mobius} the associate M\"obius transformation is of the form $m= r \circ M_t$ for some $t \in K\cup\{\infty\}$, where $r(z) = \lambda z$ and $\lambda=\frac{\a_1(C_m)}{\a_1(C)}\in K$. As $c^2\hat\delta_1$ is a homogeneous rational function of even degree in the roots, $\hat\delta_1(C_{t})=\hat\delta_1(C_m)\cdot\square$. As $\hat\delta_1=(\frac{\ell_1}{2\Delta})^2(\hat\alpha_1-\hat\beta_1)^2$ (see Definition \ref{def:richelotdual}) and 
$\ell_1/\Delta\in K$ 
it suffices to check that $(\hat\alpha_1(C)-\hat\beta_1(C))^2=(\hat\alpha_1(C_t)-\hat\beta_1(C_t))^2\cdot\square$. Explicit computation shows that $(\hat\alpha_1(C_t)-\hat\beta_1(C_t))^2=$
\begingroup\smaller[2]
$$
(\hat\alpha_1(C)-\hat\beta_1(C))^2
\frac{(-1 + \a_1^2 t^2)^2\cdot (\a_2-\a_3+\b_2-\b_3)^2}
{(\ell_1(C) + (2 \a_2 \b_2 - 2 \a_3 \b_3) t + (\a_2 \a_3 \b_2 - \a_2 \a_3 \b_3 +  \a_2 \b_2 \b_3 - \a_3 \b_2 \b_3) t^2)^2},
$$
\endgroup
with the natural extension of the formula to $t\!=\!\infty$.
\end{proof}

\section{Odd places}

Here we state an analogue of Theorem \ref{th:localconjectureinfinite} for C2D4 curves over finite extensions of $\Q_p$ for odd primes $p$. Its proof will occupy \S \ref{s:oddisogeny} and \S \ref{s:oddhilbert}. We will remove the constraints on valuations in Theorem \ref{thm:localconjI22zero} and extend it to all semistable curves in Theorem \ref{th:localtheoremI22nonzero1}.

\begin{theorem}\label{th:localconjectureoddprime}
Let $K/\Q_p$ be a finite extension for an odd prime $p$ and $C/K$ a centered semistable C2D4 curve with $\ell_1, \ell_2, \ell_3, \Delta\neq 0$ whose cluster picture with Frobenius action is one of the cases described in the table below. Let $J$ and $\widehat{J}$ be the Jacobians of $C$ and $\widehat{C}$.

Then $c_{J/K}$, $\mu_{C/K}$, $c_{\widehat{J}/K}$, $\mu_{\widehat{C}/K}$, $\lambda_{C/K}$ and $w_{C/K}$ are given by the corresponding columns in the table. If the residue field has size $|k|\ge 23$ then $\widehat{C}$ admits a model whose cluster picture with Frobenius action is given in the $\widehat{C}$ column.
$E_{C/K}$ is as given in the table provided $\mathcal{P}, \eta_1  \ne 0$ and:
\begin{itemize}[leftmargin=*]
\item $v(\ell_1)=v(\l_2)=v(\l_3)=v(\eta_2)=v(\eta_3)=0$ in cases $\Io_{n}^{\epsilon}(a,b)$, $\Io_{2n}^{\epsilon}(c,d)$ and  $\Io\times_t\II_{n}^{\epsilon}(a)$,
\item $v(\ell_1)\!=\!v(\l_2)\!=\!v(\l_3)\!=\!0$ in cases $\II_{n,m}^{\epsilon,\delta}(a)$, $\II_{n\sim n}^{\epsilon}(a)$, 
\item $v(\ell_1)=t$ in cases $\Io\!\times_t\! \Io(b,c)$ and $\Io\tilde\times_t \Io(b,c)$, and
\item $v(\l_1) = \frac{n}{2}$, $v(\l_2)=v(\l_3)=v(\eta_2) =v(\eta_3)=0$ in cases $\II_{n,m}^{\epsilon,\delta}(b)$ and $\II_{n\sim n}^{\epsilon}(b)$.
\end{itemize}
In the table $k,l,m,n,t \in \Z_{>0}$ are parameters and $r \in\Z$ is defined by the column of $v(\DG/c)$. In the $\widehat{C}$ column, a cluster of size 3 with index 0 means that the roots in it do not form a cluster, e.g. \CPTCTCR[A][B][C][A][B][C] with $r=0$ means \GR[A][A][B][B][C][C].

\begin{table}[H]
\label{tb:Inv}
\hspace*{-.8cm}
\resizebox{1.1\linewidth}{!}{
\begin{tabular}{|l|l|c||l|c|c|c|c|c|c|c| }
\hline
Case &$C$ & $v(\Delta/c)$&$\widehat{C}$ &  $c_{J/K}$& $\mu_{C/K}$&$c_{\widehat{J}/K}$&$\mu_{\widehat{C}/K}$ & $\lambda_{C/K}$ & $w_{C/K}$ & $E_{C/K}$\\
\hline
\makecell[l]{ \It$(a)$\\
\It$(b)$\\
\It$(c)$}
&
\makecell[l]{\GR[A][A][B][B][C][C]\\
\GRU[A][A][B][B][C][C]~$\dlu\in K^{\times 2}$\\
\GRU[B][A][A][B][C][C]}
&
\makecell{$r$\vspace{0.2cm}\\
$2k\!+\!r$\vspace{0.2cm}\\
$2k\!+\!r$}
&
\CPTCTCR[A][B][C][A][B][C]
&
$1$
&
$1$
&
$1$
&
$1$
&
$1$
&
$1$
&$1$\\
\hline
\makecell[l]{ \It$(d)$\\
\It$(e)$\\
\It$(f)$}
&
\makecell[l]{\GR[A][A][B][B][C][C]\\
\GRU[A][A][B][B][C][C]~$\dlu\notin K^{\times 2}$\\
\GRU[B][A][A][B][C][C]}
&
\makecell{$r$\vspace{0.2cm}\\
$2k\!+\!r$\vspace{0.2cm}\\
$2k\!+\!r$}
&
\CPTCTCRF[A][B][C][A][B][C]
&
$1$
&
$1$
&
$1$
&
$($-$1)^r$
&
$($-$1)^r$
&
$1$
&$($-$1)^r$\\
\hline
\hline
$\Io\!\times_t\!\Io(a)$
&
\CPTCTC[A][B][B][A][C][C]
&
$$
&
\CPTCTC[A][B][B][A][C][C]
&
$1$
&
$1$
&
$1$
&
$1$
&
$1$
&
$1$
&$1$\\
\hline
$\Io \tilde{\times}_t\Io(a)$
&
\CPTCTCF[A][B][B][A][C][C]
&
$$
&
\CPTCTCF[A][B][B][A][C][C]
&
$1$
&
$($-$1)^t$
&
$1$
&
$($-$1)^t$
&
$1$
&
$1$
&$1$\\
\hline
\hline
$\Io\!\times_t\!\Io(b)$
&
\CPTCTC[A][B][C][A][B][C]~$\dlu\in K^{\times 2}$
&
$2t\!+\!r$
&
\CPTCTCR[A][B][C][A][B][C]
&
$1$
&
$1$
&
$1$
&
$1$
&
$1$
&
$1$
&$1$\\
\hline
$\Io\!\times_t\!\Io(c)$
&
\CPTCTC[A][B][C][A][B][C]~$\dlu\notin K^{\times 2}$
&
$2t\!+\!r$
&
\CPTCTCRF[A][B][C][A][B][C]
&
$1$
&
$1$
&
$1$
&
$($-$1)^r$
&
$($-$1)^{r}$
&
$1$
&$($-$1)^{r}$\\
\hline
$\Io \tilde{\times}_t\Io(b)$
&
\CPTCTCF[A][B][C][A][B][C]~$\dlu\in K^{\times 2}$
&
$2t\!+\!r$
&
\CPTCTCR[A][B][C][A][B][C]
&
$1$
&
$($-$1)^t$
&
$1$
&
$1$
&
$($-$1)^{t}$
&
$1$
&$($-$1)^{t}$\\
\hline
$\Io \tilde{\times}_t\Io(c)$
&
\CPTCTCF[A][B][C][A][B][C]~$\dlu\notin K^{\times 2}$
&
$2t\!+\!r$
&
\CPTCTCRF[A][B][C][A][B][C]
&
$1$
&
$($-$1)^t$
&
$1$
&
$($-$1)^r$
&
$($-$1)^{t+r}$
&
$1$
&$($-$1)^{t+r}$\\
\hline
\hline
\makecell[l]{ $\Io_{n}^+(a)$\\
$\Io\!\times_t\II_{n}^+(a)$}
&
\makecell[l]{\CPONP[A][A][B][B][C][C]\\
\CPTCONP[A][A][B][B][C][C]}
&
$r $
&
\makecell[l]{\CPTCONNTRZEROP[S][T][R][S][T][R]\\
\CPTCONNTTZEROP[S][T][T][R][R][S]}
&
$n$
&
$1$
&
$2n$
&
$1$
&
-$1$
&
-$1$
&$1$\\
\hline
\makecell[l]{ $\Io_{n}^-(a)$\\
$\Io\!\times_t\II_{n}^-(a)$}
&
\makecell[l]{\CPONM[A][A][B][B][C][C]\\
\CPTCONM[A][A][B][B][C][C]}
&
$r$
&
\makecell[l]{\CPTCONNTRZEROM[S][T][R][S][T][R]\\
\CPTCONNTTZEROM[S][T][T][R][R][S]}
&
$\widetilde{n}$
&
$1$
&
$2$
&
$1$
&
$($-$1)^{n}$
&
$1$
&
$($-$1)^{n}$\\
\hline
\hline
$\Io_{n}^+(b)$
&
\CPONP[B][B][A][A][C][C]
&
$r$
&
\CPTCONNTRZEROP[R][T][S][R][T][S]
&
$n$
&
$1$
&
$2n$
&
$1$
&
-$1$
&
-$1$
&$1$\\
\hline
$\Io_{n}^-(b)$
&
\CPONM[B][B][A][A][C][C]
&
$r$
&
\CPTCONNTRZEROM[R][T][S][R][T][S]
&
$\widetilde{n}$
&
$1$
&
$2$
&
$1$
&
$($-$1)^{n}$
&
$1$
&
$($-$1)^{n}$\\
\hline
\hline
$\Io_{2n}^+(c)$
&
\CPONPN[B][C][A][A][B][C]
&
$$
&
\CPONP[R][R][S][S][T][T]
&
$2n$
&
$1$
&
$n$
&
$1$
&
-$1$
&
-$1$
&$1$\\
\hline
$\Io_{2n}^-(c)$
&
\CPONMN[B][C][A][A][B][C]
&
$$
&
\CPONM[R][R][S][S][T][T]
&
$2$
&
$1$
&
$\widetilde{n}$
&
$1$
&
$($-$1)^{n}$
&
$1$
&
$($-$1)^{n}$\\
\hline
\hline
$\Io_{2n}^+(d)$
&
\CPONPN[A][B][A][B][C][C]
&
$$
&
\CPONP[T][T][S][S][R][R]
&
$2n$
&
$1$
&
$n$
&
$1$
&
-$1$
&
-$1$
&$1$\\
\hline
$\Io_{2n}^-(d)$
&
\CPONMN[A][B][A][B][C][C]
&
$$
&
\CPONM[T][T][S][S][R][R]
&
$2$
&
$1$
&
$\widetilde{n}$
&
$1$
&
$($-$1)^{n}$
&
$1$
&
$($-$1)^{n}$\\

\hline
\hline
\makecell[l]{ $\II_{n,m}^{+,+}(a)$\\
$\II_{n}^+\!\times_t\II_{m}^+(a)$}
&
\makecell[l]{ \CPTNPP[B][B][C][C][A][A]\\
\CPTCTNPP[B][B][A][C][C][A]}
&
$r$
&
\makecell[l]{\CPTCTNRPP[T][R][S][S][R][T]\\
\CPTCTTNPP[R][S][S][R][T][T]}
&
$nm$
&
$1$
&
$4nm$
&
$1$
&
$1$
&
$1$
&$1$\\
\hline
\makecell[l]{$\II_{n,m}^{-,+}(a)$\\
$\II_{n}^-\!\times_t\II_{m}^+(a)$}
&
\makecell[l]{ \CPTNMP[B][B][C][C][A][A]\\
\CPTCTNMP[B][B][A][C][C][A]}
&
$r$
&
\makecell[l]{\CPTCTNRMP[T][R][S][S][R][T]\\
\CPTCTTNMP[R][S][S][R][T][T]}
&
$\widetilde{n}m$
&
$1$
&
$4m$
&
$1$
&
$($-$1)^{n+1}$
&
-$1$
&$($-$1)^{n}$\\
\hline
\makecell[l]{$\II_{n,m}^{+,-}(a)$\\
$\II_{n}^+\!\times_t\II_{m}^-(a)$}
&
\makecell[l]{ \CPTNPM[B][B][C][C][A][A]\\
\CPTCTNPM[B][B][A][C][C][A]}
&
$r$
&
\makecell[l]{\CPTCTNRPM[T][R][S][S][R][T]\\
\CPTCTTNPM[R][S][S][R][T][T]}
&
$n\widetilde{m}$
&
$1$
&
$4n$
&
$1$
&
$($-$1)^{m+1}$
&
-$1$
&$($-$1)^{m}$\\
\hline
\makecell[l]{ $\II_{n,m}^{-,-}(a)$\\
$\II_{n}^-\!\times_t\II_{m}^-(a)$}
&
\makecell[l]{ \CPTNMM[B][B][C][C][A][A]\\
\CPTCTNMM[B][B][A][C][C][A]}
&
$r$
&
\makecell[l]{\CPTCTNRMM[T][R][S][S][R][T]\\
\CPTCTTNMM[R][S][S][R][T][T]}
&
$\widetilde{n}\widetilde{m}$
&
$1$
&
$4$
&
$1$
&
$($-$1)^{n+m}$
&
$1$
&$($-$1)^{n+m}$\\
\hline
\makecell[l]{ $\II_{n\sim n}^{+}(a)$\\
$\II_{n}^+\!\tilde{\times}_t\II_{n}(a)$}
&
\makecell[l]{ \CPTNPPF[B][B][C][C][A][A]\\
\CPTCTNPPF[B][B][A][C][C][A]}
&
$r$
&
\makecell[l]{\CPTCTNRPPF[T][R][S][S][R][T]\\
\CPTCTTNPPF[R][S][S][R][T][T]}
&
$n$
&
\makecell[c]{\bigskip $1$\\
\bigskip
$($-$1)^{t}$}
&
$2n$
&
\makecell[c]{\bigskip $($-$1)^{r}$\\
\bigskip
$($-$1)^{t}$}
&
\makecell[c]{\bigskip $($-$1)^{r+1}$\\
\bigskip
-$1$}
&
-$1$
&\makecell[c]{\bigskip $($-$1)^{r}$\\
\bigskip
$1$}\\
\hline
\makecell[l]{ $\II_{n\sim n}^{-}(a)$\\
$\II_{n}^-\!\tilde{\times}_t\II_{n}(a)$}
&
\makecell[l]{ \CPTNPMF[B][B][C][C][A][A]\\
\CPTCTNPMF[B][B][A][C][C][A]}
&
$r$
&
\makecell[l]{\CPTCTNRPMF[T][R][S][S][R][T]\\
\CPTCTTNPMF[R][S][S][R][T][T]}
&
$\widetilde{n}$
&
\makecell[c]{\bigskip $1$\\
\bigskip
$($-$1)^{t}$}
&
$2$
&
\makecell[c]{\bigskip $($-$1)^{r}$\\
\bigskip
$($-$1)^{t}$}
&
\makecell[c]{\bigskip $($-$1)^{n+r}$\\
\bigskip
$($-$1)^{n}$}
&
$1$
&\makecell[c]{\bigskip $($-$1)^{n+r}$\\
\bigskip
$($-$1)^{n}$}\\
\hline
\end{tabular}}
\hbox{\footnotesize{Notation: $\tilde x=2$ if $2|x$ and $\tilde x=1$ if $2\nmid x$.}}
\end{table}
\begin{table}[H]
\hspace*{-3.1cm}
\resizebox{1.4\linewidth}{!}{
\begin{tabular}{|l|l|c||l|c|c|c|c|c|c|c| }
\hline
Case &$C$ & $v(\Delta/c)$&$\widehat{C}$ &  $c_{J/K}$& $\mu_{C/K}$&$c_{\widehat{J}/K}$&$\mu_{\widehat{C}/K}$ & $\lambda_{C/K}$ & $w_{C/K}$ & $E_{C/K}$\\
\hline
$\II_{n,m}^{+,+}(b)$
&
\makecell[l]{\CPTNPP[B][C][B][C][A][A]~$n <m$\\
\CPTNPP[B][C][B][C][A][A]~$n=m$}
&
$\frac n2\!+r$
&
\makecell[l]{\CPUNNOZEROP[S][T][S][T][R][R]\\
\CPTCTNNRPP[S][T][R][S][T][R]}
&
$nm$
&
$1$
&
$4nm$
&
$1$
&
$1$
&
$1$
&$1$\\
\hline
$\II_{n,m}^{-,+}(b)$
&
\makecell[l]{\CPTNMP[B][C][B][C][A][A]~$n<m$\\
\CPTNMP[B][C][B][C][A][A]~$n=m$}
&
$\frac n2\!+r$
&
\makecell[l]{\CPUNNOZEROPF[S][T][S][T][R][R]\\
\CPTCTNRPPF[S][T][R][S][T][R]}
&
$\widetilde{n}m$
&
$1$
&
$m$
&
$($-$1)^{r}$
&
$($-$1)^{n+1+r}$
&
-$1$
&$($-$1)^{n+r}$\\
\hline
$\II_{n,m}^{+,-}(b)$
&
\makecell[l]{\CPTNPM[B][C][B][C][A][A]~$n<m$\\
\CPTNPM[B][C][B][C][A][A]~$n=m$}
&
$\frac n2\!+r$
&
\makecell[l]{\CPUNNOZEROMF[S][T][S][T][R][R]\\
\CPTCTNRPPF[S][T][R][S][T][R]}
&
$n\widetilde{m}$
&
$1$
&
$n$
&
$($-$1)^{\frac{m-n}{2}+r}$
&
$($-$1)^{\frac{m+n}{2}+r+1}$
&
-$1$
&$($-$1)^{\frac{n+m}{2}+r}$\\
\hline
$\II_{n,m}^{-,-}(b)$
&
\makecell[l]{\CPTNMM[B][C][B][C][A][A]~$n<m$\\
\CPTNMM[B][C][B][C][A][A]~$n=m$}
&
$\frac n2\!+r$
&
\makecell[l]{\CPUNNOZEROM[S][T][S][T][R][R]\\
\CPTCTNNRMM[S][T][R][S][T][R]}
&
$\widetilde{n}\widetilde{m}$
&
$1$
&
$\widetilde{(\frac{nm}{D})}\widetilde{D}$
&
$($-$1)^{n\frac{m-n}{2}}$
&
$($-$1)^{\frac{m-n}{2}}$
&
$1$
&$($-$1)^{\frac{m-n}{2}}$\\
\hline
$\II_{n\sim n}^{+}(b)$
&
\CPTNPPF[B][C][B][C][A][A]
&
$\frac n2\!+r$
&
\CPTCTNRPM[S][T][R][S][T][R]
&
$n$
&
$1$
&
$n\widetilde{n}$
&
$1$
&
$($-$1)^{n+1}$
&
-$1$
&$($-$1)^{n}$\\
\hline
$\II_{n\sim n}^{-}(b)$
&
\CPTNPNN[B][C][B][C][A][A]
&
$\frac n2\!+r$
&
\CPTCTNRPMF[S][T][R][S][T][R]
&
$\widetilde{n}$
&
$1$
&
$\widetilde{n}$
&
$($-$1)^{r}$
&
$($-$1)^{r}$
&
$1$
&$($-$1)^{r}$\\
\hline
\hline
$\UU_{n,m,l}^+(a)$
&
\CPUP[B][B][C][C][A][A]
&
$$
&
\CPUNMLZEROP[T][R][R][S][S][T]
&
$N$
&
$1$
&
$4N$
&
$1$
&
$1$
&
$1$
&$1$\\
\hline
$\UU_{n,m,l}^-(a)$
&
\CPUM[B][B][C][C][A][A]
&
$$
&
\CPUNMLZEROM[T][R][R][S][S][T]
&
$\widetilde{(\frac{N}{M})} \widetilde{M}$
&
$($-$1)^{nml}$
&
$4$
&
$1$
&
$($-$1)^{n+m+l}$
&
$1$
&$($-$1)^{n+m+l}$\\
\hline
$\UU_{n\sim n,l}^+(a)$
&
\CPUPF[B][B][C][C][A][A]
&
$$
&
\CPUNMLZEROPF[T][R][R][S][S][T]
&
$n\!+\!2l$
&
$1$
&
$2n\!+\!4l$
&
$1$
&
-$1$
&
-$1$
&$1$\\
\hline
$\UU_{n\sim n,l}^-(a)$
&
\CPUMF[B][B][C][C][A][A]
&
$$
&
\CPUNMLZEROMF[T][R][R][S][S][T]
&
$n$
&
$($-$1)^{l}$
&
$2n$
&
$1$
&
$($-$1)^{l+1}$
&
-$1$
&$($-$1)^{l}$\\
\hline
\end{tabular}}
\hbox{\footnotesize{Notation: $\tilde x=2$ if $2|x$ and $\tilde x=1$ if $2\nmid x$, $D=gcd(n,\frac{m\!-\!n}{2})$, $N\! =\! nm\!+\!nl\!+\!ml$, $M\!=\!gcd(n,m,l)$.}}
\end{table}

\end{theorem}

\begin{remark}
In the cases I$_{n,m}^{\epsilon,\delta}(b)$, the semistability criterion (Theorem \ref{th:ss}) and the C2D4 structure on $C$ ensure that $n \equiv m 
\mod 2$. Indeed $n$ is odd if and only if inertia permutes the roots in the corresponding twin. The C2D4 structure then forces inertia to permute the roots in the twin of depth $m/2$.
\end{remark}

\section{Odd places: change of invariants under isogeny}\label{s:oddisogeny}

In this section we prove the claim of Theorem \ref{th:localconjectureoddprime} regarding Tamagawa numbers and deficiency and the cluster picture of $\widehat{C}$ when $K$ has odd residue characteristic:

\begin{theorem}\label{thm:CtoChat}
Let $K/\Q_p$ be a finite extension for an odd prime $p$. Let $C/K$ be a semistable C2D4 curve with $\ell_1, \ell_2, \ell_3, \Delta \ne 0$ whose cluster picture with Frobenius action is one of the cases in the table of Theorem \ref{th:localconjectureoddprime}. Let $J$ and $\widehat{J}$ be the Jacobians of $C$ and $\widehat{C}$.
Then 
\begin{enumerate}
\item $c_{J/K}$, $\mu_{C/K}$, $c_{\widehat{J}/K}$, $\mu_{\widehat{C}/K}$, $\lambda_{C/K}$ and $w_{C/K}$ are as given in the table.
\item If the residue field has size $|k|\ge23$ then $\widehat{C}$ admits a model whose cluster picture with Frobenius action is given in the $\widehat{C}$ column of the table.
\end{enumerate}
\end{theorem}

\begin{proof}
(2)$\implies$(1).
The formulae for $c_{J/K}, \mu_{C/K}$ and $w_{C/K}$ follow directly from Theorem~\ref{th:genus2bible}.
To determine $c_{\widehat{J}/K}$ and $\mu_{\widehat{J}/K}$ we may first pass to an unramified extension of sufficiently large degree so that $|k|\ge23$ (Lemma \ref{lem:odddegext}). As these invariants are independent of the choice of model we can change the model of $\widehat{C}$ using (2) to one with the specified cluster picture; the values for $c_{\widehat{J}/K}$ and $\mu_{\widehat{C}/K}$ then follow from Theorem \ref{th:genus2bible}. By Lemma \ref{lem:kercoker}, $\lambda_{J/K}=\mu_{C/K}\mu_{\widehat{C}/K} (-1)^{\ord_2(c_{J/K}/c_{\widehat{J}/K})}$, which gives the required values for $\lambda$.

(2)
First note that if $C'$ is a different model for $C$ obtained by a M\"obius transformation on the $x$-coordinate (as in Definition \ref{def:Cm}), there is an isomorphism between $\Jac\widehat{C}$ and $\Jac \widehat{C}'$ that preserves the kernel of the corresponding isogeny. So $\Jac\widehat{C}$ and $\Jac \widehat{C}'$ are isomorphic as abelian varieties with a principal polarisation, and hence $\widehat{C}'$ is isomorphic to $\widehat{C}$ by Torelli's theorem (see \cite{MilneJV} Cor.\ 12.2). 
We may therefore change the model of $C$ to ensure that it is centered and balanced (Theorem \ref{thm:mobiusbalance}) and that it satisfies the conclusions of Theorem \ref{thm:mobiusoverbalance}.
This change of model does not change whether $\hat{\delta}_1 \in K^{\times 2}$ (Lemma \ref{lem:d1hsquare}) or the definition of $r$ (Lemma \ref{lem:RebalanceDelta} for cases 2(a--f)).
In particular cases 2(b,c,e,f) will follow from cases 2(a,d). 
Note also that the cluster picture of $\widehat{C}$ depends on the choice of Richelot isogeny on $C$, but not on the particular choice of C2D4 structure, so that cases $\Io_{n}^{\epsilon}$(b,d) will follow from cases $\Io_{n}^{\epsilon}$(a,c).

By changing the model it will thus suffice to establish the result for the following list of cases with the given simplifying hypotheses granted by Theorem \ref{thm:mobiusoverbalance}; here $\epsilon, \delta=\pm$ are independent signs. These are proved in the sections indicated:
\begin{itemize}[leftmargin=*]
\item  2(a,d), $\Io\!\times_t\!\Io$(a) and $\Io \tilde{\times}_t\Io$(a) with $v(\ell_i)=0$ for $i=1,2,3,$ (\S \ref{ss:toric0}), 
\item  $\Io\!\times_t\!\Io$(b,c) and $\Io \tilde{\times}_t\Io$(b,c) with $v(\ell_i)=t$ for $i=1,2,3,$ (\S \ref{ss:toric0}), 
\item  $\Io_{n}^{\epsilon}$(a,c) and $\Io \!\times_t\!\II_{n}^{\epsilon}$(a) with $v(\ell_i)=0$ for $i=1,2,3,$  (\S \ref{ss:toric1}),
\item  $\II_{n,m}^{\epsilon,\delta}$(a), $\II_{n\sim n}^{\epsilon}$(a), $\II_{n}^{\epsilon}\!\times_t$\II$_{m}^{\delta}$(a), $\II_{n}^{\epsilon}\!\tilde{\times}_t$\II$_{n}$(a), $\UU_{n,m,l}^{\epsilon}$(a), $\UU_{n\sim n,l}^{\epsilon}$(a) with $v(\ell_i)\!=\!0$  for $i\!=\!1,2,3,$ (\S \ref{ss:toric2}),
\item  $\II_{n,m}^{\epsilon,\delta}$(b) and $\II_{n\sim n}^{\epsilon}$(b) with $v(\ell_1)=\frac{n}{2}$, $v(\ell_2)=v(\ell_3)=0$  (\S \ref{ss:toric2}).
\end{itemize}
\vskip -.3cm
\end{proof}
We will use without further mention that in all of the above cases the roots of $C$ and of $\widehat{C}$ are integral, that is $v(\alpha_i), v(\beta_i), v(\A_i), v(\B_i)\ge 0$. This follows from Lemma \ref{lem:integralroots} and, in cases  $\Io\!\times_t\!\Io$(b,c), $\Io \tilde{\times}_t\Io$(b,c), $\II_{n,m}^{\epsilon,\delta}$(b)  and $\II_{n\sim n}^{\epsilon}$(b), from the explicit formula in Lemma \ref{lem:richelotroots} below.

\begin{notation}
Throughout this section,
for drawing cluster pictures we will use the convention as in Theorem \ref{th:localconjectureoddprime}, that a cluster (other than $\cR$) with an index 0 means that the roots in it do not form a cluster. For example when $a=0$ the cluster picture \CPTCTCABZERO[R][S][S][R][T][T] means  \CPOCBZERO[R][S][S][R][T][T].
\end{notation}

\subsection{Preliminary results}

To control the cluster picture of $\widehat{C}$ we will extensively use the following observations.

\begin{lemma}\label{lem:richelotroots}
For a C2D4 curve with $\ell_1, \ell_2, \ell_3\neq 0$, the roots of the Richelot dual curve are
$$
 \A_1, \B_1 = \frac{1}{\ell_1}\Big(\alpha_2\beta_2 - \alpha_3\beta_3 \pm \sqrt{(\alpha_2 - \alpha_3)(\alpha_2 - \beta_3)(\beta_2- \alpha_3) (\beta_2 - \beta_3) }\Big),
$$
$$
 \A_2, \B_2 = \frac{1}{\ell_2}\Big(\alpha_3\beta_3 - \alpha_1\beta_1 \pm \sqrt{(\alpha_3 - \alpha_1)(\alpha_3 - \beta_1)(\beta_3- \alpha_1) (\beta_3 - \beta_1) }\Big),
$$
$$
 \A_3, \B_3 = \frac{1}{\ell_3}\Big(\alpha_2\beta_2 - \alpha_1\beta_1 \pm \sqrt{(\alpha_2 - \alpha_1)(\alpha_2 - \beta_1)(\beta_2- \alpha_1) (\beta_2 - \beta_1) }\Big).
$$
\end{lemma}
\begin{proof}
This follows by solving the defining quadratic polynomials $\hat{r}(x)$, $\hat{s}(x)$, $\hat{t}(x)$.
\end{proof}

\begin{proposition}\label{pr:Disc}
For a C2D4 curve $C$ with $\Delta, \ell_1, \ell_2, \ell_3\neq 0$,
\begin{enumerate}
\item\label{dlu}
$\>\ell_1^2\>\>(\A_1-\B_1)^2 = \>\>\> 4\>\> (\a_2-\a_3)(\a_2-\b_3)(\b_2-\a_3)(\b_2-\b_3),$
\item\label{dld}
$\>\ell_2^2 \>\> (\A_2-\B_2)^2 = \>\>\>4\>\>(\a_3-\a_1)(\a_3-\b_1)(\b_3-\a_1)(\b_3-\b_1),$
\item\label{dlt}
$\>\ell_3^2\>\>(\A_3-\B_3)^2 = \>\>\>4\>\>(\a_2-\a_1)(\a_2-\b_1)(\b_2-\a_1)(\b_2-\b_1),$
\item\label{dgu}
$\>\quad\>\> (\a_1-\b_1)^2 = \frac{c^2\ell_2^2\ell_3^2}{\DG^2}(\A_2-\A_3)(\A_2-\B_3)(\B_2-\A_3)(\B_2-\B_3),$
\item\label{dgd}
$\>\>\>\quad (\a_2-\b_2)^2=\frac{c^2\ell_1^2 \ell_3^2}{\DG^2}(\A_3-\A_1)(\A_3-\B_1)(\B_3-\A_1)(\B_3-\B_1),$
\item\label{dgt}
$\>\>\>\quad(\a_3-\b_3)^2=\frac{c^2\ell_1^2\ell_2^2}{\DG^2}(\A_1-\A_2)(\A_1-\B_2)(\B_1-\A_2)(\B_1-\B_2),$
\item\label{deltadelta}
$\Delta(\widehat{C})=\frac{2}{c^2}\Delta(C)$.
\end{enumerate}
\end{proposition}
\begin{proof}
These are purely algebraic identities that can be verified using Lemma \ref{lem:richelotroots}.
\end{proof}

\begin{lemma}\label{le:Frob}
Let $K/\Q_p$ be a finite extension for an odd prime $p$.
Suppose $C$ and $C'$ are semistable curves of genus 2 over $K$, whose Jacobians are isogenous.
Then both curves are in the same list of types given below (in the sense of Theorem \ref{th:genus2bible}, possibly with different parameters):
\begin{enumerate}[leftmargin=*]
\item Types $2,\quad\>\>\>\> 1\! \times\! 1$,\quad $1 \tilde\times 1$;
\item Types $1_n^+$, \quad $1\!\times\! {\rm I}_n^+$;
\item Types $1_n^-\>\>$,\quad $1\!\times\! {\rm I}_n^-$;
\item Types ${\rm I}_{n,m}^{+,+}$,\quad ${\rm U}_{n,m,k}^+$,\quad ${\rm I}_n^+\!\times\!{\rm I}_m^+$;
\item Types ${\rm I}_{n,m}^{+,-}, \quad {\rm I}_{n,m}^{-,+}, \quad {\rm I}_{n\sim n}^+, \quad {\rm U}_{n\sim n,k}^+,\quad {\rm U}_{n\sim n,k}^-, \quad {\rm I}_n^+\!\times\!{\rm I}_m^-, \quad {\rm I}_n^-\!\times\!{\rm I}_m^+, \quad {\rm I}_n^{+}\tilde\times{\rm I}_n$;
\item Types ${\rm I}_{n,m}^{-,-},\quad {\rm U}_{n,m,k}^-,\quad {\rm I}_n^-\!\times\!{\rm I}_m^-$;
\item Type ${\rm U}_{n\sim n\sim n}^+$;
\item Type ${\rm U}_{n\sim n\sim n}^-$;
\item Types ${\rm I}_{n\sim n}^-,\quad {\rm I}_n^{-}\tilde\times{\rm I}_n$.
\end{enumerate}
\end{lemma}
\begin{proof}
Since the isogeny induces an isomorphism on Galois representations, the eigenvalues of Frobenius on the toric part of the Galois representation (equivalently, on the homology of the dual graph of the special fibre of the minimal regular model) for the two curves must be the same. By \cite{m2d2} Thm.\ 18.8 the nine lists given correspond to eigenvalues (with multiplicity) being $\emptyset, \{1\}, \{-1\}, \{1, 1\}, \{1, -1\}, \{-1, -1\}, \{\zeta_3, \zeta_3^{-1}\}, \{\zeta_6, \zeta_6^{-1}\}, \{\zeta_4, \zeta_4^{-1}\}$, where $\zeta_n$ denotes a primitive $n$th root of unity.
\end{proof}

\begin{lemma}\label{lem:deltazero}
Let $K/\Q_p$ be a finite extension 
and $C/K$ a C2D4 curve with $v(\alpha_i), v(\beta_i)\ge 0$.
\begin{enumerate}[leftmargin=*]
\item
If the cluster picture of $C$ is either
\CPONNDZERO[S][T][R][R][S][T], 
\CPONNDZERO[R][S][R][S][T][T],
 \CPTNNDZERO[S][T][R][S][R][T]
, \CPUNDZERO[S][T][R][S][R][T]
, 
\CPTCONZERO[R][R][S][S][T][T]
,\CPTCTNNDZERO[S][S][R][T][T][R]
 or \CPTCTCNDZERO[S][S][R][T][T][R]
then $v(\Delta/c)=0$.
\item
If $p$ is odd and the cluster picture of $C$ is \CPUNDZERO[R][R][S][S][T][T]
 then $v(\Delta/c)=0$.
\item
If $C$ has a cluster of positive depth $d$ that contains a root of each colour then $v(\Delta/c)\ge d$.
\end{enumerate}
\end{lemma}

\begin{proof}
(1, 3) Suppose there is a cluster of depth $d$ that contains two roots of different colour: without loss of generality $\alpha_2$ and $\alpha_3$. Then substituting $\alpha_2\equiv\alpha_3\mod \pi_K^d$ into the expression for $\Delta$ gives $\Delta/c \equiv (\b_2-\b_3)(\a_1-\a_2)(\b_1-\a_2) \mod \pi_K^d$. 
For the pictures in (1) each term is a unit, so $\Delta/c$ is a unit. For (3), either $\a_1$ or $\b_1$ is in the same cluster as $\a_2$ and $\a_3$, so the corresponding term is $\equiv 0 \mod \pi_K^d$, and hence so is $\Delta/c$.

(2) In this case $\Delta/c \equiv 2(\alpha_1-\alpha_2)(\alpha_3-\alpha_1)(\alpha_3-\alpha_2)$ in the residue field, which is a unit.
\end{proof}

\begin{lemma}\label{lem:balanceChat}
Let $K/\Q_p$ be a finite extension for an odd prime $p$ with residue field of size $|k|>3$. Let $C/K$ be a semistable C2D4 curve whose cluster picture has two clusters $\s, \s'$ of size 3 with relative depth $\delta_\s=2r-z$ and $\delta_{\s'}=z$ for some integers $0\le z\le r$; we allow for the case $z\!=\!0$ when $\s'$ is not a cluster and $\s$ is a cluster that is not contained in a cluster of size 4 or 5.
Then $C$ admits another model with an identical cluster picture except for which $\delta_\s=\delta_{\s'}=r$ (and all colours, clusters, signs, Frobenius action on proper clusters, other relative depths and the depth of the full set of roots the same as for $C$).
\end{lemma}

\begin{proof}
This is essentially \cite{m2d2} Prop. 14.6(4). Write $\cR$ for the set of roots for $C$.

If Galois swaps the two clusters then they necessarily have the same depth, so there is nothing to prove. We may thus assume that $\s$ is Galois stable.

Pick $z_\s\in K$ such that $v(r-z_\s)\ge d_\s$ for $r\in \s$, and similarly $z_{\s'}$ for $\s'$. (This exists by \cite{m2d2} Lemma B.1 and Thm.\ 2.12, which ensures that $K(\cR)/K$ is tame; if $p\neq 3$ one may simply take $z_\s$ to be the average of the roots in $\s$.)
Since $z_\s, z_{\s'}\in K$, it follows that $v(z_\s-z_{\s'})\in \Z$ and hence $d_\cR\in \Z$. Applying a M\"obius transformation of the form $x\mapsto \pi_K^n x$, we may assume that $d_\cR=0$.  Applying a further M\"obius transformation of the form $x\mapsto x-z_\s+u\pi_K^{r-z}$ for a suitable unit $u$, we may assume that $v(r)=r-z$ for roots $r\in\s$ and $v(r)=0$ for $r\in\s'$. 

Now observe that for roots $r, r'$,
$$
 v\Bigl(\frac{\pi_K^n}{r} - \frac{\pi_K^n}{r'}\Bigr) = v(r-r')+n-v(r)-v(r').
$$
Thus applying the M\"obius transformation $x\mapsto \frac{\pi_K^{r-z}}{x}$ now gives a model for $C$ with the desired cluster picture.  The colouring and Frobenius action on proper clusters are clearly preserved. 
The signs of twins (these are the only clusters with signs here) are preserved by Lemma \ref{lem:mobiussign}.
\end{proof}

\subsection{Proof of Theorem \ref{thm:CtoChat}: Toric dimension 0}
\label{ss:toric0}

Consider cases 2(a,d), $\Io\!\times\!\Io$(a,b,c) and $\Io \tilde{\times}\Io$(a,b,c). 
By Lemma \ref{le:Frob}, $\widehat{C}$ has Type $2$, $1\! \times\! 1$ or $1 \tilde\times 1$, and, in particular, its cluster picture has no clusters of size 2 or 4 (Theorem \ref{th:genus2bible}).

\noindent \underline{\textbf{Cases \It(a,d), $\Io\!\times_t\!\Io$(b,c) and $\Io \tilde{\times}_t\Io$(b,c)}}.
Here $v(\Delta/c)=2t+r$, with $t=0$ for cases \It(a,d).
Proposition \ref{pr:Disc}(\ref{dlu},\ref{dld},\ref{dlt}) give $v(\A_1-\B_1)=v(\A_2-\B_2)=v(\A_3-\B_3)=0$, so the cluster picture of $\widehat{C}$ is either \GRNDZERO[D][D][D][D][D][D], 
\CPOCZERO[R][S][T][R][S][T] 
or \CPTCNDZERO[R][S][T][R][S][T]. 
Proposition \ref{pr:Disc}(\ref{dgu}) gives
$$
  v((\A_2-\A_3)(\A_2-\B_3)(\B_2-\A_3)(\B_2-\B_3))=2r,
$$
so the cluster picture is \CPTCRK[R][S][T][R][S][T] 
for some $0\le k\le 2r$.

An automorphism $\sigma\in\Gal(\bar{K}/K)$ swaps the two clusters of size 3 if and only if it swaps $\A_1$ and $\B_1$. 
Since $\Delta/\ell_1\in K$, Proposition \ref{pr:Disc}(\ref{dlu}) and the definition of $\hat{\delta}_1$ show that this is equivalent to $\sigma(\sqrt{\hat{\delta}_1})=-\sqrt{\hat\delta_1}$.
Lemma \ref{lem:balanceChat} gives the required model for $\widehat{C}$.

\medskip

\noindent \underline{\textbf{Cases} $\Io\!\times_t\!\Io$(a) and $\Io \tilde{\times}_t\Io$(a)}.
By Lemma \ref{lem:deltazero} $\v(\Delta/c)= 0$.
Proposition~\ref{pr:Disc}(\ref{dlu},\ref{dgu}) give 
$$
  v(\A_1-\B_1)=0,\qquad 
  v((\A_2-\A_3)(\A_2-\B_3)(\B_2-\A_3)(\B_2-\B_3))=0,
$$ 
which force the cluster picture of $\widehat{C}$ to be \CPTCTCABZERO[R][S][S][R][T][T], 
for some $a, b\ge 0$.
Proposition~\ref{pr:Disc}(\ref{dld},\ref{dlt}) give $v(\A_2-\B_2)=v(\A_3-\B_3)=t$, so that $a=b=t$.

The expression on the right-hand-side of Proposition \ref{pr:Disc}(\ref{dlu}) is $K$-rational and a perfect square in the residue field. Hence a Galois element swaps the two clusters of $\widehat{C}$ if and only if it swaps $\A_1$ and $\B_1$, if and only if it maps $\ell_1$ to $-\ell_1$, if and only if it swaps the two clusters of $C$.

\subsection{Proof of Theorem \ref{thm:CtoChat}: Toric dimension 1}
\label{ss:toric1}
Consider cases  $\Io_{n}^{\epsilon}$(a), $\Io_{2n}^{\epsilon}$(c) and $\Io \!\times_t\!\II_{n}^{\epsilon}(a)$.
By Lemma \ref{le:Frob} $\widehat{C}$ will have either type $\Io_*^\epsilon$ or $\Io\!\times_*\!\II_{*}^\epsilon$ for some suitable $*$s.
In particular, its cluster picture will have a cluster of size 2 or 4, but will not have two clusters of size 2 (Theorem \ref{th:genus2bible}).

\medskip

\noindent \underline{\textbf{Case} $\Io_{n}^{\epsilon}$(a)}. 
Here $r=\v(\DG/c)\ge0$.
Proposition \ref{pr:Disc}(\ref{dlu},\ref{dld},\ref{dlt}) and the cluster picture of $C$ yield $\v(\A_i-\B_i)^2 = 0$ for $i=1,2,3$. 
Hence the cluster picture of $\widehat{C}$ has no cluster of size $\ge 4$, so that it is either 
 \CPONNDZERO[D][D][D][R][S][T]
, \CPOCONZEROSWITCH[D][D][D][R][S][T]
, \CPOCINONZERO[D][D][D][R][S][T]
 or \CPTCONZERO[D][D][D][R][S][T]
 . Relabelling $\A_i \leftrightarrow \B_i$ if necessary, we may assume that $\A_1, \A_2, \A_3$ are the three leftmost roots in these pictures, so that $v(\A_i-\B_j)=0$ for all $i,j$. Proposition \ref{pr:Disc}(\ref{dgu},\ref{dgd},\ref{dgt}) then give
$$
  \v((\A_2-\A_3)(\B_2-\B_3))=2r+n, \quad
  \v((\A_3-\A_1)(\B_3-\B_1))= \v((\A_1-\A_2)(\B_1-\B_2))=2r.
$$
It follows that the cluster picture of $\widehat{C}$ with Frobenius action is \CPTCONNZKE[R][S][T][R][S][T]
 for some $0\le z\le 2r$.
Lemma \ref{lem:balanceChat} gives the required model for $\widehat{C}$.

\medskip

\noindent \underline{\textbf{Cases} $\Io_{2n}^{\epsilon}$(c)}. 
By Lemma \ref{lem:deltazero} $v(\Delta/c)\!=\!0$, so by Proposition \ref{pr:Disc}, $\v(\A_1-\B_1)\!=\!\frac n2$,
and $v(r_1-r_2)=0$ for any other pair of roots $r_1,r_2$ for $\widehat{C}$. 
The cluster picture with Frobenius action of $\widehat{C}$ is therefore \CPONNEZERO[R][R][S][S][T][T]
.

\medskip

\noindent \underline{\textbf{Case} $\Io \!\times_t\!\II_{n}^{\epsilon}$(a)}.
By Lemma \ref{lem:deltazero} $v(\Delta/c)=0$.
By Propostion \ref{pr:Disc}(\ref{dld},\ref{dgd}),
$$
 v(\A_2-\B_2)=v((\A_3-\A_1)(\A_3-\B_1)(\B_3-\A_1)(\B_3-\B_1))=0.
$$ 
Thus there is no cluster that contains both the sapphire roots, or both a ruby and a turquoise root. In particular, there are no clusters of size 4 or 5 and the cluster picture of $\widehat{C}$ is either 
\CPTCONWYZZERO[D][D][D][R][R][S] 
or \CPTCONWYZZERO[D][D][D][T][T][S] 
for some $w>0$ and $y, z\ge 0$.
Proposition \ref{pr:Disc} also gives
$$
v(\A_1-\B_1)=v(\A_3-\B_3)=t, 
$$ 
$$
 v((\A_2-\A_3)(\A_2-\B_3)(\B_2-\A_3)(\B_2-\B_3))=n+2t, 
 $$
 $$
 v((\A_1-\A_2)(\A_1-\B_2)(\B_1-\A_2)(\B_1-\B_2))\!=\!2t,
$$
so the cluster picture with Frobenius action must be \CPTCONNTTZEROE[S][T][T][R][R][S]. 

\subsection{Proof of Theorem \ref{thm:CtoChat}: Toric dimension 2}
\label{ss:toric2}

Consider cases 
$\II_{n,m}^{\epsilon, \delta}$(a,b), 
$\II_{n\sim n}^{\epsilon}$(a,b), 
$\II_{n}^{\epsilon}\!\times$\II$_{m}^{\delta}$,
$\II_{n}^{\epsilon}\!\tilde{\times}$\II$_{n}$,
$\UU_{n,m,l}^{\epsilon}$ and
$\UU_{n\sim n,l}^{\epsilon}$.
By Lemma \ref{le:Frob}, $\widehat{C}$ will have type $\II_{*,*}$, ${\rm I}_*\!\times\!{\rm I}_*$, $\II_*\tilde\times\II_*$  or ${\rm U}_{***}$, with some subscripts and signs. In particular, its cluster picture will have at least two clusters of size 2 or 4 (Theorem \ref{th:genus2bible}).

\begin{lemma}\label{le:noswapsign}
Suppose $C$ has Type $I_{n,m}^{\epsilon, \delta}$ or $I_n^\epsilon\times I_m^\delta$, and that $\widehat{C}$ has Type $I_{2n,2m}^{\epsilon', \delta'}$ or $I_{2n}^{\epsilon'}\times I_{2m}^{\delta'}$. Then $\epsilon=\epsilon'$ and $\delta=\delta'$.
\end{lemma}

\begin{proof}
Write $C_k^+$ and $C_k^-$ for the cyclic group $C_k$ on which Frobenius acts trivially and by multiplication by -1, respectively.
The N\'eron component group of $\Jac C/K^{nr}$ is $\Phi_C= C_n^\epsilon\times C_m^\delta$, and similarly for $\widehat{C}$ (see Remark \ref{rmk:Inmcomponentgroup}). Passing to a quadratic ramified extension if necessary, we may assume that $n$ and $m$ are even.

The Richelot isogeny and its dual induce Frobenius-equivariant maps $f: \Phi_C\to \Phi_{\widehat{C}}$ and $g: \Phi_{\widehat{C}}\to \Phi_C$ such that both $f\circ g$ and $g\circ f$ are multiplication by 2 maps. As $|\Phi_{\widehat{C}}|=4 |\Phi_C|$ this forces $f$ to be injective and $g$ surjective with kernel $C_2\times C_2$. Hence $\Phi_C \simeq (C_{2n}^{\epsilon'}\times C_{2m}^{\delta'})/ C_2\times C_2 \simeq C_n^{\epsilon'}\times C_m^{\delta'}$. By Remark \ref{rmk:Inmcomponentgroup}, $\epsilon=\epsilon'$ and $\delta=\delta'$.
\end{proof}

\medskip

\noindent \underline{\textbf{Cases} $\II_{n,m}^{\epsilon,\delta}$(a) and $\II_{n\sim n}^{\epsilon}$(a)}.
Here $v(\DG/c) = r \ge 0$. 
Proposition \ref{pr:Disc}(\ref{dlu},\ref{dld},\ref{dlt}) give $v(\A_i-\B_i)=0$ for all $i$. In particular there is no cluster of size $\ge 4$ and no two roots of the same colour lie in one cluster. It follows that the cluster picture of $\widehat{C}$ must be either
\CPTNNDZEROSWITCH[D][D][D][D][D][D],          
\CPOCINTNZERO[D][D][D][D][D][D],   
\CPTCTNNDZERO[D][D][D][D][D][D]   
or \CPUNDZERO[R][S][T][R][S][T].       
The latter cannot in fact occur: otherwise, by Lemma~\ref{lem:deltazero} and Proposition \ref{pr:Disc}(7) we would have 
$0=v(\frac{\Delta(\widehat{C})}{c(\widehat{C})})=v(\frac{\Delta(\widehat{C})}{\ell_1\ell_2\ell_3/\Delta(C)})=v(\frac{\Delta(C)^2}{c^2})=2r$,
at which point Proposition \ref{pr:Disc}(4) forces $v(s-t)=0$ for each sapphire root $s$ and turquoise root $t$, so that there is no sapphire-turquoise twin.

We may now relabel the roots $\A_i\leftrightarrow \B_i$, so that the three leftmost roots in the given pictures are $\A_1, \A_2$ and $\A_3$, in some order, so that $v(\A_i-\B_j)=0$ for all $i, j$. Proposition \ref{pr:Disc}(\ref{dgu},\ref{dgd},\ref{dgt}) give 
$$
v((\A_2-\A_3)(\B_2-\B_3))=2r, \quad v((\A_3-\A_1)(\B_3-\B_1))=2r+n,
$$
$$
v((\A_1-\A_2)(\B_1-\B_2))=2r+m.
$$ 
Hence 
the cluster picture must be \CPTCTNNDRZMN[T][R][S][S][R][T] 
for some integer $0\le z\le 2r$.

If $C$ has type $\II_{n,m}^{\epsilon,\delta}$, Galois preserves $\{\alpha_2, \beta_2\}$ and $\{\alpha_3, \beta_3\}$, so that $\ell_1\in K$ (see Definition \ref{def:invariants}). 
By Proposition \ref{pr:Disc}(\ref{dlu}) 
$(\A_1-\B_1)^2 =\frac{1}{4\ell_1^2}(\alpha_2-\alpha_3)(\alpha_2-\beta_3)(\beta_2-\alpha_3)(\beta_2-\beta_3)$, 
which is Galois-stable, a unit and a perfect square in the residue field, so $\A_1-\B_1 \in K$ and the two possible clusters of size 3 for $\widehat{C}$ are not interchanged by Galois.
It follows from Lemma \ref{le:noswapsign} that $\widehat{C}$ has cluster picture with Frobenius action \CPTCTNNDRZMNED[T][R][S][S][R][T]
.
Lemma \ref{lem:balanceChat} gives the required model for $\widehat{C}$.

Conversely, if $C$ has type $\II_{n\sim n}^{\epsilon}$, then Frobenius swaps  $\{\alpha_2, \beta_2\}$ and $\{\alpha_3, \beta_3\}$, so $\Frob(\ell_1)=-\ell_1$, and hence $\Frob(\A_1-\B_1)=\B_1-\A_1$, as above. Thus $\widehat{C}$ has type $\II_{2n\sim 2n}^{\epsilon}$ or $\II_{2n}^{\epsilon}\tilde{\times}_r\II_{2n}$ by Lemma~\ref{le:Frob}, and cluster picture with Frobenius action 
\CPTCTNNDRMNPE[T][R][S][S][R][T]
.

\medskip

\noindent \underline{\textbf{Cases} $\I_{n,m}^{\epsilon,\delta}$(b) and $\II_{n\sim n}^{\epsilon}$(b)}.
In these cases $v(\Delta/c)=r+\frac{n}2$ and $v(\ell_1)=\frac{n}2$.
Proposition \ref{pr:Disc} gives $v(\A_1-\B_1)=\frac{m-n}4$, $v(\A_i-\B_i)=0$ for $i=2,3$, and 
$$
 v((\A_2-\A_3)(\A_2-\B_3)(\B_2-\A_3)(\B_2-\B_3))=n+2r,
$$
$$
v((\A_3-\A_1)(\A_3-\B_1)(\B_3-\A_1)(\B_3-\B_1)) = 2r,
$$
$$
v((\A_1-\A_2)(\A_1-\B_2)(\B_1-\A_2)(\B_1-\B_2)) = 2r.
$$

First suppose $C$ is of type $\I_{n,m}^{\epsilon,\delta}$ for some $n<m$. 
Without loss of generality $\alpha_2\equiv \alpha_3\mod \pi_K^{m/2}$. This gives $\Delta/c \equiv (\beta_2-\beta_3)(\alpha_1-\alpha_2)(\beta_1-\alpha_2) \mod \pi_K^{m/2}$, so that $\Delta$ has valuation exactly $n/2$ and hence $r=0$. 
The above expressions and the restriction on the type then force the cluster picture of $\widehat{C}$ to be \CPUMNZERO[R][R][S][T][S][T].
Moreover, the sum of the depths of the sapphire-turquoise twins is $n$, and the explicit desription of the roots in Lemma \ref{lem:richelotroots} shows that each depth is at least $n/2$, so that each must be exactly $n/2$. 

By Lemma \ref{le:Frob}, if $\epsilon=\delta$ then $\widehat{C}$ has type $\UU_{\frac{m-n}2,n,n}^{\epsilon}$ and hence cluster picture with Frobenius action
\CPUMNNEZERO[R][R][S][T][S][T]
; and if $\epsilon\neq\delta$ then $\widehat{C}$ either has type
$\UU_{\frac{m-n}2,n\sim n}^{\epsilon}$ or $\UU_{\frac{m-n}2,n\sim n}^{\delta}$. In the latter case, the sign is determined by whether $\frac{\l_1\l_2\l_3}{\Delta} \in K^{\times 2}$ or not (see Definition \ref{no:signandfrob}). Working modulo $\pi_K^{m/2}$ we have
$\ell_1 \equiv \b_2-\b_3, ~\ell_2\equiv \a_2 + \b_3, ~\ell_3\equiv \a_2 + \b_2, $
$$
\Delta/c\equiv (\beta_2-\beta_3)(\alpha_2-\alpha_1)(\alpha_2-\beta_1), \quad \theta_{\{\a_2, \a_3\}}^2\equiv c(\a_2-\a_1)(\a_2-\b_1)(\a_2-\b_2)(\a_2-\b_3).
$$
All of the factors in these expressions are units except for $\b_2-\b_3$, which has valuation exactly $\frac{n}{2}$ (as it is smaller than $\frac{m}{2})$. Dividing through by $\b_2-\b_3$ and then working modulo $\pi_K^{\min(\frac{m-n}{2}, \frac{n}{2})}$ (so that also $\b_2\equiv \b_3$), we get
$$
 \frac{\Delta}{c\ell_1\ell_2\ell_3} \equiv \frac{(\a_2-\a_1)(\a_2-\b_1)}{c(\a_2+\b_2)^2} \equiv \frac{\theta_{\{\a_2, \a_3\}}^2 }{c^2(\a_2-\b_2)^2(\a_2+\b_2)^2}.
$$
Hence $\frac{\Delta}{c\ell_1\ell_2\ell_3}=\square\cdot \theta_{\{\a_2, \a_3\}}^2$ in $K^{\times 2}$.
Thus the sign is the same as of the cluster $\{\a_2, \a_3\}$ for $C$, and $\widehat{C}$ has cluster picture with Frobenius action \CPUMNNEZDROF[R][R][S][T][S][T]
.

Now suppose that the type is $\I_{n,n}^{\epsilon,\delta}$ or $\II_{n\sim n}^{\epsilon}$.
As $v(\A_i-\B_i)=0$, the same argument as for case $\I_{n,m}^{\epsilon,\delta}$(a) shows that the cluster picture of $\widehat{C}$ is
\CPTCTNNDZERO[D][D][D][D][D][D], 
\CPOCINTNZERO[D][D][D][D][D][D], 
or \CPTNNDZEROSWITCH[D][D][D][D][D][D]
,
in which the three leftmost roots have different colours, and similarly for the rightmost three (the fourth picture again cannot occur as it would yield $0=2v(\Delta)-v(\ell_1)= 2r +\frac{n}{2}>0$).
The average valuation of \raise3pt\hbox{\OR[S]} - \raise3pt\hbox{\OR[T]} (that is, of a sapphire minus a turquoise root) is higher than that of \raise3pt\hbox{\OR[R]} - \raise3pt\hbox{\OR[T]}, so at least one of the twins must consist of a sapphire and a turquoise root. The average valuation of \raise3pt\hbox{\OR[R]} - \raise3pt\hbox{\OR[T]} is the same as that of \raise3pt\hbox{\OR[R]} - \raise3pt\hbox{\OR[S]}, so both twins must be sapphire-turquoise, and the cluster picture is
 \CPTCTNUVRK[S][T][R][S][T][R] 
for some $0\le k \le 2r$ and some $u, v>0$.
For the two twins, from the valuations we know that $u+v=n$. By Theorem \ref{th:genus2bible}, the Tamagawa numbers satisfy $c_{\Jac{\widehat{C}}/K^{nr}}=uv$ and  $c_{\Jac C/K^{nr}}=n^2/4$, so that as the Richelot isogeny has degree 4 we necessarily have $uv=2^j n^2$ for some $j\in \Z$ (Lemma \ref{lem:kercoker}). 
A little exercise in elementary number theory shows that, as $\frac{u}{n}+\frac{v}{n}=1$ and $\frac{u}n\cdot\frac{v}n=2^j$, it follows that $u=v=\frac{n}2.$ 

Frobenius will swap the two twins (and hence the two clusters of size 3) if and only if it swaps the residues of $\A_2$ and $\B_2$. Working in the residue field, Proposition \ref{pr:Disc}(2), Definition \ref{no:signandfrob} and the cluster picture of $C$ tell us that
$$
 (\A_2-\B_2)^2\equiv \frac{(\a_3\!-\!\a_1)(\a_3\!-\!\b_1)(\b_3\!-\!\a_1)(\b_3\!-\!\b_1)}{(\a_3\!+\!\b_3)^2} \equiv 
 $$
 $$
 \equiv \frac{\theta_{\{\a_2,\b_2\}}^2\theta_{\{\a_3,\b_3\}}^2}{c^2(\a_3\!+\!\b_3)^2(\a_3\!-\!\b_3)^4} \equiv \square \cdot \theta_{\{\a_2,\b_2\}}^2\theta_{\{\a_3,\b_3\}}^2.
$$
Thus Frobenius preserves the two twins of $\widehat{C}$ when $C$ has type $\I_{n,n}^{+,+}$, $\I_{n,n}^{-,-}$ or $\I_{n\sim n}^+$, and swaps them for types
$\I_{n,n}^{+,-}$ and $\I_{n\sim n}^-$. Together with Lemmata \ref{le:Frob} and \ref{lem:balanceChat} this gives a model for $\widehat{C}$ with the desired cluster picture.

\medskip

\noindent \underline{\textbf{Cases} \UU$_{n,m,l}^{\epsilon}$(a) and \UU$_{n\sim n,l}^{\epsilon}$(a)}.
By Lemma \ref{lem:deltazero}(2) $v(\Delta/c)=0$.
Also by Proposition \ref{pr:Disc}(\ref{dlu},\ref{dld},\ref{dlt})  $v(\A_i-\B_i)=0$ for all $i$, so, in particular, there is no cluster of size $\ge 4$ and a cluster of size 3 cannot contain two roots of the same colour. 
By Lemma \ref{lem:deltazero}(3) and Proposition \ref{pr:Disc}(\ref{deltadelta}) it follows that there is no cluster of size 3 either, since $v(\frac{\Delta(\widehat{C})}{c(\widehat{C})})=v(\frac{\Delta(\widehat{C})}{\ell_1\ell_2\ell_3/\Delta(C)})=v(\frac{\Delta(C)^2}{c^2})=0$.
By Proposition \ref{pr:Disc}(\ref{dgu},\ref{dgd},\ref{dgt})
$$
v((\A_2-\A_3)(\A_2-\B_3)(\B_2-\A_3)(\B_2-\B_3))=n, 
$$
$$
 v((\A_3-\A_1)(\A_3-\B_1)(\B_3-\A_1)(\B_3-\B_1))\!=\!m,
$$
$$
v((\A_1-\A_2)(\A_1-\B_2)(\B_1-\A_2)(\B_1-\B_2))\!=\!l.
$$
It follows that the cluster picture of $\widehat{C}$ is \CPUNMLZERO[T][R][R][S][S][T]. 

The only twins that can be swapped by Frobenius are the ones containing the ruby roots. This happens if and only if $\Frob(\A_1-\B_1)=\B_1-\A_1$, which, by Proposition \ref{pr:Disc}(\ref{dlu}), is equivalent to $\Frob(\ell_1)=-\ell_1$.
It follows that the two twins for $\widehat{C}$ are swapped when $C$ is in case $\UU_{n\sim n,l}^{\epsilon}(a)$, and not swapped
in case \UU$_{n,m,l}^{\epsilon}(a)$. 
The sign for $\widehat{C}$ is determined by whether $\frac{\l_1\l_2\l_3}{\Delta}$ is a square (see Definition \ref{no:signandfrob}).
As $\a_1\equiv 0, \a_2\equiv \b_2$ and $\a_3\equiv \b_3$, we have that $\frac{\l_1\l_2\l_3}{\Delta} \equiv \frac{8(\a_2-\a_3)\a_2\a_3}{c(2\a_2^2\a_3-2\a_3^2\a_2)}\equiv \frac{4}{c}$ in the residue field, so this is the same as the sign for $C$.

\medskip

\noindent \underline{\textbf{Cases} $\II_{n}^{\epsilon}\!\times_t$\II$_{m}^{\delta}$(a) and $\II_{n}^{\epsilon}\!\tilde{\times}_t$\II$_{n}$(a)}.
By Lemma \ref{lem:deltazero} $v(\Delta/c) = 0$.
Also Proposition \ref{pr:Disc}(\ref{dlu},\ref{dgu}) give $v(\A_1-\B_1)=v((\A_2-\A_3)(\A_2-\B_3)(\B_2-\A_3)(\B_2-\B_3))=0$, so, in particular, there can be no cluster of size $\ge 4$.
Moreover, Proposition \ref{pr:Disc}(\ref{dld},\ref{dlt}) give $v(\A_2-\B_2)=v(\A_3-\B_3)=t$, so the two sapphire roots lie in a non-trivial cluster, as do the two turquoise roots, and hence $\widehat{C}$ cannot have three twins. 
Thus the possible cluster pictures for $\widehat{C}$ are
\CPTCTNNDZERO[D][D][D][D][D][D]
, 
 \CPOCINTNZERO[D][D][D][T][T][R]
 ,
 \CPOCINTNZEROSWITCH[S][S][R][D][D][D]
  and
\CPTNNDZEROSWITCH[S][S][R][T][T][R]
,
where the three leftmost roots are the two sapphire ones and a ruby one, in some order.
Proposition \ref{pr:Disc}(\ref{dgd},\ref{dgt}) give 
$$
  v((\A_3\!-\!\A_1)(\A_3\!-\!\B_1)(\B_3\!-\!\A_1)(\B_3\!-\!\B_1))\!=\!n+2t, 
 $$
 $$ 
  v((\A_1-\A_2)(\A_1-\B_2)(\B_1-\A_2)(\B_1-\B_2))\!=\!m+2t,
$$
which forces the cluster picture of $\widehat {C}$ to be \CPTCTNNMTZERO[R][S][S][R][T][T]. 

The same argument as in cases $\II_{n,m}^{\epsilon, \delta}$  and $\II_{n\sim n}^\epsilon$, shows that for case $\II_{2n}^{*}\!\times_t\II_{2m}^{*}$ the cluster picture with Frobenius action of $\widehat{C}$ is  \CPTCTNNMTZEROED[R][S][S][R][T][T]
, and for case  $\II_{n}^{\epsilon}\!\tilde{\times}_t$\II$_{n}$ it is \CPTCTNNMTZEROEDF[R][S][S][R][T][T]
.

\bigskip

This completes the proof of Theorem \ref{thm:CtoChat}.

\section{Odd places: $\lambda_{C/K}w_{C/K} = E_{C/K}$}\label{s:oddhilbert}

In this section we will complete the proof of Theorem \ref{th:localconjectureoddprime} by justifying the values it gives for~$E_{C/K}$ and showing that $\lambda_{C/K}w_{C/K} = E_{C/K}$. Throughout we will use the division into cases as in Theorem \ref{th:localconjectureoddprime}.

\begin{theorem}\label{th:errortermoddprime}
Let $K/\Q_p$ be a finite extension for an odd prime $p$ and $C/K$ a centered semistable C2D4 curve with $\mathcal{P}, \Delta, \eta_1  \ne 0$. Suppose that the cluster picture with Frobenius action of $C$ is one of the cases in Theorem \ref{th:localconjectureoddprime} and that\begin{itemize}[leftmargin=*]
\item $v(\ell_1)=v(\l_2)=v(\l_3)=v(\eta_2)=v(\eta_3)=0$ in cases $\Io_{n}^{\epsilon}(a,b)$, $\Io_{2n}^{\epsilon}(c,d)$ and  $\Io\times_t\II_{n}^{\epsilon}(a)$,
\item $v(\ell_1)\!=\!v(\l_2)\!=\!v(\l_3)\!=\!0$ in cases $\II_{n,m}^{\epsilon,\delta}(a)$, $\II_{n\sim n}^{\epsilon}(a)$, 
\item $v(\ell_1)=t$ in cases $\Io\!\times_t\! \Io(b,c)$ and $\Io\tilde\times_t \Io(b,c)$, and
\item $v(\l_1) = \frac{n}{2}$, $v(\l_2)=v(\l_3)=v(\eta_2) =v(\eta_3)=0$ in cases $\II_{n,m}^{\epsilon,\delta}(b)$ and $\II_{n\sim n}^{\epsilon}(b)$.
\end{itemize}
Then $E_{C/K}$ is as given in Theorem \ref{th:localconjectureoddprime}, and $E_{C/K}=\lambda_{C/K}w_{C/K}$.
\end{theorem}

\begin{proof}
Combine Lemmata \ref{lem:lw2}, \ref{lem:lw1t1}, \ref{lem:lw12n}, \ref{lem:lwI2nI2m}, \ref{lem:tor0E}, \ref{lem:tor1E} and \ref{lem:tor2E} below. 
\end{proof}
The remainder of this section is devoted to the proof of this result. Throughout the section, $K/\Q_p$ will be a finite extension for an odd prime $p$ and $C/K$ will be a centered semistable C2D4 curve with~$\mathcal{P}, \Delta, \eta_1  \ne 0$. 

\subsection{The value of $\lambda_{C/K}w_{C/K}$}\label{se:lambdaw}

Here we convert  $\lambda_{C/K}w_{C/K}$ into Hilbert symbols. As both $\lambda_{C/K}$ and $w_{C/K}$ are sensitive to the signs of twins, 
we first express these signs (defined via $\theta_{\t}^2$, see Definition \ref{no:signandfrob}) in terms of our standard invariants from Definition \ref{def:invariants}.  

\begin{lemma}\label{LL} Let $K/\Q_p$ be a finite extension for an odd prime $p$ and $C/K$ a centered semistable C2D4 curve with $\cP, \eta_1, \DG \ne 0$. For a twin cluster $\t$, $\tanc[\t]$ 
satisfies the following equalities:\begin{enumerate}[leftmargin=*]
\item\label{1} In case $\Io_{n}^{\epsilon}(a)$ $($\CPONNEZEROT[A][A][B][B][C][C]$)$: $\tanc[\t] = c\xi\sq$.
\item\label{2} In case $\Io_{n}(b)$ $($\CPONNEZEROT[B][B][A][A][C][C]$)$: $\tanc[\t]=c\eta_1\eta_2 \sq$.
\item \label{3} In case $\Io_{2n}(c)$ $($\CPONNZEROT[B][C][A][A][B][C]$)$: $\tanc[\t] =-\frac{\ell_1}{\DG} \eta_1\sq$.
\item \label{4} In case $\Io_{2n}(d)$ $($\CPONNZEROT[A][B][A][B][C][C]$)$: $\tanc[\t] = -2c\xi\dld \eta_2\sq$.
\item \label{5}In case \II$^{\epsilon,\delta}_{n,m}(a)$ $($\CPTNTANG[B][B][C][C][A][A]$)$: $\tanc[\t_2] = 2c\eta_2\sq$ and $\tanc[\t_3] = 2c\eta_3\sq$.
\item \label{6} In case \II$_{n\sim n}^{\epsilon}(a)$ $($\CPTNTANGF[B][B][C][C][A][A]$)$: $\tanc[\t_2]\tanc[\t_3] = \eta_2\eta_3 \sq$.
\item \label{7} In cases \II$^{\epsilon,\delta}_{n,m}(b)$ $($\CPTNTANG[B][C][B][C][A][A]$)$ and \II$_{n\sim n}^{\epsilon}$(b) $($\CPTNTANGF[B][C][B][C][A][A]$)$: 

if $n<m$ then $\tanc[\t_\b] = \frac{\l_1}{\DG} \sq$,

for any $n, m$,  $\tanc[\t_\a]\tanc[\t_\b] = -\eta_1(\dgd\eta_2+\dgt\eta_3)(\dld\eta_3+\dlt\eta_2)\sq$.
\item \label{8} In case $1 \times_t$\II$_{n}(a)$ $($\CPTCONZEROTANG[A][A][B][B][C][C]$)$: $\tanc[\t]= c\xi\sq$.
\item \label{9} In case \II$_{n}^{\epsilon}\times_t$\II$^{\delta}_{m}(a)$ $($\CPTCTNTANG[B][B][A][C][C][A]$)$: $\tanc[\t_2] = 2c\eta_2\sq$ and $\tanc[\t_3] = 2c\eta_3\sq$.
\item \label{10} In case \II$_{n}^{\epsilon}\tilde{\times}_t$\II$_{n}(a)$ $($\CPTCTNTANGF[B][B][A][C][C][A]$)$:  $\tanc[\t_2]\tanc[\t_3] = \eta_2\eta_3 \sq$.
\end{enumerate}
Here in (5), (6), (9), (10), $\t_2 = \{\a_2, \b_2\}$ and $\t_3 = \{\a_3, \b_3\}$, and in (7), $\t_\a = \{\a_2,\a_3\}$ is the left twin and $\t_\b =\{\b_2,\b_3\}$ is the right twin. 

\end{lemma}

\begin{proof}
We will write $k$ for the residue field of $K$,  $x\to\bar{x}$ for the reduction map to $\bar{k}$, and $\equiv$ for equality in the residue field (unless specified otherwise). In each case we will show that, after a suitable scaling to make both sides units, the claimed identities hold over the residue field. The result then follows by Hensel's lemma.

\noindent (1) Here $\a_1\equiv-\a_1\equiv0$, and $\frac{4}{c}\tanc[\t]\equiv 4\a_2\b_2\a_3\b_3\equiv \xi$ are units; also $\xi, \tanc[\t]\in K$. 

\noindent (2) Here $\frac{4}{c}\tanc[\t]\equiv 4(\a_2\!-\!\a_1)(\a_2\!+\!\a_1)(\a_2\!-\!\a_3)(\a_2\!-\!\b_3)\equiv\eta_1\eta_2$ are units, and $\eta_1, \eta_2,\tanc[\t]\in K$.

\noindent (3) Let $\t = \{\a_2, \a_3\}$. Then $\frac{1}{c}\tanc[\t]\equiv (\a_2\!-\!\a_1)(\a_2\!+\!\a_1)(\a_2\!-\!\b_2)(\a_2\!-\!\b_3)$, $\ell_1\equiv \b_2-\b_3$, $\eta_1\equiv(\b_2-\a_2)(\a_2-\b_3)$ and $\Delta/c\equiv (\b_2\!-\!\b_3)(\a_2\!-\!\a_1)(\a_2\!+\!\a_1)$. Hence $\frac{1}{c}\tanc[\t]\equiv-\frac{\ell_1}{\Delta/c}\eta_1(\a_2-\a_1)^2(\a_2+\a_1)^2$ are units, which gives the result as $\tanc[\t],\frac{\ell_1}{\Delta}, \eta_1\in K$ and $(\bar{\a}_2-\bar{\a}_1)(\bar{\a}_2+\bar{\a}_1)\in k^\times$.

\noindent (4) Let $\t=\{\a_1,\a_2\}$. Then  $\frac{1}{c}\tanc[\t]\equiv 2\a_1(\a_1\!-\!\b_2)(\a_1\!-\!\a_3)(\a_1\!-\!\b_3)$, $\hat{\delta}_2\equiv 4(\a_3\!+\!\a_1)(\a_3\!-\!\a_1)(\b_3\!+\!\a_1)(\b_3\!-\!\a_1)$, $\xi\equiv 4\a_1(\a_1\!+\!\b_2)(\a_1\!+\!\a_3)(\a_1\!+\!\b_3)$ and $\eta_2\equiv (\b_2-\a_1)(\b_2+\a_1)$. Hence $\frac{16}{c}\tanc[\t](\a_1\!+\!\b_2)^2(\a_1\!+\!\a_3)^2(\a_1\!+\!\b_3)^2\equiv -2\xi\hat{\delta}_2\eta_2$, which gives the results as $\tanc[\t], \xi, \hat{\delta}_2, \eta_2\in K$ and $(\bar{\a}_1\!+\!\bar{\b}_2)$, $(\bar{\a}_1\!+\!\bar{\a}_3)(\bar{\a}_1\!+\!\bar{\b}_3)\in k^\times$.

\noindent (5, 6) Here  $\frac{4}{c}\tanc[\t_2]\equiv 4(\a_2\!-\!\a_1)(\a_2\!+\!\a_1)(\a_2\!-\!\a_3)^2\equiv 2\eta_2 (\a_2-\a_3)^2$ are units. For (5), $\tanc[\t_2], \eta_2\in K$ and $(\bar{\a}_2-\bar{\a}_3)\in k^\times$. The argument for $\t_3$ is similar. For (6) we obtain $\frac{4}{c^2}\tanc[\t_2]\tanc[\t_3]\equiv 4(\a_2\!-\!\a_1)(\a_2\!+\!\a_1)(\a_3\!-\!\a_1)(\a_3\!+\!\a_1)(\a_2\!-\!\a_3)^4\equiv \eta_2\eta_3 (\a_2-\a_3)^4$ units, with $\tanc[\t_2]\tanc[\t_3], \eta_2\eta_3\in K$ and $(\bar{\a}_2\!-\!\bar{\a}_3)^2\in k^\times$.

\noindent (7) Here $\frac{1}{c}\tanc[\t_\b]\equiv (\b_2\!-\!\a_1)(\b_2\!+\!\a_1)(\b_2\!-\!\a_2)^2$. Also $\ell_1\equiv \a_2-\a_3 \bmod \pi_K^{m/2}$ and $\Delta/c\equiv (\a_2-\a_3)(\b_2-\a_1)(\b_2+\a_1)\bmod \pi_{K}^{m/2}$.

If $n<m$ then $v(\ell_1)=v(\Delta/c)=\frac{n}{2}$ and $\frac{\Delta}{c\ell_1}\equiv (\b_2-\a_1)(\b_2+\a_1)$. Thus $\frac{1}{c}\tanc[\t_\b]\equiv \frac{\Delta}{c\ell_1}(\b_2-\a_2)^2$ are units with $\tanc[\t_\b], \frac{\Delta}{\ell_1}\in K$ and $(\bar{\b}_2- \bar{\a}_2)\in k^\times$.

In general, $\frac{4}{c^2}\tanc[\t_\a]\tanc[\t_\b](\a_2-\b_2)^4\equiv 4(\a_2\!-\!\a_1)(\a_2\!+\!\a_1)(\b_2\!-\!\a_1)(\b_2\!+\!\a_1)(\a_2\!-\!\b_2)^8\equiv  -\eta_1(\dgd\eta_2+\dgt\eta_3)(\dld\eta_3+\dlt\eta_2)$. This gives the second claim as $\tanc[\t_\a]\tanc[\t_\b], \eta_1, (\dgd\eta_2+\dgt\eta_3), (\dld\eta_3+\dlt\eta_2)\in K$ and $(\bar{\a}_2-\bar{\b}_2)^2\in k^\times$.

\noindent (8)  Here $\a_1\equiv-\a_1\equiv0 \bmod \pi_K^{t+\frac{n}{2}}$, so that $\frac{4}{c}\tanc[\t]\equiv 4\a_2\b_2\a_3\b_3 \equiv \xi \bmod \pi_K^{t+\frac{n}{2}}$. As $v(\b_2)=v(\a_3)=v(\b_3)=0$ and $v(\a_2)=t$ it follows that $\frac{4\tanc[\t]}{c\xi}\equiv 1$, which gives the result as $\tanc[\t], \xi \in K$.

\noindent (9, 10) Here $\frac{1}{c}\tanc[\t_2]\equiv (\a_2\!-\!\a_1)(\a_2\!+\!\a_1)(\a_2\!-\!\a_3)(\a_2\!-\!\b_3) \bmod \pi_K^{t+\frac{n}{2}}$ and $\eta_2\equiv 2 (\a_2\!-\!\a_1)(\a_2\!+\!\a_1) \bmod \pi_K^{t+\frac{n}{2}}$. As $v(\a_2-\a_1)=t$, the other factors are units and $-\a_1\equiv\a_3\equiv\b_3$ we obtain $\frac{2}{c(\a_2-\a_1)}\tanc[\t_2]\equiv 2(\a_2+\a_1)^3 \equiv \frac{\eta_2}{(\a_2-\a_1)}(\a_2+\a_1)^2$ are units, and hence $\frac{2}{c\eta_2}\tanc[\t_2]\equiv (\a_2+\a_1)^2$. This gives the result for (9) as $\tanc[\t_2], \eta_2\in K$ and $(\bar{\a}_2+\bar{\a}_1)\in k^\times$, with a similar proof for $\tanc[\t_3]$. For (10), we obtain $\frac{4}{c^2\eta_2\eta_3}\tanc[\t_2]\tanc[\t_3]\equiv (\a_2+\a_1)^4$ with  $\tanc[\t_2]\tanc[\t_3], \eta_2\eta_3\in K$ and $(\bar{\a}_2+\bar{\a}_1)^2\in k^\times$.
\end{proof}

\begin{lemma}\label{lem:dlu} Let $K/\Q_p$ be a finite extension for an odd prime $p$ and $C/K$ a centered semistable C2D4 curve with $\l_1, \l_2, \l_3, \DG \ne 0$. Then
\begin{itemize}
\item $\dlu \in K^{\times 2}$ if and only if $(\A_1-\B_1)^2 \in K^{\times2}$,
\item $\dlu \notin K^{\times2}$ in case $1\tilde{\times}1(a)$,
\item $\dlu \notin K^{\times2}$ in case $\II^{\epsilon}_{n}\tilde{\times}_t\II_{n}(a)$,
\item $\dlu \in K^{\times 2}$ if and only if $\epsilon \delta = +$ in case $\II_{n,n}^{\epsilon,\delta}(b)$,
\item  $\dlu \in K^{\times 2}$ if and only if $\epsilon = +$ in case $\II_{n\sim n}^{\epsilon}(b)$.
\end{itemize}
Moreover in the last four cases, $\dlu$ has even valuation.  
\end{lemma}

\begin{proof}
Taking an odd degree unramified extension, we may assume that $|k| \ge23$.
By Proposition \ref{pr:Disc}(1) $\dlu= \frac{\l_1^2}{\DG^2}(\A_1-\B_1)^2 = \square (\A_1-\B_1)^2 = \square \dgu(\widehat{C})$ since $\frac{\l_1}{\DG} \in K$.
By Theorem \ref{thm:CtoChat} $\widehat{C}$ admits a model $\widehat{C}'$ whose cluster picture is the one given in the table of Theorem \ref{th:localconjectureoddprime}. 
By Remark \ref{rmk:models} we can write $\widehat{C}' = \widehat{C}_m$ for a M\"obius transformation $m(x) = \frac{\mathbf{a} x+ \mathbf{b}}{\mathbf{c} x +\mathbf{d}}$ as in Definition \ref{def:Cm}.
We explicitly compute $\dgu (\widehat{C}_m) = \frac{(\A_1-\B_1)^2(\aa\dd-\bb\cc)^2}{(\cc\A_1+\dd)^2(\cc\B_1+\dd)^2}=(\A_1-\B_1)^2\square = \dgu(\widehat{C})\square$.
Therefore, it is sufficient to check the valuation of $\A_1-\B_1$ and whether $\A_1-\B_1$ is preserved by the action of Galois on the picture of $\widehat{C}$ in the table of Theorem \ref{th:localconjectureoddprime}. 
\end{proof}

\begin{lemma}\label{lem:lw2}Let $K/\Q_p$ be a finite extension for an odd prime $p$ and $C/K$ a centered semistable C2D4 curve with $\cP, \eta_1, \DG \ne 0$.
Suppose the cluster picture of $C$ is \GR[A][A][B][B][C][C] (cases \It$(a,d)$), 
\GRU[A][A][B][B][C][C] (cases \It$(b,e)$), or 
\GRU[B][A][A][B][C][C] (cases \It$(c,f)$).
Then
$$
\lambda_{C/K}w_{C/K}= \begin{rcases}\begin{dcases}1 &\text{ if } \dlu =\square\\
(-1)^r&\text{ if } \dlu \ne \square\end{dcases}\end{rcases} = \begin{cases}(\dlu, -\frac{\l_1}{\DG})&\text{ cases }\It(a,b,d,e),\\
(\dlu, -c\frac{\l_1}{\DG})&\text{ cases }\It(c,f).\end{cases}
$$
\end{lemma}

\begin{proof}
The first equality follows from Theorem \ref{th:localconjectureoddprime}. 
The second equality is clear in cases \It (a,b,c), so suppose $C$ is as in cases \It (d,e,f).
Note that $\dlu = \frac{1}{\DG^2}(\a_2-\b_3)(\a_2-\a_3)(\b_2-\a_3)(\b_2-\b_3)$ has even valuation, so it will suffice to show that $v(\frac{\l_1}{\DG})\equiv r \mmod 2$ in cases \It (d,e) and $v(\frac{\l_1}{\DG})\equiv r +v(c) \mmod 2$ in case \It (f).

Since $\dlu \neq \square$, some Galois element swaps $\A_1$ and $\B_1$ by Lemma \ref{lem:dlu}. It follows from Theorems \ref{th:localconjectureoddprime} and \ref{th:genus2bible} that $\widehat{C}$ has cluster picture with Frobenius action \CPTCRF[R][D][D][R][D][D] if $r>0$ and \GRUM[D][R][R][D][D][D] for some $m \ge 0$ if $r=0$ (with $m=0$ if and only if the 5 roots do not form a cluster).

Let $D = \Sigma_{x,x'}v(x\!-\!x')$, the sum taken over $x\in \{\A_1,\B_1\}$ and $x' \in \{\A_2,\B_2,\A_3,\B_3\}$; it follows from the cluster pictures that $D=6m$ if $r=0$ and $D=4r$ if $r>0$. The result follows by combining the following equalities: 
\begin{itemize}[leftmargin=*]
\item $v(\frac{\l_1}{\DG}) + v(\l_2\l_3) = v(\frac{\l_1\l_2\l_3}{\DG}) \equiv r + \leftchoice{m}{\text{if }r=0}{0}{\text{if }r>0} \bmod 2$ \hfill (Theorem \ref{th:ss}(3) for $\widehat{C}$),
\item $v(\dgd\dgt)=v(\frac{c^4\l_1^4}{\DG^4}\l_2^2\l_3^2)+D \equiv 2v(\l_2\l_3)+ \leftchoice{2m}{\text{if }r=0}{0}{\text{if }r>0} \bmod 4$ \hfill (Proposition \ref{pr:Disc}(5,6)),
\item $v(\dgd\dgt)\equiv \leftchoice{0}{\text{cases \It (d,e)}}{2k}{\text{case \It (f)}} \bmod 4$ \hfill \text{(Definition \ref{def:invariants})},
\item $v(c) \equiv k \bmod 2$ \hfill (Theorem \ref{th:ss}(3) for $C$).
\end{itemize}
\end{proof}

\begin{lemma}\label{lem:lw1t1}
Let $K/\Q_p$ be a finite extension for an odd prime $p$ and $C/K$ a centered semistable C2D4 curve with $\cP, \eta_1, \DG \ne 0$. Suppose that the cluster picture with Frobenius action of $C$ is \CPTCTC[A][B][B][A][C][C] (case $\Io \!\times_t\!\Io(a)$), 
\CPTCTCF[A][B][B][A][C][C] (case $\Io \tilde{\times}_t~\Io(a)$),
\CPTCTC[A][B][C][A][B][C] (cases $\Io \!\times_t\!\Io(b,c)$) 
or
\CPTCTCF[A][B][C][A][B][C] (cases $\Io \tilde{\times}_t\Io(b,c)$).
Suppose moreover that $v(\ell_1)=t$ in cases  $\Io\! \times_t\!\Io(b,c)$ and $\Io\tilde\times_t\Io(b,c)$.
Then
$$
\lambda_{C/K}w_{C/K}= \begin{rcases}\begin{dcases}
1 & \text{ cases }\Io\! \times_t\!\Io(a,b) \text{ and } \Io \tilde{\times}_t\Io(a)\\
(-1)^{r} & \text{ cases }\Io\! \times_t\!\Io(c)\\
(-1)^t & \text{ cases }\Io \tilde{\times}_t\Io(b)\\
(-1)^{t+r} & \text{ cases }\Io \tilde{\times}_t\Io(c)\\
\end{dcases}\end{rcases}
=
 (c,\dgu)(\hat\delta_1,-\frac{\ell_1}{\Delta}). 
 $$
\end{lemma}

\begin{proof}
The first equality follows from Theorem \ref{th:localconjectureoddprime}. 

\underline{Case $\Io\! \times_t\!\Io$(a).}
Galois fixes $\a_1$ and does not permute colours, so $\a_1, \Delta\in K$ and $\delta_1, \hat\delta_1\in K^{\times 2}$.

\underline{Case $\Io \tilde\times_t\Io$(a).}
Here $\delta_1, \hat\delta_1\not\in K^{\times 2}$ and both have even valuation by Lemma \ref{lem:dlu}, so $\hat\delta_1=\delta_1\cdot\square$. Moreover, here $v(\ell_1)=0$ and $v(\Delta/c)=0$ by Lemma \ref{lem:deltazero}(1), so $ (c,\dgu)(-\frac{\l_1}{\DG}, \dlu) = (c\ell_1/\Delta,\delta_1)=1$, as both terms have even valuation. 

\underline{Cases  $\Io\! \times_t\!\Io$(b,c) and $\Io \tilde\times_t\Io$(b,c).}
Note that $v(c)\equiv t\bmod 2$ by Theorem \ref{th:ss}(3). As $v(\delta_1)$ is even,
$$
(c, \delta_1) = \leftchoice{1}{\delta_1\in K^{\times 2}}{(-1)^t}{\delta_1\not\in K^{\times 2}}= \leftchoice{1}{\text{case } \Io \times_t\Io(b,c)}{(-1)^t}{\text{case } \Io \tilde\times_t\Io(b,c)}.
$$
Similarly, as $v(\hat\delta_1)$ is even and
$v(\ell_1/\Delta)\!=\!v(\ell_1)\!-\!v(\Delta/c)\!+\!v(c)\!\equiv\! t\!-\!(2t\!+\!r)\!+\!t\! \equiv\! r \bmod 2$,
$$
(\hat\delta_1,-\ell_1/\Delta) 
= \leftchoice{1}{\hat\delta_1\in K^{\times 2}}{(-1)^{v(\ell_1/\Delta)}}{\hat\delta_1\not\in K^{\times 2}}
= \leftchoice{1}{\text{cases }\Io \times_t\Io(b), \Io \tilde\times_t\Io(b)}{(-1)^r}{\text{cases }\Io \times_t\Io(c), \Io \tilde\times_t\Io(c)}.
$$
\end{proof}

\begin{lemma}\label{lem:lw12n} Let $K/\Q_p$ be a finite extension for an odd prime $p$ and $C/K$ a centered semistable C2D4 curve with $\cP, \eta_1, \DG \ne 0$. Suppose that the cluster picture with Frobenius action of $C$ is \CPONNEZEROT[A][A][B][B][C][C] (case \Io$_{n}^{\epsilon}(a)$),
\CPTCONZEROTANG[A][A][B][B][C][C] (case $1\!\times_t\! \II_n^{\epsilon}(a)$),
\CPONNEZEROT[B][B][A][A][C][C] (case \Io$_{n}^{\epsilon}(b)$),
\CPONNZEROT[B][C][A][A][B][C] (case \Io$_{2n}^{\epsilon}(c)$),
or
\CPONNZEROT[A][B][A][B][C][C] (case \Io$_{2n}^{\epsilon}(d)$).
Then
$$
\lambda_{C/K}w_{C/K}= (\pm 1)^{n}=
\begin{cases}
  (\dgu,c\xi) & \text{ cases }\Io_{n}^{\pm}(a), 1\!\times_t\! \II_n^\pm(a),\\
  (\dgd, c\eta_1\eta_2)& \text{ cases }\Io_{n}^{\pm}(b),\\
  (\dlu, -\frac{\l_1}{\DG} \eta_1)& \text{ cases }\Io_{2n}^{\pm}(c),\\
  (\dlt, -2c\xi \dld\eta_2) & \text{ cases }\Io_{2n}^{\pm}(d).
\end{cases}
$$
\end{lemma}

\begin{proof}
The first equality follows from Theorem \ref{th:localconjectureoddprime}. 

In case $\Io_{n}^{\pm}$(a), the cluster picture shows that $n\!=\!v(\delta_1)$, and Lemma \ref{LL}(1) that the sign $\pm$ is determined by whether $\tan[\a_1][-\a_1] = c\xi \sq$ is a square in $K$. As $\tan[\a_1][-\a_1]$ has even valuation (Remark \ref{theta}), it follows that $(\pm 1)^n=(\delta_1, c\xi)$, as required. The other cases follow similarly, observing that $n=v(\delta_2), v(\hat\delta_1)\!-\!2v(\Delta), v(\hat\delta_3)$ or $v(\delta_1)\!-\!2t$, and that the sign is determined by whether $c\eta_1\eta_2, -\frac{\ell_1}{\Delta} \eta_1, -2c\xi \dld\eta_2$ or $c\xi$ is a square in $K$ for cases  $\Io_{n}^{\pm}$(b), $\Io_{2n}^{\pm}$(c), $\Io_{2n}^{\pm}$(d) and $1\!\times_t\! \II_n^\pm$(a), respectively.
(For $\Io_{2n}^{\pm}$(c),  $v(\Delta)\in\Z$ as $v(\Delta/c)=0$ by Lemma \ref{lem:deltazero}.)
\end{proof}

\begin{lemma}\label{lem:lwI2nI2m}
Let $K/\Q_p$ be a finite extension for an odd prime $p$ and $C/K$ a centered semistable C2D4 curve with $\cP, \eta_1, \DG \ne 0$. Suppose that the cluster picture with Frobenius action of $C$ is
 \CPTNTANG[B][B][C][C][A][A] (case $\II_{n,m}^{\epsilon,\delta}(a)$),
 \CPTCTNTANG[B][B][A][C][C][A] (case $\II_{n}^{\epsilon}\times_t\II_{m}^{\delta}(a)$),
 \CPTNTANGF[B][B][C][C][A][A] (case $\II_{n\sim n}^{\epsilon}(a)$),
\CPTCTNTANGF[B][B][A][C][C][A] ($\II^{\epsilon}_{n}\tilde{\times}_t\II_{n}(a)$),
\CPTNTANG[B][C][B][C][A][A] (case $\II_{n,m}^{\epsilon,\delta}(b)$),
\CPTNTANGF[B][C][B][C][A][A] (case $\II_{n\sim n}^{\epsilon}(b)$),
\CPUE[B][B][C][C][A][A] (case $\UU_{n,m,l}^{\epsilon}(a),$),
\CPUEF[B][B][C][C][A][A] (case $\UU_{n\sim n,l}^{\epsilon}(a)$). 
Moreover, suppose that $v(\l_1)=\frac n2$ in cases $\II_{n,m}^{\epsilon,\delta}(b)$ and $\II_{n\sim n}^{\epsilon}(b)$. Then $\lambda_{C/K}w_{C/K}=$ 
\begingroup\smaller[2]$$
\begin{cases}
\epsilon^{n} \delta^{m}\\
\epsilon^{n} \delta^{m}\\
\epsilon^{n}(-1)^{r}\\
\epsilon^{n} \\
\epsilon^{n+r}\delta^{\frac{m+n}{2}+r}\\
\epsilon ^{n+r}(-1)^{n}\\
\epsilon^{n+m+l} \\
\epsilon^{l} \\
\end{cases}
=
\begin{cases}
(\dgd,2c\eta_2)(\dgt,2c\eta_3) &  \text{ case }\II_{n,m}^{\epsilon,\delta}(a),\\
(\dgd,2c\eta_2)(\dgt,2c\eta_3) &  \text{ case } \II_{n}^{\epsilon}\times_t\II_{m}^{\delta}(a),\\
(\pi_K^{r}, \dlu)(\pi_K^{n}, \eta_2\eta_3)& \text{ case }\II_{n\sim n}^{\epsilon}(a),\\
(\pi_K^n, \eta_2\eta_3) &  \text{ case }\II^{\epsilon}_{n}\tilde{\times}_t\II_{n}(a),\\
(\pi_K^n,- 2\eta_1 (\dgd+\dgt) (\dgd\eta_2+\dgt\eta_3)(\dld\eta_3+\dlt\eta_2))(\dlu,\frac{\ell_1}{\DG})& \text{ case }\II_{n,m}^{\epsilon,\delta}(b),\\
(\pi_K^n,- 2\eta_1 (\dgd+\dgt) (\dgd\eta_2+\dgt\eta_3)(\dld\eta_3+\dlt\eta_2))(\dlu,\frac{\ell_1}{\DG})& \text{ case }\II_{n\sim n}^{\epsilon}(b),\\
(\dgu\dgd\dgt,c) &  \text{ case }\UU_{n,m,l}^{\epsilon}(a),\\
(\dgu,c)& \text{ case }\UU_{n\sim n,l}^{\epsilon}(a).\\
\end{cases}$$
\endgroup
\end{lemma}

\begin{proof}
The first equality follows from Theorem \ref{th:localconjectureoddprime}. 

\underline{Cases $\II_{n,m}^{\epsilon,\delta}(a), \II_{n}^{\epsilon}\times_t\II_{m}^{\delta}(a)$}: here $v(\dgd) = n+2t$ and $v(\dgt) = m+2t$ (with $t=0$ for $\II_{n,m}^{\epsilon,\delta}(a)$). Also $\epsilon = +$ if and only if $\tan[\a_2][\b_2] = 2c\eta_2\square \in K^{\times 2}$, by Lemma \ref{LL}(\ref{5}). Similarly $\delta = +$ if and only if $\tan[\a_3][\b_3] = 2c\eta_3\square \in K^{\times 2}$, by Lemma \ref{LL}(\ref{5}). The results follows since $\tan[\a_2][\b_2]$ and $\tan[\a_3][\b_3]$ have even valuations by Remark \ref{theta}.

\underline{Case $\II_{n\sim n}^{\epsilon}(a)$}: here $\epsilon = +$ if and only if $\tan[\a_2][\b_2]\tan[\a_3][\b_3]= \eta_2\eta_3\square \in K^{\times 2}$, by Lemma \ref{LL}(\ref{6}). Also $\dlu \notin K^{\times 2}$ by  Lemma \ref{lem:dlu}. The result follows as $v(\dlu)=0$ and $\tan[\a_2][\b_2]$ and $\tan[\a_3][\b_3]$ have even valuations by Remark \ref{theta}.%

\underline{Case $\II^{\epsilon}_{n}\tilde{\times}_t\II_{n}$(a)}: here $\epsilon = +$ if and only if $\tan[\a_2][\b_2]\tan[\a_3][\b_3]=\eta_2\eta_3\square \in K^{\times 2}$, by Lemma \ref{LL}(\ref{10}).

\underline{Cases $\II_{n,m}^{\epsilon,\delta}$(b) and $\II_{n\sim n}^{\epsilon}$(b)}: write the right hand side as $(\pi_K^n, 2(\dgd+\dgt)) \cdot(\pi_K^n,- \eta_1 (\dgd\eta_2+\dgt\eta_3)(\dld\eta_3+\dlt\eta_2))(\dlu,\frac{\ell_1}{\DG})$. From the cluster picture  $2(\dgd+\dgt)$ is a square in the first case, and a non square unit in the second case. So  $(\pi_K^n, 2(\dgd+\dgt))=1$ for $\II_{n,m}^{\epsilon,\delta}(b)$, and $(\pi_K^n, 2(\dgd+\dgt))=(-1)^n$ for $\II_{n\sim n}^{\epsilon}$(b). By Lemma \ref{LL}(\ref{7}), $\tan[\a_2][\a_3]\tan[\b_2][\b_3] = - \eta_1 (\dgd\eta_2+\dgt\eta_3)(\dld\eta_3+\dlt\eta_2)\square$ is a square if and only if $\epsilon \delta = +$ in the first case, and if and only if $\epsilon =+$ in the second case. In summary 
$$
(\pi_K^n,- 2\eta_1 (\dgd+\dgt) (\dgd\eta_2+\dgt\eta_3)(\dld\eta_3+\dlt\eta_2)) = \begin{cases}\epsilon^{n}\delta^n & \text{ case }\II_{n,m}^{\epsilon,\delta}$(b)$\\
\epsilon ^{n}(-1)^{n}&\text{ case }\II_{n\sim n}^{\epsilon}$(b)$.\end{cases}
$$

Also $v(\Delta) = v(c) + \frac{n}{2}+r \equiv  \frac{n}{2}+r \mmod 2$ by the semistability criterion (Theorem \ref{th:ss}), so that $v(\frac{\l_1}{\DG}) \equiv r \mmod 2$ and $v(\dlu) = v(\DG^2\dlu)-2v(\DG) \equiv \frac{m-n}{2} \mmod 2$ (with $n=m$ in the case $\II_{n\sim n}^{\epsilon}$(b)). If $n<m$, without loss of generality $\alpha_2\equiv \alpha_3\mod \pi_K^{m/2}$. This gives $\Delta/c \equiv (\beta_2-\beta_3)(\alpha_1-\alpha_2)(\beta_1-\alpha_2) \mod \pi_K^{m/2}$, so that $\Delta$ has valuation exactly $n/2$ and hence $r=0$ so $(\dlu,\frac{\ell_1}{\DG})=(\pi_K^{\frac{m-n}{2}}, \frac{\ell_1}{\DG})=\delta^{\frac{m-n}{2}}$ by Lemma \ref{LL}(7). Conversely if $n=m$ then $(\dlu,\frac{\ell_1}{\DG}) = (\pi_K^r, \dlu)$.  By Lemma \ref{lem:dlu}, $\dlu$ is a square if and only if $\epsilon \delta = +$ in the case $\II_{n,m}^{\epsilon,\delta}$(b), and if and only if $\epsilon =+$ in the case
$\II_{n\sim n}^{\epsilon}$(b) which proves the result.

\underline{Cases $\UU_{n,m,l}^{\epsilon}$(a) and $\UU_{n\sim n,l}^{\epsilon}$(a)}: here $v(c)$ is even by the semistability criterion (Theorem \ref{th:ss}), $v(\dgu)=l, v(\dgd)=n$ and $v(\dgt)=m$. By definition $\epsilon =+$ if and only if $c \in K^{\times 2}$ (see Definition \ref{no:signandfrob}). 
\end{proof}

\subsection{The value of $E_{C/K}$} We now turn to the value of $E_{C/K}$ and show that $\lambda_{C/K}w_{C/K} = E_{C/K}$ in all cases of Theorem \ref{th:errortermoddprime}.

For convenience we first recall some basic properties of Hilbert symbols. 
Recall that $(A,B)=1$  if $A$ or $B$ is a square and whenever $A,B$ are both units for odd places.

\begin{lemma}\label{lem:newHS}
Let $K/\Q_p$ be a finite extension for an odd prime $p$, and $A,B,C \in K^{\times}$. Then
\begin{enumerate}
\item $(A+B,-AB) = (A,B)$, whenever $A+B \in K^{\times}$.
\item $(A,-BC)=(\pi_K^{v(B)/2},-BC)$, whenever $A^2=B+C$ with $v(B)$ and $v(C)$ even, and with $v(B)\le v(C)$ (equivalently $v(B)\le v(A^2)$). If, moreover, $4|v(B)$ or $2|v(A)$, then  $(A,-BC)=1$.

\end{enumerate}

\end{lemma}
\begin{proof}
(1) By \cite{Serre} Ch. 3 Prop. 2(ii) $(x,1-x)=(x,-x)=1$, whenever $x, 1-x \in K^{\times}$. Using this we have $(A+B, -AB) = (A+B, -\frac{A}{B}) = ( 1+ \frac{A}{B}, -\frac{A}{B} )(B,-\frac{A}{B}) = 1\times(B, -AB) = (B,A)$. 

(2) If $v(A^2)>v(B) = v(C)$ then $B \equiv -C \mod \pi_K^{2v(A)}$ hence $-BC$ is a square.
If $v(C) \ge v(B) = v(A^2)$ then  $(A,-BC)=(\pi_K^{v(A)},-BC) = (\pi_K^{v(B)/2},-BC)$.

The second statement is then clear as either $\pi_K^{v(\frac B2)}$ is a square, or both $A$ and $-BC$ have even valuation.
\end{proof}

We will also make extensive use of the following identities in conjunction with Lemma \ref{lem:newHS}(2). 

\begin{lemma}\label{L}~Let $K$ be a field and $C/K$ a centered C2D4 curve with $\l_1 \ne 0$. Then
\begin{enumerate}[leftmargin=*]
\item\label{le:I20} $(\frac{1}{2}(\dgd+\dgt))^2 = \dgd\dgt + \ell_1^2z^2$,
where $z=\frac{(\a_2-\b_2-\a_3+\b_3)(\a_2-\b_2+\a_3-\b_3)}{2\ell_1} \in K$,
\item\label{le:I22}
$\eta_1 ^2 = 4\DG^2\dlu + \dgd\dgt$, 
\item\label{le:I41}
$\xi^2= \dld\dlt+\dgu(\frac{\xi_p-\xi_m}{\a_1})^2$, where $\frac{\xi_p-\xi_m}{\a_1} \in K$,
$\xi_p = 2(\a_2+\a_1)(\b_2+\a_1)(\a_3+\a_1)(\b_3+\a_1)$ and $\xi_m = 2(\a_2-\a_1)(\b_2-\a_1)(\a_3-\a_1)(\b_3-\a_1)$,
\item \label{le:I42}
\begin{enumerate}
\item $\eta_2^2 =\dlt+\dgd(\a_2+\b_2)^2$,
\item $\eta_3^2 = \dld+\dgt (\a_3+\b_3)^2$,
\end{enumerate}
\item\label{le:I43}
$(\dgd \eta_2+\dgt \eta_3)^2 = 4\eta_2\eta_3\dgd\dgt + \ell_1^2( \frac{\dgd \eta_2-\dgt \eta_3}{\ell_1})^2$, where $ \frac{\dgd \eta_2-\dgt \eta_3}{\ell_1} \in K$
\item\label{le:I6}
 $(\dld \eta_3+\dlt \eta_2)^2 = 4\eta_2\eta_3\dld\dlt + \ell_1^2( \frac{\dld\eta_3-\dlt\eta_2}{\ell_1})^2,$   where $ \frac{\dld\eta_3-\dlt\eta_2}{\ell_1} \in K$. 
 \item\label{le:dl1}
 $-\frac{\DG \ell_1}{c} = \DG^2\dlu+\ell_1^2\dgu -(\a_2\b_2-\a_3\b_3)^2.$
\end{enumerate}
\end{lemma}
\begin{proof}
Follows from direct computations using Definition \ref{def:invariants}.
\end{proof}

\begin{lemma}\label{lem:tor0E}Let $K/\Q_p$ be a finite extension for an odd prime $p$ and $C/K$ a centered semistable C2D4 curve with $\cP, \eta_1, \DG \ne 0$.
Suppose the cluster picture with Frobenius action of $C$ is \GR[A][A][B][B][C][C] (cases \It$(a,d)$), 
\GRU[A][A][B][B][C][C] (cases \It$(b,e)$), 
\GRU[B][A][A][B][C][C] (cases \It$(c,f)$), 
\CPTCTC[A][B][B][A][C][C] (case $\Io \!\times_t\!\Io(a)$), 
\CPTCTCF[A][B][B][A][C][C] (case $\Io \tilde{\times}_t\Io(a)$),
\CPTCTC[A][B][C][A][B][C] (cases $\Io \!\times_t\!\Io(b,c)$) 
or
\CPTCTCF[A][B][C][A][B][C] (cases $\Io \tilde{\times}_t\Io(b,c)$).
Then
$$
E_{C/K}= 
\begin{cases}
(\hat\delta_1, -\frac{\ell_1}{\Delta}) &\text{ cases }\It(a,b,d,e),\\
(\hat\delta_1, -c\frac{\ell_1}{\Delta}) &\text{ cases }\It(c,f),\\
(c,\delta_1)(\hat\delta_1,-\frac{\ell_1}{\Delta}) & \text{ cases }\Io\!\times_t\!\Io(a,b,c), \Io \tilde{\times}_t\Io(a,b,c).
\end{cases}
$$
\end{lemma}

\begin{proof}
We will abbreviate cases $\Io\!\times_t\!\Io$(a,b,c) as 1(a,b,c),  $\Io \tilde{\times}_t\Io$(a,b,c) as $\tilde 1$(a,b,c) and 2(a,b,c) as 2abc. We set $k=t$ for these cases, as the two parameters will play an identical role. We also set $k=0$ for the cases 2(a,d).

From the cluster picture of $C$ we find that
\begin{itemize}[leftmargin=*]
\item in cases 2abde $v(\delta_2)=v(\delta_3)=v(\hat\delta_2)=v(\hat\delta_3)=2k$, $v(\Delta^2\hat\delta_1)=4k$, $v(\xi)\ge 2k$, $v(\eta_2), v(\eta_3)\ge k$,
\item in cases 2cf $v(\delta_2)=0$, $v(\delta_3)=2k$, $v(\hat\delta_2)=4k$, $v(\hat\delta_3)=2k$, $v(\Delta^2\hat\delta_1)=2k$, $v(\xi)\ge 3k$, $v(\eta_1)\ge k$, $v(\eta_2)=0$, $v(\eta_3)\ge 2k$,
\item in cases 1a, $\tilde 1$a $v(\delta_2)=v(\delta_3)=v(\hat\delta_2)=v(\hat\delta_3)=2k$, $v(\delta_1)=v(\Delta^2\hat\delta_1)=0$, $v(\xi)\ge 2k$, $v(\eta_2), v(\eta_3)\ge k$, 
\end{itemize}
By Theorem \ref{th:ss}(3),  $v(c)\equiv k \bmod 2$. 

The result follows by combining the following (see below for proof of $\dagger$). We write ``both even" to mean $(A,B)=1$ because $v(A)$ and $v(B)$ are even. 
\begin{itemize}[leftmargin=*]
\item $(\frac12 (\delta_2\!+\!\delta_3), -\ell_1^2\delta_2\delta_3)$=1 (both even for 2cf; \ref{L}(\ref{le:I20}) and \ref{lem:newHS}(2) for all others),
\item $(2,\delta_2\delta_3)=1$ \hfill (both even),
\item[$\dagger$]
$(\dgd \eta_2\! +\! \dgt \eta_3, -\ell_1^2\dgd\dgt\eta_2\eta_3)(\dld\eta_3\!+\!\dlt \eta_2, -\ell_1^2\dld\dlt\eta_2\eta_3)(\eta_2\eta_3, -\dgd\dgt\dld\dlt) =$\\
\phantom{aaaaaaaaaaaaaaaaaaaaaaaaaaaaaaaaaaaaa} =$\leftchoice{1}{\text{cases 2cf}}{(\pi_K^k,\delta_2\delta_3\hat\delta_2\hat\delta_3)}{\text{all other cases}}$,

\item $(\hat\delta_2\hat\delta_3,-2)=1$ \hfill (both even),
\item $(\xi, -\delta_1\hat\delta_2\hat\delta_3)=\leftchoice{(\pi_K^{3k},-\delta_1\hat\delta_2\hat\delta_3)}{\text{cases 2cf}}{1}{\text{all other cases}}$ \hfill (\ref{L}(\ref{le:I41}) and \ref{lem:newHS}(2)),
\item $(\eta_1, -\Delta^2\hat\delta_1\delta_2\delta_3)=\leftchoice{(\pi_K^{k},-\Delta^2\hat\delta_1\delta_2\delta_3)}{\text{cases 2cf}}{1}{\text{all other cases}}$ \hfill (\ref{L}(\ref{le:I22}) and \ref{lem:newHS}(2)),
\item $(\ell_1^2, -\ell_2\ell_3)=1$ \\  \phantom{hahahahah}(if $\ell_1\not\in K$ then Galois swaps $\ell_2$ and $\ell_3$, so $v(\ell_1^2), v(\ell_2\ell_3)\in 2\Z$),
\item $(c, \Delta^2\hat\delta_1\hat\delta_2\hat\delta_3\delta_1\delta_2\delta_3)\!=\!(\pi_K^k, \Delta^2\hat\delta_1\hat\delta_2\hat\delta_3\delta_1\delta_2\delta_3)$ for 2cf \\
 \phantom{hahahahahahahahahaahahaha} ($v(c)\equiv k \bmod 2$, $v(\Delta^2\hat\delta_1\hat\delta_2\hat\delta_3\delta_1\delta_2\delta_3)\in 2\Z$),
\item $(c,\Delta^2)=1$ for 2cf \hfill ($\Delta\in K$),
\item $(c, \hat\delta_2\hat\delta_3\delta_2\delta_3)=(\pi_K^k, \hat\delta_2\hat\delta_3\delta_2\delta_3)$ for 2abde, 1abc, $\tilde{1}$abc \\
\phantom{hahahahahahahahahaahahahahahaha} ($v(c)\equiv k \bmod 2$, $v(\hat\delta_2\hat\delta_3\delta_2\delta_3)\in 2\Z$),
\item $(c, \delta_1)=1$ for cases 2abde \hfill (both even for 2ad; $\delta_1\in K^{\times 2}$ for 2be).
\end{itemize}

\noindent {\em{Proof of $\dagger$}}:
If Galois does not swap sapphire roots with turquoise roots, then $\ell_1, \delta_2,\delta_3, \hat\delta_2, \hat\delta_3,\eta_2,\eta_3\in K$, and the claim follows from:
\begin{itemize}[leftmargin=*]
\item $(\dgd \eta_2\! +\! \dgt \eta_3, -\ell_1^2\dgd\dgt\eta_2\eta_3)=(\delta_2\eta_2,\delta_3\eta_3)$ \hfill (\ref{lem:newHS}(1), $\ell_1^2\in K^{\times 2}$),
\item $(\dld\eta_3\!+\!\dlt \eta_2, -\ell_1^2\dld\dlt\eta_2\eta_3)=(\hat\delta_2\eta_3,\hat\delta_3\eta_2)$  \hfill (\ref{lem:newHS}(1), $\ell_1^2\in K^{\times 2}$),
\item $(\eta_2, -\delta_2\hat\delta_3)=\leftchoice{1}{\text{cases 2cf}}{(\pi_K^k,-\delta_2\hat\delta_3)}{\text{all other cases}}$ \hfill (both even for 2cf; \ref{L}(\ref{le:I42}a) and \ref{lem:newHS}(2) for all others),
\item $(\eta_3, -\delta_3\hat\delta_2)=\leftchoice{1}{\text{cases 2cf}}{(\pi_K^k,-\delta_3\hat\delta_2)}{\text{all other cases}}$ \hfill (\ref{L}(\ref{le:I42}b) and \ref{lem:newHS}(2)),
\item $(\delta_2, \delta_3)=1$  \hfill (both even),
\item $(\hat\delta_2, \hat\delta_3)=1$   \hfill (both even),
\end{itemize}
If Galois does swap the sapphire roots with the turquoise roots, then we are not in cases 2cf; moreover $v(\eta_2)=v(\eta_3)=n$, say, as $\eta_2$ and $\eta_3$ are Galois conjugate. Recall that $n\ge k$ in all these cases, and note that if $n\neq k$ then $\frac{1}{\pi_K^{2k}}\hat\delta_3\equiv -\frac{1}{\pi_K^{2k}}\delta_2(\a_2+\b_2)^2$ and $\frac{1}{\pi_K^{2k}}\hat\delta_2\equiv -\frac{1}{\pi_K^{2k}}\delta_3(\a_3+\b_3)^2$ in the residue field by Lemma \ref{L}(\ref{le:I42}), so that $\delta_2\delta_3\hat\delta_2\hat\delta_3\in K^{\times 2}$. The result follows from:
\begin{itemize}[leftmargin=*]
\item $(\delta_2\eta_2\!+\!\delta_3\eta_3,-\ell_1^2\delta_2\delta_3\eta_2\eta_3)=(\pi_K^{n},-\ell_1^2\delta_2\delta_3\eta_2\eta_3)$ \hfill (\ref{L}(\ref{le:I43}) and \ref{lem:newHS}(2)),
\item $(\hat\delta_2\eta_3\!+\!\hat\delta_3\eta_2,-\ell_1^2\hat\delta_2\hat\delta_3\eta_2\eta_3)=(\pi_K^{n},-\ell_1^2\hat\delta_2\hat\delta_3\eta_2\eta_3)$ \hfill (\ref{L}(\ref{le:I6}) and \ref{lem:newHS}(2)),
\item $(\eta_2\eta_3,-\delta_2\delta_3\hat\delta_2\hat\delta_3)=1$ \hfill (both even),
\item $(\pi_K^n, \hat\delta_2\hat\delta_3\delta_2\delta_3)=(\pi_K^k, \hat\delta_2\hat\delta_3\delta_2\delta_3)$ \hfill (either $n=k$ or $\hat\delta_2\hat\delta_3\delta_2\delta_3=\square$).
\end{itemize}
\end{proof}

\begin{lemma}\label{lem:tor1E}
Let $K/\Q_p$ be a finite extension for an odd prime $p$ and $C/K$ a centered semistable C2D4 curve with $\cP, \eta_1, \DG \ne 0$. Suppose that the cluster picture with Frobenius action of $C$ is \CPONNEZEROT[A][A][B][B][C][C] (case \Io$_{n}^{\epsilon}(a)$),
\CPTCONZEROTANG[A][A][B][B][C][C] (case $1\!\times_t\! \II_n^{\epsilon}(a)$),
\CPONNEZEROT[B][B][A][A][C][C] (case \Io$_{n}^{\epsilon}(b)$),
\CPONNZEROT[B][C][A][A][B][C] (case \Io$_{2n}^{\epsilon}(c)$),
or
\CPONNZEROT[A][B][A][B][C][C] (case \Io$_{2n}^{\epsilon}(d)$).
Suppose moreover that 
$v(\l_1)=v(\l_2)=v(\l_3)=v(\eta_2)=v(\eta_3)=0$.  
Then
$$
E_{C/K}=
\begin{cases}
(\dgu,c\xi) & \text{ cases }\Io_{n}^{\pm}(a),  1\!\times_t\! \II_n^\pm(a),\\
(\dgd, c\eta_1\eta_2)& \text{ cases }\Io_{n}^{\pm}(b),\\
(\dlu, -\frac{\l_1}{\DG} \eta_1)& \text{ cases }\Io_{2n}^{\pm}(c),\\
(\dlt, -2c\xi \dld\eta_2) & \text{ cases }\Io_{2n}^{\pm}(d).
\end{cases}
$$
\end{lemma}

\begin{proof}

Note that if Galois does not swap sapphire and turquoise roots, then necessarily $\ell_1$, $\Delta$, $\delta_2$, $\delta_3$, $\hat\delta_2$, $\hat\delta_3$, $\eta_2$, $\eta_3\in K$. From the cluster picture of $C$ we read off the following basic properties depending on the case:
\begin{itemize}[leftmargin=*]
\item $1_n^\pm$(a): $
v(\delta_1)=n, v(\Delta^2\hat\delta_1)=v(\delta_2)=v(\delta_3)=v(\hat\delta_2)=v(\hat\delta_3)=v(\xi)=0$, 
$v(c) \equiv 0 \bmod 2$ (Theorem \ref{th:ss}(3)),

\item $1_n^\pm$(b): 
$v(\delta_2)=n, v(\Delta^2\hat\delta_1)=v(\delta_1)=v(\delta_3)=v(\hat\delta_2)=v(\hat\delta_3)=v(\eta_1)=0$, 
$v(c) \equiv 0 \bmod 2$ (Theorem \ref{th:ss}(3)), 
$\ell_1, \Delta, \delta_2, \delta_3, \hat\delta_2, \hat\delta_3, \eta_2, \eta_3\in K$, $\hat\delta_1\in K^{\times 2}$

\item$1_n^\pm$(c): 
$v(\Delta^2\hat\delta_1)=n$, $v(\delta_1)=v(\delta_2)=v(\delta_3)=v(\hat\delta_2)=v(\hat\delta_3)=v(\eta_1)=0$, 
$v(\Delta)= v(c) \equiv 0 \bmod 2$ (Theorem \ref{th:ss}(3) and Lemma \ref{lem:deltazero}(1)).

\item $1_n^\pm$(d): 
$v(\hat\delta_3)=n, v(\Delta^2\hat\delta_1)=v(\delta_1)=v(\delta_2)=v(\delta_3)=v(\hat\delta_2)=v(\xi)=0$, 
$v(\Delta)= v(c) \equiv 0 \bmod 2$ (Theorem \ref{th:ss}(3) and Lemma \ref{lem:deltazero}(1)), 
$\ell_1$, $\Delta$, $\delta_2$, $\delta_3$, $\hat\delta_2$, $\hat\delta_3$, $\eta_2$, $\eta_3 \in K$,

\item $1\!\times_t\! \II_n^\pm$(a):
$v(\delta_1)=2t\!+\!n$, $v(\delta_2)=v(\hat\delta_2)=0$, $v(\delta_3)=v(\Delta^2\hat\delta_1)=v(\hat\delta_3)=2t$, $v(\xi)=t$,
$v(\Delta)=v(c)\equiv t \bmod 2$ (Theorem \ref{th:ss}(3) and Lemma \ref{lem:deltazero}(1)),
$\Delta^2\in K^{\times 2}$.
\end{itemize}
Recall also that, by hypothesis, $v(\l_1)=v(\l_2)=v(\l_3)=v(\eta_2)=v(\eta_3)=0$ in all cases.

We now simplify the Hilbert symbols in the expression for $E_{C/K}$ according to the case, as shown below. These are obtained simply by cancelling expressions of the form $(A,B)$ where both $A$ and $B$ have even valuation or where $A$ is a perfect square, or by using the identities of Lemma \ref{L} together with the Hilbert symbol result for expressions of the form $A^2=B+C$ of Lemma \ref{lem:newHS}(2) as indicated in the final column (the letters in brackets refer to the cases, with ``e'' referring to $1\!\times\!\II_n^\pm(a)$). The proof of the entry ($*$) is given below. The result follows by taking the product of the Hilbert symbols.

$$
\hskip-1.7cm
\begin{array}{|c|c|c|c|c|c|c|}
\hline
\vphantom{X^{X^X}}
\text{Symbol}  & 1^\pm_n(a) & 1^\pm_n(b) & 1^\pm_{2n}(c) & 1^\pm_{2n}(d) & 1\times_t \II_n^\pm(a) & \text{Proof} \cr
\hline
\vphantom{X^{X^X}}
(\frac12 (\delta_2\!+\!\delta_3),-\ell_1^2\delta_2\delta_3) & 1 & (\frac12 \delta_3,\delta_2) & 1 & 1 & 1 & \ref{L}(\ref{le:I20}), \ref{lem:newHS}(2) [acd]\cr 
(2, -\ell_1^2\delta_2\delta_3) & 1 & (2, \delta_2) & 1 & 1 & 1 & \cr 
(\delta_2\eta_2\!+\!\delta_3\eta_3, -\ell_1^2\eta_2\eta_3\delta_2\delta_3) & 1 & (\delta_3\eta_3,\delta_2) & 1 & 1 & 1 & \ref{L}(\ref{le:I43}), \ref{lem:newHS}(2) [acd]\cr 
(\hat\delta_2\eta_3\!+\!\hat\delta_3\eta_2, -\ell_1^2\eta_2\eta_3\hat\delta_2\hat\delta_3) & 1 & 1 & 1 & (\hat\delta_2\eta_3,\hat\delta_3) & 1 & \ref{L}(\ref{le:I6}), \ref{lem:newHS}(2) [abc] \cr 
(\xi, -\delta_1\hat\delta_2\hat\delta_3) & (\xi, \delta_1) & 1 & 1 & (\xi, \hat\delta_3) & (\xi,\delta_1)(\pi_K^t,-\hat\delta_2\hat\delta_3) & \ref{L}(\ref{le:I41}), \ref{lem:newHS}(2) [bc] \cr 
(\eta_2\eta_3, \delta_2\delta_3\hat\delta_2\hat\delta_3) & 1 & (\eta_2\eta_3,\delta_2) & 1 & (\eta_2\eta_3,\hat\delta_3) & 1 & \cr 
(c, \delta_1\delta_2\delta_3\hat\delta_2\hat\delta_3) & (c, \delta_1) & (c, \delta_2) & 1 &  (c,\hat\delta_3)& (c, \delta_1)(\pi_K^t,\delta_2\delta_3\hat\delta_2\hat\delta_3) &\cr 
(\eta_1, -\Delta^2\hat\delta_1\delta_2\delta_3) & 1 & (\eta_1, \delta_2) & (\eta_1, \hat\delta_1) & 1 & (\pi_K^t,-\hat\delta_1\delta_2\delta_3) & \ref{L}(\ref{le:I22}), \ref{lem:newHS}(2) [ade]\cr 
(\hat\delta_1, -\ell_1/\Delta) & \phantom{(*)}\>1\>(*) & 1 & (\hat\delta_1, -\ell_1/\Delta) & 1 & (\pi_K^t, \hat\delta_1) & \cr 
(\ell_1^2, -\ell_2\ell_3) & 1 & 1 & 1 & 1 & 1 &  \cr 
(2, -\ell_1^2) & 1 & 1 & 1 & 1 & 1 &  \cr 
(-2, \hat\delta_2\hat\delta_3) & 1 & 1 & 1 & (-2,\hat\delta_3) & 1 & \cr 
\hline
\end{array}
$$
Proof of ($*$):
In case  $1^\pm_n$(a), $v(\Delta^2\hat\delta_1)=v(\ell_1)=0$, so if $v(\Delta)=0$ then clearly $(\hat\delta_1,-\ell_1/\Delta)=1$. Suppose $v(\Delta)>0$ with $v(\Delta/c)>0$. By Lemma \ref{L}(\ref{le:dl1}), $\Delta^2\hat\delta_1\equiv (\a_2\b_2-\a_3\b_3)^2$ in the residue field. As $\frac{\a_2\b_2-\a_3\b_3}{\ell_1}$ is fixed by Galois it is $K$-rational. Thus $\frac{\Delta^2}{\ell_1^2}\hat\delta_1\in K^{\times 2}$, and, as $\Delta/\ell_1\in K$, also $\hat\delta_1\in K^{\times 2}$.
If $v(\Delta)>0$ with $v(\Delta/c)=0$ then $v(\Delta) = -v(c)$. By the semistability criterion (Theorem \ref{th:ss}) $v(c)$ is even. Thus $v(\Delta)$ is even and $(\hat\delta_1,-\ell_1/\Delta)=1$.
 \end{proof}

\begin{lemma}\label{lem:tor2E}
Let $K/\Q_p$ be a finite extension for an odd prime $p$ and $C/K$ a centered semistable C2D4 curve with $\cP, \eta_1, \DG \ne 0$. Suppose that the cluster picture with Frobenius action of $C$ is
 \CPTNTANG[B][B][C][C][A][A] (case $\II_{n,m}^{\epsilon,\delta}(a)$),
 \CPTCTNTANG[B][B][A][C][C][A] (case $\II_{n}^{\epsilon}\times_t\II_{m}^{\delta}(a)$),
 \CPTNTANGF[B][B][C][C][A][A] (case $\II_{n\sim n}^{\epsilon}(a)$),
\CPTCTNTANGF[B][B][A][C][C][A] ($\II^{\epsilon}_{n}\tilde{\times}_t\II_{n}(a)$),
\CPTNTANG[B][C][B][C][A][A] (case $\II_{n,m}^{\epsilon,\delta}(b)$),
\CPTNTANGF[B][C][B][C][A][A] (case $\II_{n\sim n}^{\epsilon}(b)$),
\CPUE[B][B][C][C][A][A] (case $\UU_{n,m,l}^{\epsilon}(a),$),
\CPUEF[B][B][C][C][A][A] (case $\UU_{n\sim n,l}^{\epsilon}(a)$). 
Moreover, suppose that 
\begin{itemize}
\item In cases $\II_{n,m}^{\epsilon,\delta}(a)$ and $\II_{n\sim n}^{\epsilon}(a)$, $v(\ell_1)=v(\ell_2)=v(\ell_3)=0$,
\item In cases $\II_{n,m}^{\epsilon,\delta}(b)$ and $\II_{n\sim n}^{\epsilon}(b)$, $v(\ell_2)=v(\ell_3)=v(\eta_2)=v(\eta_3)=0$, $v(\ell_1)=\frac{n}{2}$.
\end{itemize}
Then
\begingroup\smaller[2]
$$
E_{C/K}= 
\begin{cases}
(\dgd,2c\eta_2)(\dgt,2c\eta_3) &  \text{ case }\II_{n,m}^{\epsilon,\delta}(a),\\
(\dgd,2c\eta_2)(\dgt,2c\eta_3) &  \text{ case }\II_{n}^{\epsilon}\!\times_t\!\II_{m}^{\delta}(a),\\
(\pi_K^{r}, \dlu)(\pi_K^{n}, \eta_2\eta_3)& \text{ case }\II_{n\sim n}^{\epsilon}(a),\\
(\pi_K^n, \eta_2\eta_3) &  \text{ case }\II^{\epsilon}_{n}\tilde{\times}_t\II_{n}(a),\\
(\pi_K^n,- 2\eta_1 (\dgd+\dgt) (\dgd\eta_2+\dgt\eta_3)(\dld\eta_3+\dlt\eta_2))(\dlu,\frac{\ell_1}{\DG})& \text{ case }\II_{n,m}^{\epsilon,\delta}(b),\\
(\pi_K^n,- 2\eta_1 (\dgd+\dgt) (\dgd\eta_2+\dgt\eta_3)(\dld\eta_3+\dlt\eta_2))(\dlu,\frac{\ell_1}{\DG})& \text{ case }\II_{n\sim n}^{\epsilon}(b),\\
(\dgu\dgd\dgt,c) &  \text{ case }\UU_{n,m,l}^{\epsilon}(a),\\
(\dgu,c)& \text{ case }\UU_{n\sim n,l}^{\epsilon}(a).\\
\end{cases}
$$
\endgroup
\end{lemma}

\begin{proof}

We first record basic properties of the quantities appearing in the Hilbert symbols on the statement. These follow directly from the definitions and the cluster picture of $C$, along with the semistability criterion (Theorem \ref{th:ss}) to control the parity of $v(c)$ and occasionally Hensel's Lemma to show that certain quantities are squares in $K$ when they are squares over the residue field.

We then simplify the Hilbert symbols defining $E_{C/K}$ according to the cluster picture of $C$. Here we cancel expressions of the form $(A,B)$ where both $A$ and $B$ have even valuation or where $A$ is a perfect square, and occasionally use the idenity $(A+B,-AB)=(A,B)$ (Lemma \ref{lem:newHS}(1), indicated by a $*$) or the Hilbert symbol expression of Lemma \ref{lem:newHS}(2) together with the appropriate identity from Lemma \ref{L} (indicated by a $\dagger$). The simplified expressions are displayed in the tables below. The result follows by taking the product of the Hilbert symbols.

\begin{itemize}
\item[$ \II_{n,m}^{\epsilon,\delta}$(a):] $v(\ell_1)=v(\ell_2)=v(\ell_3)=v(\delta_1)=v(\Delta^2\hat\delta_1)=v(\hat\delta_2)=v(\hat\delta_3)=v(\eta_2)=v(\eta_3)=0$, 
$v(\xi)\ge 0$, 
$v(c)\equiv 0\bmod 2$, 
$\ell_1, \Delta, \delta_2, \delta_3, \eta_2, \eta_3 \in K^\times$, 
$\eta_1=2\square$, $\hat\delta_1=\square$;

\item[$\II_n^\epsilon\!\times_t\!\II_m^\delta$(a):] 
$v(\Delta^2\hat\delta_1)=0$, 
$v(\hat\delta_2)=v(\hat\delta_3)=2t$, 
$v(\xi)\ge 2t$, 
$\ell_1, \Delta, \delta_2, \delta_3, \hat\delta_2, \hat\delta_3, \eta_2, \eta_3 \in K^\times$, 
$\delta_1=\square$, $\eta_1=2\square$, $\hat\delta_1=\square$,
$\hat\delta_2=\square$, $\hat\delta_3=\square$; 

\item[$\II_{n\sim n}^{\epsilon}$(a):]
$v(\ell_1)=v(\ell_2)=v(\ell_3)=v(\delta_1)=v(\Delta^2\hat\delta_1)=v(\hat\delta_2)=v(\hat\delta_3)=v(\eta_1)=v(\eta_2)=v(\eta_3)=0$, $v(\delta_2)=v(\delta_3)=n$, $v(\xi)\ge 0$, $v(c)\equiv 0\bmod 2$, $v(\Delta)=v(c)+r\equiv r \bmod 2$;

\item[$\II_n^\epsilon\tilde\times_t\II_n$(a):] $v(\ell_1)=v(\ell_2)=v(\ell_3)=v(\delta_1)=v(\Delta^2\hat\delta_1)=v(\eta_1)=0$,
$v(\hat\delta_2)=v(\hat\delta_3)=2t$, 
$v(\eta_2)=v(\eta_3)=t$, 
$v(\delta_2)=v(\delta_3)=n\!+\!2t$, 
$v(\xi)\ge 2t$, 
$v(\Delta)=v(c) \equiv t \bmod 2$ (Theorem \ref{th:ss}(3), Lemma \ref{lem:deltazero}), 
$\hat\delta_1=\delta_1\square$ (as $v(\delta_1)=0$, $\delta_1\neq\square$ and, by Lemma \ref{lem:dlu} $v(\hat\delta_1)\in 2\Z$, $\hat\delta_1\neq\square$).
\end{itemize}

$$
\hskip-.2cm
\begin{array}{|c|c|c|c|}
\hline
\vphantom{X^{X^X}}
\text{Symbol}  
  & \II_{n,m}^{\epsilon,\delta}\text{(a)}, \II_{n}^{\epsilon}\!\times_t\!\II_{m}^{\delta}\text{(a)}& \II_{n\sim n}^{\epsilon}\text{(a)} & \II_n^\epsilon\tilde\times_t\II_n\text{(a)} \cr
\hline
\vphantom{X^{X^X}}
(\frac12 (\delta_2\!+\!\delta_3),-\ell_1^2\delta_2\delta_3) 
  & \phantom{*}\>\>(\delta_2,\delta_3)(2,\delta_2\delta_3) \>\> * & \phantom{\dagger}\quad\> (\pi_K^n,-\ell_1^2\delta_2\delta_3)\quad\>\dagger &  \phantom{\dagger}\quad\>\>\> (\pi_K^n,-\ell_1^2\delta_2\delta_3) \quad\>\>\>\dagger \cr 
(2, -\ell_1^2\delta_2\delta_3) 
  & (2,\delta_2\delta_3) & 1 & 1 \cr 
(\delta_2\eta_2\!+\!\delta_3\eta_3, -\ell_1^2\eta_2\eta_3\delta_2\delta_3)
  & \phantom{*}\>\>\>\>\>(\delta_2\eta_2, \delta_3\eta_3) \>\>\>\>\> * & \phantom{\dagger} \>(\pi_K^n,-\ell_1^2\eta_2\eta_3\delta_2\delta_3) \>\dagger & \phantom{\dagger}\> (\pi_K^{n+t},-\ell_1^2\eta_2\eta_3\delta_2\delta_3) \>\dagger \cr 
(\hat\delta_2\eta_3\!+\!\hat\delta_3\eta_2, -\ell_1^2\eta_2\eta_3\hat\delta_2\hat\delta_3) 
  &\phantom{*}\>\>\>\>\>\>\>\>\> (\eta_2, \eta_3) \>\>\>\>\>\>\>\>\> * & \phantom{\dagger}\qquad\quad\>\>\>\> 1\qquad\quad\>\>\>\>\dagger  & \phantom{\dagger}\>\>\>(\pi_K^t, -\ell_1^2\eta_2\eta_3\hat\delta_2\hat\delta_3)\>\>\>\dagger \cr 
(\xi, -\delta_1\hat\delta_2\hat\delta_3) 
  & \phantom{\dagger}\qquad\quad\>\> 1\qquad\quad\>\>\dagger & \phantom{\dagger}\qquad\quad\>\>\>\>1 \qquad\quad\>\>\>\>\dagger & \phantom{\dagger}\qquad\qquad\>1 \qquad\qquad\>\dagger \cr 
(\eta_2\eta_3, \delta_2\delta_3\hat\delta_2\hat\delta_3) 
  & (\eta_2\eta_3,\delta_2\delta_3) & 1 & 1\cr 
(c, \delta_1\delta_2\delta_3\hat\delta_2\hat\delta_3) 
  & (c, \delta_2\delta_3) & 1 & (\pi_K^t,\delta_1\delta_2\delta_3\hat\delta_2\hat\delta_3) \cr 
(\eta_1, -\Delta^2\hat\delta_1\delta_2\delta_3) 
  & (2, \delta_2\delta_3) & 1 & 1 \cr 
(\hat\delta_1, -\ell_1/\Delta) 
  & 1 & (\hat\delta_1,\pi_K^r) & (\delta_1,\pi_K^{t}) \cr 
(\ell_1^2, -\ell_2\ell_3) 
  & 1 & 1 & 1 \cr 
(2, -\ell_1^2) 
  & 1 & 1 & 1 \cr 
(-2, \hat\delta_2\hat\delta_3) 
  & 1 & 1 & 1 \cr 
\hline
\end{array}
$$

\begin{itemize}
\item[$\II_{n,m}^{\epsilon,\delta}$(b):] 
$v(\ell_2)=v(\ell_3)=v(\eta_2)=v(\eta_3)=v(\delta_1)=v(\delta_2)=v(\delta_3)=v(\hat\delta_2)=v(\hat\delta_3)=v(\eta_1)=0$, $v(\ell_1)=\frac{n}{2}$,  $v(\xi)\ge 0$, $v(c)\equiv 0\bmod 2$, 
$v(\Delta^2)=n\!+\!2r\!+\!2v(c)\equiv n\bmod 2$, 
$\ell_2\ell_3=\square$, 
$\eta_1=-\square$,
$v(\dgd+\dgt) = v(\dgd\eta_2+\dgt \eta_3) = v(\dld\eta_3+\dlt\eta_2) = 0$ (because $\delta_2\equiv \delta_3$, $\eta_2\equiv \eta_3$, $\hat\delta_2\equiv\hat\delta_3$ in the residue field); %
\item[$\II_{n\sim n}^{\epsilon}$(b):] 
$v(\ell_2)=v(\ell_3)=v(\eta_2)=v(\eta_3)=v(\delta_1)=v(\delta_2)=v(\delta_3)=v(\hat\delta_2)=v(\hat\delta_3)=v(\eta_1)=0$, $v(\ell_1)=\frac{n}{2}$, 
$v(\hat\delta_1\Delta^2)=n$,
$v(\xi)\ge 0$, $v(c)\equiv 0\bmod 2$, 
$v(\hat\delta_1)\equiv v(\hat\delta_1\Delta^2)-v(\Delta^2/c^2)\equiv 0\bmod 2$, 
$\ell_2\ell_3=\square$, 
$v(\dgd+\dgt) = v(\dgd\eta_2+\dgt \eta_3) = v(\dld\eta_3+\dlt\eta_2) = 0$ (because $\delta_2\equiv \delta_3$, $\eta_2\equiv \eta_3$, $\hat\delta_2\equiv\hat\delta_3$ in the residue field); %

\item[$\UU_{n,m,l}^{\epsilon}$(a):]
$v(\ell_1)=v(\ell_2)=v(\ell_3)=v(\hat\delta_1\Delta^2)=v(\hat\delta_2)=v(\hat\delta_3)=v(\eta_1)=v(\eta_2)=v(\eta_3)=v(\xi)=0$,
$v(\Delta)\equiv v(c)\equiv 0 \bmod 2$ (Theorem \ref{th:ss}(3), Lemma \ref{lem:deltazero}), 
$v(\hat\delta_1)\equiv 0 \bmod2$,
$\delta_2, \delta_3, \eta_2, \eta_3\in K$, $\eta_1, \eta_2, \eta_3=2\square$;

\item[$\UU_{n\sim n,l}^{\epsilon}$(a):]
$v(\ell_1)=v(\ell_2)=v(\ell_3)=v(\hat\delta_1\Delta^2)=v(\hat\delta_2)=v(\hat\delta_3)=v(\eta_1)=v(\eta_2)=v(\eta_3)=v(\xi)=0$,
$v(\delta_2)=v(\delta_3)=n$, 
$v(\Delta)\equiv v(c)\equiv 0 \bmod 2$ (Theorem \ref{th:ss}(3), Lemma \ref{lem:deltazero}), 
$v(\hat\delta_1)\equiv 0 \bmod2$
$\eta_2\eta_3=\square$.
\end{itemize}

$$
\hskip-1cm
\begin{array}{|c|c|c|c|c|}
\hline
\vphantom{X^{X^X}}
\text{Symbol}  
  & \II_{n,m}^{\epsilon,\delta}\text{(b)} & \II_{n\sim n}^{\epsilon}\text{(b)} & \UU_{n,m,l}^{\epsilon}\text{(a)} & \UU_{n\sim n,l}^{\epsilon}\text{(a)}  \cr
\hline
\vphantom{X^{X^X}}
(\frac12 (\delta_2\!+\!\delta_3),-\ell_1^2\delta_2\delta_3) 
   &(2(\delta_2\!+\!\delta_3), \pi_K^n) &(2(\delta_2\!+\!\delta_3), \pi_K^n) & \phantom{*}\> (\delta_2, \delta_3)(2, -\ell_1^2\delta_2\delta_3)\>*  &\phantom{\dagger}\> (\pi_K^n,-\ell_1^2\delta_2\delta_3)\>\dagger  \cr 
(2, -\ell_1^2\delta_2\delta_3) 
  & (2,\ell_1^2) &  (2,\ell_1^2)  & (2, -\ell_1^2\delta_2\delta_3)  & 1 \cr 
(\delta_2\eta_2\!+\!\delta_3\eta_3, -\ell_1^2\eta_2\eta_3\delta_2\delta_3)
  & (\delta_2\eta_2\!+\!\delta_3\eta_3, \pi_K^n) &(\delta_2\eta_2\!+\!\delta_3\eta_3, \pi_K^n)  & \phantom{*}\qquad\> (2\delta_2, 2\delta_3)\qquad\> * & \phantom{\dagger}\> (\pi_K^n, -\ell_1^2\delta_2\delta_3)\>\dagger \cr 
(\hat\delta_2\eta_3\!+\!\hat\delta_3\eta_2, -\ell_1^2\eta_2\eta_3\hat\delta_2\hat\delta_3) 
  & (\hat\delta_2\eta_3\!+\!\hat\delta_3\eta_2, \pi_K^n) & (\hat\delta_2\eta_3\!+\!\hat\delta_3\eta_2, \pi_K^n)  &\phantom{\dagger}\qquad \qquad1\qquad\qquad\dagger & \phantom{\dagger}\qquad\>\>\>\>1\>\>\>\>\qquad\dagger \cr 
(\xi, -\delta_1\hat\delta_2\hat\delta_3) 
  & \phantom{\dagger}\qquad\>\>\>1\>\>\>\qquad\dagger & \phantom{\dagger}\qquad \>\>\>1\>\>\> \qquad\dagger &\phantom{\dagger}\qquad\qquad 1\qquad\qquad\dagger &\phantom{\dagger}\qquad\>\>\>\> 1\>\>\>\>\qquad\dagger \cr 
(\eta_2\eta_3, \delta_2\delta_3\hat\delta_2\hat\delta_3) 
  & 1 & 1 & 1 & 1 \cr 
(c, \delta_1\delta_2\delta_3\hat\delta_2\hat\delta_3) 
  & 1 & 1 & (c, \delta_1\delta_2\delta_3) & (c,\delta_1) \cr 
(\eta_1, -\Delta^2\hat\delta_1\delta_2\delta_3) 
  & (\eta_1, \pi_K^n)(\hat\delta_1, -1) & (\eta_1, \pi_K^n) & (2,\delta_2\delta_3) & 1  \cr 
(\hat\delta_1, -\ell_1/\Delta) 
  & (\hat\delta_1, -\ell_1/\Delta) & (\hat\delta_1, \ell_1/\Delta) & 1 & 1 \cr 
(\ell_1^2, -\ell_2\ell_3) 
 & (\pi_K^n,-1) & (\pi_K^n,-1) & 1 & 1 \cr 
(2, -\ell_1^2) 
  & (2, \ell_1^2) & (2, \ell_1^2) & 1 & 1 \cr 
(-2, \hat\delta_2\hat\delta_3) 
  & 1 & 1 & 1 & 1 \cr 
\hline
\end{array}
$$
\end{proof}

\section{Even places}\label{s:2adic}

In this section we look at C2D4 curves with good ordinary reduction over 2-adic fields.
Such curves admit a nice model --- essentially, curves with good ordinary reduction turn out to be those with cluster picture $\CPU[D][D][D][D][D][D]$ with depth of each twin precisely $v(4)$. 
Theorem \ref{thm:2orderrorterm} then shows that Conjecture \ref{conj:local} holds for curves with this model and a specific Richelot isogeny. In the next section we will show that the conjecture is independent of the choice of model and independent of the choice of the isogeny, and hence that it holds for all curves with good ordinary reduction at 2-adic primes. 

We begin with a preliminary lemma about 2-adic fields and Hilbert symbols.

\begin{lemma}\label{lem:extensionsQ2}
Let $K/\Q_2$ be a finite extension. Then

(i) $K(\sqrt{x})/K$ is unramified if and only if $x=\square\cdot (1+4t)$ for some $t\in\cO_K$;

(ii) If $x=\square\cdot (1+4t)$ for some $t\in\cO_K$ then
$$
 (x,u)=1\qquad \text{for all } u\in\cO_K^\times. 
$$

(iii) If $F=K(\sqrt{L})$ is the quadratic unramified extension and $x\in F$ such that $x^2+\Frob_{F/K}x^2$ is a unit, then
$$
  (x^2+\Frob_{F/K}x^2,-1)= (2,L).
$$
\end{lemma}

\begin{proof}
(i) 
Fix a set of representatives $S\supseteq\{0,1\}$ of $\cO_K/(\pi_K)$ and consider the equation
$$
  (a_0+a_1\pi_K+ a_2\pi_K^2+\ldots)^2\equiv 1+4t \bmod 4\pi_K, \qquad a_i\in S,
$$
for a given $t\in\cO_K$. Equating the coefficients of powers of $\pi_K$, we must necessarily have $a_0=1$ and $a_1=\ldots = a_{n-1}=0$ for $n=v(2)$. The equation is then soluble if and only if $a_n^2+a_n\equiv t \bmod \pi_K$ is soluble.
Hence it is always soluble in the quadratic unramified extension of $K$, but not in $K$ for a suitable choice of $t$.

It follows that elements of the form $x=\square\cdot (1+4t)$ with $t\in\cO_K$ are squares in the quadratic unramified extension of $K$, and that some of these elements have $K(\sqrt{x})\neq K$. The set of such elements is a subgroup of $K^\times$ that properly contains $K^{\times 2}$, and hence must contain all the elements $x\in K$ such that $K(\sqrt{x})/K$ is unramified.

(ii) Follows from (i) and the fact that all units in $K$ are norms from any unramified extension.

(iii)
Write $x=a+b\sqrt{L}$ for some $a, b\in K$. Then
$
  x^2 + \Frob_{F/K}x^2 = 2a^2 + 2Lb^2, 
$
so that
$$
 (x^2 + \Frob_{F/K}x^2, -L) = (2,-L) = (2,L),
$$
as $(a^2+Lb^2)=(a+b\sqrt{-L})(a-b\sqrt{-L})$ is a norm from $K(\sqrt{-L})$, and $(2,-1)=1$.
Also
$$
 (x^2 + \Frob_{F/K}x^2, L) = 1,
$$
because $x^2 + \Frob_{F/K}x^2$ is a unit by hypothesis, and hence a norm from the unramified extension $K(\sqrt{L})/K$. The result follows.
\end{proof}

We now turn to models of genus 2 curves with ordinary reduction.

\begin{proposition}\label{prop:2ordinary}
Let $K/\Q_2$ be a finite extension with residue field $k$ of size $|k|\ge 4$, and let $C/K$ be a genus 2 curve with good reduction. 

(i) $C/K$ has ordinary reduction if and only if it admits a model of the form
$$
  y^2=cf(x)
$$
with $c\equiv 1 \bmod 4$, and with $f(x)\in\cO_K[x]$ monic whose roots can be labelled $\alpha_1$, $\beta_1$, $\alpha_2$, $\beta_2$, $\alpha_3$, $\beta_3$ 
so that for all $i\neq j$
$$
(x\!-\!\alpha_i)(x\!-\!\beta_i)\!\in\! K^{nr}[x], \quad v(\alpha_i\!-\!\beta_i)\!=\! v(4), \quad v(\alpha_i\!-\!\alpha_j)\!=\!v(\beta_i\!-\!\beta_j)\!=\!v(\alpha_i\!-\!\beta_j)\!=\!0.
$$

For this model, the kernel of the reduction map on the 2-torsion points of $\Jac C/K$ consists of $0$, $[(\alpha_1,0),(\beta_1,0)]$, $[(\alpha_2,0),(\beta_2,0)]$, and $[(\alpha_3,0),(\beta_3,0)]$.

(ii)
If $C/K$ has ordinary reduction and $\phi$ is the Richelot isogeny whose kernel is precisely the 2-torsion points in the kernel of the reduction map, then the Richelot dual curve $\widehat{C}$ also has good ordinary reduction. 

(iii)
If $C/K$ is a C2D4 curve with ordinary reduction, such that the kernel of the associated Richelot isogeny agrees with the kernel of the reduction map on 2-torsion points, 
and if $|k|\ge 32$,
then the model in (i) can be further taken to have 
the C2D4 structure defined by the $\alpha_i$, $\beta_i$ and
$$
\frac{\delta_2+\delta_3}{16}, \quad \frac{\delta_2\eta_2+\delta_3\eta_3}{32}, \quad \frac{\hat\delta_2\eta_3+\hat\delta_3\eta_2}{8} \quad \in\quad \cO_K^\times.
$$
\end{proposition}

\begin{proof}
(i) This is \cite{2adics} Theorem 1.2 and Proposition 1.16.

\medskip

(ii) Take the model $y^2=cf(x)$ for $C/K$ given by (i). We will show that $\widehat{C}$ has a similar model, and hence also has good ordinary reduction by (i). 

Using the fact that $\alpha_i,\beta_i$ have non-negative valuations (they satisfy a polynomial with unit leading term and integral coefficients) and the valuations of their pairwise differences, one readily checks that:
\begin{itemize}
\item $v(\ell_1)=v(\ell_2)=v(\ell_3)=v(2)$,
\item $\hat{r}(x), \hat{s}(x), \hat{t}(x)\in \cO_K^{nr}[x]$, so that $v(\hat\alpha_i),v(\hat\beta_i)\ge 0$  (see Definition \ref{def:richelotdual}),
\item $\Delta/c\equiv 2(\alpha_1-\alpha_2)(\alpha_3-\alpha_1)(\alpha_3-\alpha_2) \bmod 4$, so that $v(\Delta)=v(2)$.
\end{itemize}
By Proposition \ref{pr:Disc}(1,2,3), we deduce that $v(\hat\alpha_i\!-\!\hat\beta_i)=0$ for $i=1, 2, 3$, so no cluster of $\widehat{C}$ of size $<6$ can contain two roots of the same colour. Also, there cannot be such a cluster that contains roots of all three colours: otherwise Lemma \ref{lem:deltazero}(3) 
would imply that $v\bigl(\frac{\Delta(\widehat{C})}{c(\widehat{C})}\bigr)\!>\!0$, but by Proposition \ref{pr:Disc}(7) $v\bigl(\frac{\Delta(\widehat{C})}{c(\widehat{C})}\bigr)=v\bigl(\frac{2\Delta(C)/c^2}{\ell_1\ell_2\ell_2/\Delta(C)}\bigr)=0$. Thus the only possible clusters are twins that contain roots of different colours.

By Proposition \ref{pr:Disc}(4,5,6),
$$
v(\hat\alpha_2\!-\!\hat\alpha_3)+v(\hat\alpha_2\!-\!\hat\beta_3)+v(\hat\beta_2\!-\!\hat\alpha_3)+v(\hat\beta_2\!-\!\hat\beta_3) = v(4),
$$
and similarly for the indices $1,2$ and $1,3$. It follows that the cluster picture of $\widehat{C}$ must be $\CPUNDZERO[S][T][R][S][R][T]$, with the depth of each twin exactly $v(4)$. 
Since $r(x), s(x), t(x)\in K^{nr}[x]$, the polynomials $\hat{r}(x), \hat{s}(x), \hat{t}(x)$ also have coefficients in $K^{nr}$. 

To deduce that $\widehat{C}$ has good ordinary reduction using (i) it remains to show that the leading term $c(\widehat{C})=\frac{\ell_1\ell_2\ell_3}{\Delta}$ is of the form $\square\cdot (1+4z)$ for some $z\in \cO_K$. Note that the quadratic twist $\widehat{C}'$ of $\widehat{C}$ by $c(\widehat{C})$ does have good ordinary reduction by (i), and hence its Jacobian has good reduction. The Jacobian of $\widehat{C}$ also has good reduction, being isogenous to $\Jac C/K$. Thus $K\Bigl(\sqrt{c(\widehat{C})}\Bigr)/K$ must be unramified, and hence $c(\widehat{C})$ is of the required form by Lemma \ref{lem:extensionsQ2}(i).

\medskip

(iii)
Take the model for $C/K$ given by (i). Relabelling the roots if necessary, by (ii) we may assume that the C2D4 structure is given by the $\alpha_i$, $\beta_i$. We now need to adjust the model so that the claimed invariants are units.

As $v(\alpha_1-\beta_1)\ge v(4)$, the term $\frac{\alpha_1+\beta_1}{2}$ lies in $\cO_K$. Applying the translation to the $x$-coordinate $x\mapsto x+\frac{\alpha_1+\beta_1}{2}$ we may thus assume that the C2D4 model is centered, that is $\alpha_1=-\beta_1$.

Recall from Definition \ref{def:Mt} that for $t\in K\setminus\{\frac1{\alpha_1}, -\frac1{\alpha_1}\}$ we have a M\"obius transformation $M_t$ and model $C_t$.
We now proceed as in the proof of Theorem \ref{thm:mobiusoverbalance} to pick a suitable value for $t\in\cO_K$ that gives a model with the required properties.

By construction, the roots $\alpha_i(C_t)=M_t(\alpha_i)$ and $\beta_i(C_t)=M_t(\beta_i)$ for the model $C_t$ satisfy $(x\!-\!\alpha_i(C_t))(x\!-\!\beta_i(C_t))\in K^{nr}[x]$.
As in the first paragraph of the proof of Theorem \ref{thm:mobiusoverbalance}, if $t\not\equiv -1/\alpha_i, -1/\beta_i$ in $k$ for any $i$, then the model $C_t$ will also have $v(\alpha_i(C_t)), v(\beta_i(C_t))\ge 0$, $v(\alpha_i(C_t)-\beta_i(C_t))=v(4)$ and $v(r-r')=0$ for all other pairs of roots. As in the last paragraph of the proof of (ii), $c(C_t)$ is necessarily of the form $\square\cdot (1+4z)$ for some $z\in\cO_K$. It remains to ensure that 
$\frac{\delta_2\!+\!\delta_3}{16}, \frac{\delta_2\eta_2\!+\!\delta_3\eta_3}{32}$ and $\frac{\hat\delta_2\eta_3\!+\!\hat\delta_3\eta_2}{8}$
are units for the $C_t$ model.

Write $\alpha_i=\beta_i+4u_i$, so that $v(u_i)=0$.
A direct computation shows that
$$
\begin{array}{rl}
\frac{\delta_2(C_t)+\delta_3(C_t)}{16}&=\frac{F_1(t)}{(1+\alpha_2t)^2(1+\beta_2t)^2(1+\alpha_3t)^2(1+\beta_3t)^2}, \cr
\frac{\delta_2(C_t)\eta_2(C_t)+\delta_3(C_t)\eta_3(C_t)}{32}&=\frac{F_2(t)}{(1+\alpha_2t)^4(1+\beta_2t)^4(1+\alpha_3t)^4(1+\beta_3t)^4}, \cr
\frac{\hat\delta_2(C_t)\eta_3(C_t)+\hat\delta_3(C_t)\eta_2(C_t)}{8}&=\frac{F_3(t)}
{(1+\alpha_2t)^4(1+\beta_2t)^4(1+\alpha_3t)^4(1+\beta_3t)^4},
\end{array}
$$
where $F_1(t), F_2(t), F_3(t)$ are polynomials in $\cO_K[t]$ that reduce in $k[t]$ to
$$
\begin{array}{cl}
F_1(t)\equiv & (\beta_3^4u_2^2+\beta_2^4u_3^2)t^4+(u_2^2+u_3^2),\cr
F_2(t)\equiv & (\beta_2^4\beta_3^8u_2^2+\beta_3^4\beta_2^8u_3^2)t^{10} + (\beta_2^2\beta_3^8u_2^2+\beta_3^2\beta_2^8u_3^2)t^{8}+ \cr
&+(\beta_2^4u_2^2+\beta_3^4u_3^2)t^2+
(\beta_2^2u_2^2+\beta_3^2u_3^2), \cr
F_3(t) \equiv & (\beta_2^8\beta_3^6+\beta_3^8\beta_2^6)t^8+ (\beta_2^8+\beta_3^8)t^2+(\beta_2^6+\beta_3^6).
\end{array}
$$
None of these is the zero polynomial in $k(t)$, as we now explain. Since $\beta_2\not\equiv \beta_3$ and $z\mapsto z^2$ is an automorphism of $k$, we deduce that $\beta_2^2\not\equiv \beta_3^2$, $\beta_2^4\not\equiv \beta_3^4$ and $\beta_2^8\not\equiv \beta_3^8$; this deals with $F_3(t)$. Moreover, if $u_2^2\equiv u_3^2$, then $\beta_3^4u_2^2\not\equiv\beta_2^4u_3^2$, so $F_1(t)$ is not zero. Finally, if $\beta_2^2u_2^2\equiv \beta_3^2u_3^2$ then
$\beta_2^2\beta_3^8u_2^2\not\equiv \beta_3^2\beta_2^8u_3^2$, so $F_2(t)$ is not zero.

Thus, so long as $t\in\cO_K$ avoids the residues of the $\bar{k}$-roots of $F_1(t)$, $F_2(t)$ and $F_3(t)$ (at most $22$ such) and the residues of $-1/\alpha_i, -1/\beta_i$ (at most $6$ such), the required expressions will be units.
\end{proof}

\begin{theorem}\label{thm:2orderrorterm}
Let $K/\Q_2$ be a finite extension and $C/K$ a C2D4 curve with good ordinary reduction given by $y^2=cf(x)$ with $f(x)\in\cO_K[x]$ monic, satisfying $(x\!-\!\alpha_i)(x\!-\!\beta_i)\!\in\! K^{nr}[x]$ and $v(\alpha_i\!-\!\beta_i)=v(4)$ for all $i$, and $v(r\!-\!r')=0$ for all other pairs of roots $r, r'$.
Suppose moreover that $\cP, \Delta\neq 0$ and that $\frac{\delta_2+\delta_3}{16}, \frac{\delta_2\eta_2+\delta_3\eta_3}{32}, \frac{\hat\delta_2\eta_3+\hat\delta_3\eta_2}{8}\in \cO_K^\times$. Then Conjecture \ref{conj:local} holds for C/K.
\end{theorem}

\begin{proof}
As $\Jac C/K$ has good reduction, $w_{\Jac C/K}=1$.

By Proposition \ref{prop:2ordinary}(ii), $\widehat{C}/K$ also has good ordinary reduction, so that neither $C/K$ nor $\widehat{C}/K$ are deficient. Since the kernel of the reduction map on 2-torsion points coincides with the kernel of the Richelot isogeny $\phi$ on $\Jac C/K$, Lemma \ref{lem:kercoker} and Theorem \ref{first cokernel theorem} show that $\lambda_{C/K,\phi}=1\cdot 1\cdot (-1)^{[K:\Q_2]\cdot 2}=1$. It remains to prove that the product of Hilbert symbols is also trivial, that is $E_{C/K}=1$.

Since $\alpha_1$, $\beta_1$ are integral, $\alpha_1\!+\!\beta_1\in \cO_K$ and $4|(\alpha_1\!-\!\beta_1)$, it follows that $\frac{\alpha_1+\beta_1}{2}\in\cO_K$. 
The change of variables $x\mapsto x\!+\!\frac{\alpha_1+\beta_1}{2}$ does not change any of the hypotheses on $C$ in the statement, including that $f(x)\in\cO_K[x]$. We may (and will) therefore assume that $C$ is centered, that is $\alpha_1=-\beta_1$.


Let $\gamma_i=\frac{\alpha_i+\beta_i}{2}$,  so that $\{\gamma_1, \gamma_2, \gamma_3\}$ is preserved by Galois. Note that 
$\gamma_i \equiv \alpha_i\equiv \beta_i \bmod 2$, 
$\gamma_i^2 \equiv \alpha_i^2\equiv \beta_i^2 \bmod 4$ and 
$\gamma_i^4 \equiv \alpha_i^4\equiv \beta_i^4 \bmod 8$. 
As $C$ is centered, $\gamma_1\equiv 0 \bmod 2$, and $v(\gamma_2)= v(\gamma_3)=0$.

Explicit computation shows that
\begin{itemize}
\item $\Delta^2\hat\delta_1 \equiv (\gamma_2-\gamma_3)^4 \bmod 8,  \text{ so that } \Delta^2\hat\delta_1 =\square,$
\item $\frac{1}{2}\eta_2 \equiv \gamma_2^2 \bmod 4,\frac{1}{2}\eta_3 \equiv \gamma_3^2 \bmod 4,  \text{ so  } \eta_2\eta_3 =\square(1+4z),$
\item $\frac{1}{4}\hat\delta_2\equiv\gamma_3^4\bmod 8$, $\frac{1}{4}\hat\delta_3\equiv\gamma_2^4\bmod 8$, so $\hat\delta_2\hat\delta_3=\square$,
\item $\frac{1}{4}\xi \equiv \gamma_2^2\gamma_3^2\bmod 4$, so that $\xi= \square (1+4z')$,
\item $\frac{1}{2}\eta_1\equiv (\gamma_2-\gamma_3)^2 \equiv \frac{1}{4}\ell_1^2 \bmod 4, \text{ so that }  \eta_1 =2\ell_1^2\square(1+4z'')$
\end{itemize}
for some $z, z', z''\in\cO_K$. Moreover
\begin{itemize}
\item $\ell_1^2 = \square (1+4z)$ for some $z\in\cO_K$ by Lemma \ref{lem:extensionsQ2}(i), and
\item $\frac{\ell_1\ell_2\ell_3}{\Delta}\!=\!\square(1\!+\!4z)$ for some $z\in\cO_K$ by the final paragraph of the proof of Proposition~\ref{prop:2ordinary}(ii).
\end{itemize}

Recall from Lemma \ref{lem:extensionsQ2}(ii) that if $a=1+4z$ for $z\in \cO_K$ and $b\in K^\times$ has even valuation, then $(a,b)=1$. We can thus simplify the Hilbert symbols defining $E_{C/K}$ as follows:
$$
\begin{array}{ll}
 (\xi, -\delta_1\hat\delta_2\hat\delta_3)&=1,\cr
 (\eta_2\eta_3,-\delta_2\delta_3\hat\delta_2\hat\delta_3)&=1,\cr
 (c, \delta_1\delta_2\delta_3\hat\delta_2\hat\delta_3)&=1,\cr
 (\hat\delta_2\hat\delta_3,-2)&=1,\cr
 (\eta_1, -\delta_2\delta_3\hat\delta_2\hat\delta_3)&=(2\ell_1^2, -\delta_2\delta_3)=(2,-\delta_2\delta_3),\cr
 (\hat\delta_1,-\frac{\ell_1}{\Delta})(\ell_1^2,-\ell_2\ell_3)\!\!\!\!\!&= (\ell_1^2\hat\delta_1,-\frac{\ell_1}{\Delta})(\ell_1^2,\frac{\ell_1\ell_2\ell_3}{\Delta})=(\bigl(\frac{\ell_1}{\Delta}\bigr)^2\hat\delta_1\Delta^2,-\frac{\ell_1}{\Delta})=1.
\end{array}
$$
Hence $E_{C/K}=$
$$
 (2,\ell_1^2\delta_2\delta_3)(\delta_2+\delta_3,-\ell_1^2\delta_2\delta_3)(\delta_2\eta_2+\delta_3\eta_3,-\ell_1^2\eta_2\eta_3\delta_2\delta_3)(\hat\delta_2\eta_3+\hat\delta_3\eta_2,-\ell_1^2\hat\delta_2\hat\delta_3\eta_2\eta_3).
$$

Suppose that $\Frob_K$ does not swap the sets $\{\alpha_2, \beta_2\}$ and $\{\alpha_3, \beta_3\}$. In particular $\ell_1\in K$, so $\ell_1^2=\square$.
Also $\gamma_i, \delta_i, \eta_i, \hat\delta_i\in K$, as they are all fixed by $\Frob_K$. 
By Lemma \ref{lem:newHS}(1), 
$$
\begin{array}{rl}
(\delta_2+\delta_3,-\delta_2\delta_3)\!\!\!\!&=(\delta_2, \delta_3),\cr
(\delta_2\eta_2+\delta_3\eta_3,-\delta_2\delta_3\eta_2\eta_3)\!\!\!\!&=(\delta_2\eta_2, \delta_3\eta_3),\cr
(\hat\delta_2\eta_3+\hat\delta_3\eta_2,-\hat\delta_2\hat\delta_3\eta_2\eta_3)\!\!\!\!&=(\hat\delta_2\eta_2, \hat\delta_3\eta_3).
\end{array}
$$
As $\frac{1}{4}\hat\delta_2\equiv\gamma_3^4\bmod 8$ we have $\hat\delta_2=\square$, and as $\frac{1}{2}\eta_2\equiv\gamma_2^2\bmod 4$ we have $\eta_2=2\square(1+4z)$ for some $z\in\cO_K$, and similarly for $\hat\delta_3$ and $\eta_3$. Thus
$$
E_{C/K}=(2,\delta_2\delta_3)(\delta_2,\delta_3)(\delta_2\eta_2,\delta_3\eta_3)(\eta_2, \eta_3)=(2,\delta_2\delta_3)(\delta_2,\eta_3)(\delta_3,\eta_2)=
$$
$$
=(2,\delta_2\delta_3)(\delta_2,2)(\delta_3,2)=1.
$$

Suppose now that $\Frob_K$ swaps the sets $\{\alpha_2, \beta_2\}$ and $\{\alpha_3, \beta_3\}$, so that it interchanges $\delta_2\leftrightarrow\delta_3$, $\eta_2\leftrightarrow\eta_3$, $\hat\delta_2\leftrightarrow\hat\delta_3$.
Since $\ell_1, (\alpha_2\!-\!\beta_2)(\alpha_3\!-\!\beta_3)\in K^{nr}$ and all units are norms from quadratic unramified extensions, $(\ell_1^2,u)=(\delta_2\delta_3,u)=1$ for all $u\in\cO_K^\times$. Recall that by hypothesis  $\frac{\delta_2+\delta_3}{16}$, $\frac{\delta_2\eta_2+\delta_3\eta_3}{32}$,$ \frac{\hat\delta_2\eta_3+\hat\delta_3\eta_2}{8}\in \cO_K^\times$, and that $\hat\delta_2\hat\delta_3=\square$, which simplifies the expression for $E_{C/K}$:
$$
\begin{array}{ll}
E_{C/K}\!\!\!\!\!&=(2, \ell_1^2\delta_2\delta_3)(\frac{\delta_2\!+\!\delta_3}{16},-\ell_1^2\delta_2\delta_3)
(\frac{\delta_2\eta_2\!+\!\delta_3\eta_3}{32},-\delta_2\delta_3\eta_2\eta_3)\cdot\cr
&\phantom{=}\cdot(2,-\delta_2\delta_3\eta_2\eta_3)
(\frac{\hat\delta_2\eta_3\!+\!\hat\delta_3\eta_2}{8},-\eta_2\eta_3)(2,-\eta_2\eta_3)\cr\cr
 &= (2, \ell_1^2)(\frac{\delta_2+\delta_3}{16}\cdot\frac{\delta_2\eta_2+\delta_3\eta_3}{32}\cdot\frac{\hat\delta_2\eta_3+\hat\delta_3\eta_2}{8},-1).
\end{array}
$$
Write $\alpha_2-\beta_2=4u_2$ and $\alpha_3-\beta_3=4u_3$, so that $v(u_2)=v(u_3)=0$ and $\Frob_K$ interchanges $u_2^2\leftrightarrow u_3^2$.
Then
$\frac{1}{16}(\delta_2+\delta_3)=u_2^2+u_3^2$,
$\frac{1}{32}(\delta_2\eta_2+\delta_3\eta_3)\equiv u_2^2 \gamma_2^2+u_3^2\gamma_3^2\bmod 4$ and
$\frac{1}{8}(\hat\delta_2\eta_3+\hat\delta_3\eta_2)\equiv \gamma_2^6+\gamma_3^6\bmod 4$, and all three are units. 
By Lemma \ref{lem:extensionsQ2}(ii) $(1+4z,-1)=1$ for all $z\in\cO_K$, so
$$
E_{C/K}=(2,\ell_1^2)((u_2^2+u_3^2)(u_2^2 \gamma_2^2+u_3^2\gamma_3^2)(\gamma_2^6+\gamma_3^6),-1).
$$
By Lemma \ref{lem:extensionsQ2}(iii), $(\gamma_2^6+\gamma_3^6,-1)=(2,\ell_1^2)$, as $\gamma_2^3$ lies in the quadratic unramified extension $K(\ell_1)$ of $K$. If $\Frob_K^2 u_2=u_2$ then similarly $(u_2^2+u_3^2,-1)=(2,\ell_1^2)$ and $(u_2^2 \gamma_2^2+u_3^2\gamma_3^2,-1)=(2,\ell_1^2)$, so that $E_{C/K}=1$. Otherwise $\Frob_K(u_2u_3(\gamma_2-\gamma_3))=u_3(-u_2)(\gamma_3-\gamma_2)=u_2u_3(\gamma_2-\gamma_3)$, so that
$$
(u_2^2 \gamma_2^2+u_3^2\gamma_3^2)(u_2^2+u_3^2) = (u_2^2\gamma_2+u_3^2\gamma_3)^2+u_2^2u_3^2(\gamma_2-\gamma_3)^2
$$
is a sum of two $K$-rational squares and hence a norm from $K(\sqrt{-1})$. This again implies that $((u_2^2 \gamma_2^2+u_3^2\gamma_3^2)(u_2^2+u_3^2),-1)=1$, so that $E_{C/K}=1$.
\end{proof}

\subsection{Special families $\mathcal{F}$, $\mathcal{F}_{C2D4}$}

It will be convenient to have a family of curves defined over $\Q_2$ with good ordinary reduction and all roots defined over $\Q_2$. Theorem \ref{thm:2orderrorterm} excludes these. 
\begin{notation}
\label{def:familyf}
For a finite extension $K/\Q_2$, we write $\mathcal{F}$ for the family of genus~2 curves $C:y^2\!=\!cf(x)$ with $c\!\equiv\! -1\bmod 2^3$ and for which $f(x)$ factors as $\prod (x-\alpha_i)(x-\beta_i)$ with
$$
 \alpha_1\equiv -5,\quad
 \beta_1\equiv 5,\quad
 \alpha_2\equiv -4,\quad
 \beta_2\equiv -12,\quad
 \alpha_3\equiv 2,\quad
 \beta_3\equiv-6,
$$
the congruences taken modulo $2^8$.

We write $\mathcal{F}_{C2D4}$ for the family of C2D4 curves that satisfy these congruences with respect to their given C2D4 structure. (In other words, $C$ satisfies the above congruences ``with its given ordering of roots''.)
\end{notation}


\begin{theorem}\label{thm:errorterm2}
Let $K/\Q_2$ be a finite extension and $C\in\mathcal{F}_{C2D4}$ a C2D4 curve. Then
\begin{enumerate}
\item $C/K$ has good ordinary reduction,
\item the Richelot dual curve $\widehat{C}$ has good reduction,
\item the kernel of the Richelot isogeny is contained in the kernel of reduction on $\Jac C(\bar{K})$,
\item $E_{C/K}=1$,
\item conjecture \ref{conj:local} holds for $C/K$.
\end{enumerate}
\end{theorem}

\begin{proof}
(1) The transformation $x_2=\frac{x}{2}$, $y_2=\frac{y}{8}+\frac{x^2}{8}+\frac{x}{4}$ gives a model $C'$ whose reduction is $y_2^2+x_2^2y_2+x_2y_2=x_2^6+x_2^4+1$. Over $\F_2$, this curve is smooth with local polynomial $1-T^2+4T^4$, and is therefore good ordinary over every extension of $\F_2$.

(2) Computing the model for $\widehat{C}$ as in Definition \ref{def:richelotdual} and applying the substitutions $x_2\!=\!\frac{x}{2}$ and $y_2\!=\!\frac{y}{16}\!+\!\frac{x^2}{8}\!+\!\frac{x}{4}$ gives a model that reduces to $y_2^2\! +\! x_2^2y_2\! +\! y_2x_2\! =\! x_2^5\! +\! x_2^3\! + \!x_2^2\! +\! x_2$.

(3) Using the change of variables in the proof of (1), one checks that the 2-torsion points on $\Jac C$ given by $[(\alpha_i,0),(\beta_i,0)]$ map to points on $\Jac C'$ that reduce to 0. 

(4) Direct computation shows that, up to multiplying by elements that are $1 \bmod 8$, $C$ has $\DG \!=\! 84$, $\ell_1 \!=\! -12$, $\l_2\!=\!-4$, $\l_3\!=\!-16$, $\dgu\!=\!25$, $\dgd \!=\! 64$, $\dgt\!=\!64$, $\dlu \!=\! -\frac{1}{7}$, $\dld \!=\! -924$, $\dlt \!=\! -4284$, $\eta_1\!=\! 8$, $\eta_2 \!=\! 110$, $\eta_3 \!=\! -10$, $\xi \!=\! 10196$, $\delta_2\!+\!\delta_3\!=\!128$, $\delta_2\eta_2 \!+\!\delta_3\eta_3\!=\!6400$, $\hat\delta_2\eta_3\!+\!\hat\delta_3\eta_2\!=\!-419160$. As elements that are $1 \bmod 8$ are squares, this gives $E_{C/K}=1$.

(5) $C$ and $C'$ are not deficient by (1, 2), so by Theorem \ref{first cokernel theorem}, $\lambda_{C/K}=1$. By (1) $w_{C/K}=1$, and by (4) $E_{C/K}=1$, which proves the result. 
\end{proof}

\section{Deforming C2D4 curves}\label{s:deformation}

As explained in \S\ref{ss:overview}, we will not attempt to prove other cases of Conjecture \ref{conj:local} by direct computation, as there are several hundred possible cluster pictures corresponding to semistable C2D4 curves. Instead, we will exploit the fact that we already have a good supply of C2D4 curves over number fields for which we have proved the 2-parity conjecture (through Conjecture \ref{conj:local} and Theorem \ref{thm:introlocalglobal}) and use Lemma \ref{lem:lastplace}. For this we will need to be able to approximate C2D4 curves over local fields by curves over number fields, that are well-behaved at all other places. In this section we prove two results that will let us do this (see Theorems \ref{thm:C2D4Lift}, \ref{thm:doublelift}). Roughly speaking they say that:

$\bullet$ A C2D4 curve $C/K_v$ can be approximated by a curve $C'/K$ such that Conjecture \ref{conj:local} holds for $C/K_v$ if and only if it holds for $C'/K_v$, and moreover holds for $C'/K_w$ for all $v\neq w$;

$\bullet$ A curve with two C2D4 structures $C/K_v$ can be similarly approximated by $C/K$ admitting two C2D4 structures.

In \S \ref{s:locconj2} this will let us show that Conjecture \ref{conj:local} is independent of the choice of C2D4 model for a curve $C$, and, moreover, it holds with respect to one C2D4 structure if and only if it holds with respect to another (Theorems \ref{thm:anothermodel}, \ref{thm:changeisogeny}). These, in turn, will let us complete our proof of Theorem~\ref{thm:introlocal} on Conjecture \ref{conj:local} and deduce our main results on the 2-parity and parity conjectures~in~\S\ref{s:final}.

\subsection{Continuity of local invariants}

\begin{definition}\label{def:close}
Let $C, C'$ be two C2D4 curves over a local field of characteristic 0. We will say that $C$ and $C'$ are {\em $\epsilon$-close} if the leading coefficients and roots of their defining polynomials satisfy $|\alpha_i-\alpha_i'|, |\beta_i-\beta_i'|, |c-c'|<\epsilon$ for $i=1, 2, 3$. 

For curves with $\cP,\Delta,\eta_1,\cP',\Delta',\eta_1'\neq 0$ we will say that they are {\em arithmetically} close if 
$$
   w_{C/K}=w_{C'/K}, \quad \lambda_{C/K}=\lambda_{C'/K}
$$
and
$$
  \frac{\delta_1}{\delta_1'} \equiv
  \frac{\delta_2\delta_3}{\delta_2'\delta_3'} \equiv
  \frac{\delta_2\!+\!\delta_3}{\delta_2'\!+\!\delta_3'} \equiv
  \frac{\eta_2\eta_3}{\eta_2'\eta_3'}\equiv
  \frac{\hat\delta_2\hat\delta_3}{\hat\delta_2'\hat\delta_3'} \equiv 
  \frac{\eta_1}{\eta_1'} \equiv
  \frac{\xi}{\xi'} \equiv 
  \frac{\ell_1^2}{\ell_1'^2} \equiv
  \frac{\ell_1/\Delta}{\ell_1'/\Delta'} \equiv 
  $$
  $$
  \equiv \frac{c}{c'} \equiv
  \frac{\delta_2\eta_2\!+\!\delta_3\eta_3}{\delta_2'\eta_2'\!+\!\delta_3'\eta_3'} \equiv
  \frac{\hat\delta_2\eta_3\!+\!\hat\delta_3\eta_2}{\hat\delta_2'\eta_3'\!+\!\hat\delta_3'\eta_2'} \equiv
  1 \bmod \pi_K
$$
for $K/\Q_p$ finite with $p$ odd; the congruence taken mod $4\pi_K$ instead for $p\!=\! 2$; for $K\iso \R$ we require that all these ratios are positive, and have no requirement for $K\iso \C$.
\end{definition}

\begin{lemma}\label{continuity}
For C2D4 curves $C$ over a local field $K$ of characteristic 0 the invariants
$\delta_1$, $\delta_2$, $\delta_3$, $\eta_1$, $\eta_2$, $\eta_3$, $\hat\delta_2$, $\hat\delta_3$, $\xi$, $\ell_1^2$, $\Delta$ and $w_{C/K}$ 
are continuous in the roots $\alpha_1, \beta_1, \alpha_2, \beta_2, \alpha_3, \beta_3,$ and the leading coefficient~$c$.
If $\Delta\neq 0$, so are $\frac{\ell_1}{\Delta}$ and $\lambda_{C/K}$.
\end{lemma}

\begin{proof}
This is clear for all the invariants except possibly $w$ and $\lambda$, as they are rational functions in the roots and the leading coefficient $c$. 
For archimedean $K$, $w\!=\!1$, while $\lambda$ is a locally constant function in the roots and $c$ by Lemmata \ref{lem:kercoker}, \ref{pr:n} and the first paragraph of the proof of Theorem~\ref{th:localconjectureinfinite}.
For $K$ non-archimedean, the special fibre of the minimal regular model of $C/K^{nr}$ (with Frobenius action) is locally constant, and hence so is the deficiency term $\mu_{C/K}$ (cf. \cite{m2d2} Lemma 12.3) and the local Tamagawa number $c_{\Jac C/K}$ (cf. \cite{BosLor} Thm.\ 2.3); 
the coefficients of the equation of the dual curve $\widehat{C}$ are continuous in the roots and in $c$, so $\lambda$ is also locally constant.
The Galois representation $T_\ell(\Jac C)\iso H^1_{et}(C/K,\Z_\ell)(1)$ is also locally constant (\cite{Kisin} p.569), and hence so is the root number $w_{C/K}$.
\end{proof}

%
%

\begin{lemma}\label{lem:cts2}
Let $C$ be a C2D4 curve over a local field of characteristic 0 with $\cP,\Delta,\eta_1\neq 0$.
Then there exists $\epsilon>0$ such that every C2D4 curve that is $\epsilon$-close to $C$ is arithmetically close to $C$.
\end{lemma}

\begin{proof}
Clear from Lemma \ref{continuity}.
\end{proof}

\begin{lemma}\label{lem:cts3}
Let $C, C'$ be arithmetically close C2D4 curves over a local field of characteristic 0 with $\cP$, $\Delta$, $\eta_1$, $\cP'$, $\Delta'$, $\eta_1'\neq 0$. Then Conjecture \ref{conj:local} holds for C if and only if it holds for $C'$.
\end{lemma}

\begin{proof}
Clear as $\lambda$ and $w$ are the same, and the terms in the Hilbert symbols in $E_{C/K}$ change by squares.
\end{proof}

\begin{remark}
If $t$ is sufficiently close to 0, then $C_t$ is $\epsilon$-close to $C$. More generally, if the coefficients of $m(z)=\frac{\aa z+\bb}{\cc z+\dd}$ are sufficiently close to $\aa\!=\!\dd\!=\!1, \bb\!=\!\cc\!=\!0$, then  $C_m$ is $\epsilon$-close to~$C$ (see Definitions \ref{def:Cm} and \ref{def:Mt}).
\end{remark}

\subsection{D4 quartics}

To approximate C2D4 curves over $p$-adic fields by curves over number fields $y^2=cf(x)$, we must ensure that the Galois group of $f(x)$ does not become too large. We will also want to ensure that the new curve is well-behaved at all other primes.

\begin{lemma}\label{le:D4Family} 
A separable monic quartic polynomial $q(x)$ over a field $K$ has $\Gal(q(x))\subseteq D_4$ if and only if its roots are of the form
$$
r_1 = z_1+z_2\sqrt{a} + z_3\sqrt{b\!+\!d\sqrt{a}} + z_4\sqrt{a}\sqrt{b\!+\!d\sqrt{a}}, 
$$
$$
r_2 = z_1+z_2\sqrt{a} - z_3\sqrt{b\!+\!d\sqrt{a}} - z_4\sqrt{a}\sqrt{b\!+\!d\sqrt{a}},
$$
$$
r_3 = z_1-z_2\sqrt{a} - z_3\sqrt{b\!-\!d\sqrt{a}} + z_4\sqrt{a}\sqrt{b\!-\!d\sqrt{a}},
$$
$$
r_4 = z_1-z_2\sqrt{a} + z_3\sqrt{b\!-\!d\sqrt{a}} -z_4\sqrt{a}\sqrt{b\!-\!d\sqrt{a}},
$$
for some $a \in K^{\times}$ and $b,d,z_1,z_2,z_3,z_4 \in K$.
\end{lemma}
\begin{proof}
Clearly a quartic with such roots has Galois group $\subseteq D_4$.
For the converse, if the Galois group contains a 4-cycle the result is clear.
If it is contained in the Klein group $V_4$, the result follows from Lemma \ref{le:V4Family} below by taking $d=0$.
If it is either trivial or $C_2$ acting by a transposition, the result follows by taking $a=1$ and $d=b-1$.
\end{proof}

\begin{lemma}\label{le:notriple}
Let $K$ be a local field of odd residue characteristic and $q(x)$ as in Lemma \ref{le:D4Family}. If $a,b,d,z_1,z_2,z_3,z_4 \in \mathcal{O}_K$ and either 
$16(b^2-ad^2)(z_3^2-az_4^2)^2$ 
or 
$2az_2^2 -bz_3^2-2adz_3z_4-abz_4^2$
is a unit then $\bar{q}(x)$ does not have roots of multiplicity $\ge 3$. 
\end{lemma}

\begin{proof}
One checks that 
$
(r_1-r_2)^2(r_3-r_4)^2 = 16(b^2-ad^2)(z_3^2-az_4^2)^2
$
and 
$$
\frac{1}{4}(r_1-r_2)^2(r_3-r_4)^2+ (r_1-r_3)(r_1-r_4)(r_2-r_3)(r_2-r_4) =4(2az_2^2 -bz_3^2-2adz_3z_4-abz_4^2)^2.
$$ 
\end{proof}

\begin{remark}\label{re:I22Roots}
$(r_1-r_3)(r_2-r_4)+(r_2-r_3)(r_1-r_4) = 4(2az_2^2 -bz_3^2-2adz_3z_4-abz_4^2)$.
\end{remark}

\begin{lemma}\label{le:D4Lift}
Let $K$ be a number field and $S$ a finite set of places of $K$ including all primes above 2, but excluding at least one infinite place. For each $v \in S$, let $q_v(x) \in \mathcal{O}_{K_v}[x]$ (or $\in K_v[x]$ for $v|\infty$) be a monic quartic with $\Gal(q_v(x)) \subseteq D_4$. There exists a monic quartic $q(x) \in \mathcal{O}_K[x]$ with $\Gal(q(x)) \subseteq D_4$ such that 

i) for each $v\in S$, 
the roots of $q(x)$ are arbitrarily close to those of $q_v(x)$ (with respect to an ordering that respects the $D_4$-action),

ii) for all $v \notin S$, $q(x) \bmod v$ has no roots of multiplicity $\ge 3$. 
\end{lemma}

\begin{proof}
For each $v \in S$, write the roots of $q_v$ as in Lemma \ref{le:D4Family} with parameters $a_v, b_v, d_v, z_{1,v}, z_{2,v},$ $ z_{3,v}, z_{4,v} \in K_v$. Note that any monic polynomial $q(x) \in K[x]$ in the form of Lemma \ref{le:D4Family} whose parameters are $v$-adically close to those of $q_v(x)$ satisfies (i).

Use strong approximation (and the infinite place outside $S$) to choose $a,b,d,z_1,z_2,z_3,z_4 \in K$ that lie in $\mathcal{O}_{K_v}$ for primes $v \notin S$ as follows. 
First choose any $b,d,z_1,z_3,z_4$ that are $v$-adically close to $b_v, d_v, z_{1,v}, z_{3,v}, z_{4,v}$ for $v \in S$. 
Now choose $a$ that is $v$-adically close to $a_v$ for $v\in S$ with $\text{gcd}(a,bz_3)$ supported on $S$; this ensures that if a prime $w \notin S$ divides $(b^2-ad^2)(z_3^2-az_4^2)$ then $w \nmid a$. 
Finally, choose $z_2$ that is $v$-adically close to $z_{2,v}$ for $v\in S$ such that for primes $w \notin S$ that divide $(b^2-ad^2)(z_3^2-az_4^2)$ the expression $2az_2^2 -bz_3^2-2az_3z_4-abz_4^2$ is a $w$-adic unit.
By Lemma \ref{le:notriple} $q(x)$ now satisfies~(ii). 

Finally, note that as $q(X)$ lies in $\mathcal{O}_{K_v}[x]$ both for $v \in S$ (by (i)) and $v \notin S$ (by construction), it lies in $\mathcal{O}_K[x]$.
\end{proof}


\subsection{V4 quartics}

We will need a similar result for quartics with Galois group $V_4$. Recall that for us $V_4$ consists of double transpositions. 

\begin{lemma}\label{le:V4Family} 
A separable monic quartic polynomial $q(x)$ over a field $K$ has $\Gal(q(x))\subseteq V_4$ if and only if its roots are of the form
$$
r_1 = z_1+ z_2\sqrt{a} + z_3\sqrt{b} + z_4\sqrt{ab}, \qquad 
r_2 = z_1 - z_2\sqrt{a} + z_3\sqrt{b} - z_4\sqrt{ab}, 
$$
$$
r_3 = z_1+ z_2\sqrt{a} - z_3\sqrt{b} - z_4\sqrt{ab}, \qquad 
r_4 = z_1 - z_2\sqrt{a} - z_3\sqrt{b} + z_4\sqrt{ab}, 
$$
for some $a, b \in K^\times$ and $z_1, z_2, z_3, z_4 \in K$.
\end{lemma}
\begin{proof}
Clearly a quartic with such roots has Galois group $\subseteq V_4$. Conversely, the splitting field of a quartic with $\Gal q(x)\subseteq V_4$ is of the form $K(\sqrt{a}, \sqrt{b})$ for some $a, b\in K^\times$; the $z_i$ are then found by solving the given linear system of equations for the $r_i$. 
\end{proof}

\begin{lemma}\label{le:notriplev4}
Let $K$ be a local field of odd residue characteristic and $q(x)$ as in Lemma \ref{le:V4Family}. 
If $a,b,z_1,z_2,z_3,z_4 \in \mathcal{O}_K$ and either 
$a(z_2^2-bz_4^2)$ or 
$b(z_3^2-bz_4^2)$ 
is a unit then $\bar{q}(x)$ does not have roots of multiplicity $\ge 3$. 
\end{lemma}

\begin{proof}
One checks that
$(r_1\!-\!r_2)(r_3\!-\!r_4) = 4a(z_2^2\!-\!bz_4^2)$
and
$(r_1\!-\!r_3)(r_2\!-\!r_4) = 4b(z_3^2\!-\!az_4^2)$.
\end{proof}

\begin{lemma}\label{le:V4Lift}
Let $K$ be a number field and $S$ a finite set of places of $K$ including all primes above 2, but excluding at least one infinite place. For each $v \in S$, let $q_v(x) \in \mathcal{O}_{K_v}[x]$ (or $\in K_v[x]$ for $v|\infty$) be a monic quartic with $\Gal(q_v(x)) \subseteq V_4$. There exists a monic quartic $q(x) \in \mathcal{O}_K[x]$ with $\Gal(q(x)) \subseteq V_4$ such that 

i) for each $v\in S$, 
the roots of $q(x)$ are arbitrarily close to those of $q_v(x)$,

ii) for all $v \notin S$, $q(x) \bmod v$ has no roots of multiplicity $\ge 3$. 
\end{lemma}

\begin{proof}
The proof is the same as for Lemma \ref{le:D4Lift}, except that the parameters $a,b,z_1,z_2,z_3,z_4\in K$ (lying in $\cO_{K_v}$ for $v\not\in S$) are chosen as follows.
First choose $z_1$ and $a$ to be $v$-adically close to $z_{1,v}$ and $a_v$ for $v\in S$.
Choose $z_3$ and $b$ to be $v$-adically close to $z_{3,v}$ and $b_v$ for $v\in S$ such that $\text{gcd}(a,bz_3)$ is supported on $S$ --- this ensures that for primes $v\notin S$ that divide $a$,  $b(z_3^2-az_4^2)$ is a unit.
Choose $z_2$ to be $v$-adically close to $z_{2,v}$ for $v\in S$ such that $\text{gcd}(b,z_2)$ is supported on $S$ and $az_2^2\neq bz_3^2$ --- in particular, this ensures that for primes $v\notin S$ that divide $b$,  $a(z_2^2-bz_4^2)$ is a unit.
Finally choose $z_4$ to be $v$-adically close to $z_{4,v}$ for $v\in S$ such that $\text{gcd}(az_2^2-bz_3^2,z_4)$ is supported on $S$ --- this ensures that for primes $v\not\in S$ that do not divide $ab$ either $(z_2^2-bz_4^2)$ or $(z_3^2-az_4^2)$ is a unit.
By Lemma \ref{le:notriplev4} $q(x)$ now satisfies~(ii). 
\end{proof}


\subsection{Glueing a quadratic to a polynomial}

\begin{lemma}\label{le:gluequadratic} 
Let $K$ be a number field and $q(x)\in\cO_K[x]$ be a separable monic polynomial.
Let $S$ be a finite set of places of $K$ including all primes above 2 and all primes with residue field of size less than $\deg q(x)+2$, but excluding at least one infinite place. 
For each $v \in S$, let $h_v(x)\in \cO_{K_v}[x]$ (or $\in K_v[x]$ for $v|\infty$) be a monic quadratic. 
There exists a monic quadratic $h(x) \in \cO_K[x]$ such that 
\begin{enumerate}[leftmargin=*]
\item for each $v \in S$ the coefficients of $h(x)$ are arbitrarily close to those of $h_v(x)$, and
\item for each prime $v \notin S$, $h(x)q(x) \bmod v$ has either (a) no repeated roots or exactly one double root in the residue field $\overline{\F}_v$, or (b) all its repeated roots coming from those of $q(x) \bmod v$.
\end{enumerate}
\end{lemma}

\begin{proof}
Write $h_v(x)=x^2+a_v x + b_v$. Using strong approximation (and the infinite place outside of $S$) pick $a\in\cO_K$ that is $v$-adically close to the $a_v$ for all $v\in S$ and such that $P=\prod (a+r+r')\neq 0$, the product taken over all pairs of roots of $q(x)$ (including repeats).

Let $T$ be the set of primes outside $S$ that divide $P\cdot$ Disc$(q(x))$.

Using strong approximation now pick $b\in \cO_K$ so that
\begin{enumerate}
\item[(i)] $b$ is $v$-adically close to $b_v$ for all $v \in S$, and
\item[(ii)] $x^2\! +\! ax\!+\!b \bmod v$ is separable and coprime to $q(x) \bmod v$ for all~$v\in T$.
\end{enumerate} 
The fact that the residue field at $v\notin S$ has size at least $\deg q(x)+2$ ensures that for each $a \bmod v$ there is always a polynomial over $\F_v$ that satisfies (ii).

We can now take $t(x)=x^2+ax+b$. Indeed, condition (ii) ensures that
\begin{itemize}
\item if $q(x) \bmod v$ has a double root then the roots of $h(x) \bmod v$ are distinct from each other and from the roots of $q(x) \bmod v$;
\item the roots of $h(x) \bmod v$ cannot both coincide with roots of $q(x) \bmod v$ for any $v\notin S$: otherwise we would have $P\equiv 0 \bmod v$, so that $v\in T$, which contradicts (ii);
\item if $h(x)\bmod v$ has a double root (this would then be $\equiv-a/2 \bmod v$) for $v\not\in S$, then it does not coincide with a root of $q(x)\bmod v$.
\end{itemize}
\end{proof}

\subsection{Approximating a curve}

\begin{lemma}\label{lem:smallrf}
Let $K$ be a finite extension of $\Q_p$ for an odd prime $p$. 
There are polynomials $f_1(x)$ and $f_2(x)$ of the form $f(x)=r(x)s(x)t(x)$ with $r(x), s(x), t(x)\in \cO_K[x]$ monic quadratic and $\Gal(\bar{K}/K)$ acting on the roots of $s(x)t(x)$ as a subgroup of $V_4$, such that
\begin{enumerate}
\item $f_1(x)$ has cluster picture \GR[D][D][D][D][D][D], 
and
\item $f_2(x)$ has cluster picture \CPTCV[A][B][C][A][B][C] and $v(\a_2\!+\!\b_2\!-\!\a_3\!-\!\b_3)=1$, 
\end{enumerate}
where \smash{\raise4pt\hbox{\clusterpicture\Root[A]{1}{first}{r1};\endclusterpicture}}~\smash{\raise4pt\hbox{\clusterpicture\Root[A]{1}{first}{r1};\endclusterpicture}}
, \smash{\raise4pt\hbox{\clusterpicture\Root[B]{1}{first}{r1};\endclusterpicture}}~\smash{\raise4pt\hbox{\clusterpicture\Root[B]{1}{first}{r1};\endclusterpicture}} $(\a_2,\b_2)$ 
 and \smash{\raise4pt\hbox{\clusterpicture\Root[C]{1}{first}{r1};\endclusterpicture}}~\smash{\raise4pt\hbox{\clusterpicture\Root[C]{1}{first}{r1};\endclusterpicture}} $(\a_3,\b_3)$ 
denote the roots of $r, s$ and $t$, respectively.
\end{lemma}
\begin{proof}Since $s(x), t(x)\in K[x]$, the constraint on $\Gal(\bar{K}/K)$ means that its elements will either act trivially on the  roots of $s(x)t(x)$ or simultaneously swap the roots of $s(x)$ and of $t(x)$.

(1) Take $r(x)=x(x-1)$ and $s(x)$, $t(x)$ quadratics whose images over the residue field $k(x)$ are distinct and irreducible (these exist even when $k=\F_3$).

(2) Take $r(x) = x(x-1)$, $s(x)=(x-\pi)(x-\pi-1)$, $t(x)=(x+\pi)(x+\pi-1)$, where $\pi=\pi_K$ is a uniformiser.
\end{proof}

\begin{theorem}\label{thm:C2D4Lift}
Let $K$ be a number field. 
Let $v$ be a place of $K$ and $C_v/K_v$ be a C2D4 curve with $\cP, \Delta, \eta_1\neq 0$. 
Let $v'\neq v$ be an archimedean place~of~$K$.

For every $\epsilon>0$, there is a C2D4 curve $C/K$ with $\cP, \Delta, \eta_1\neq 0$ such that \\
\noindent (i) $C/K_v$  and $C_v/K_v$ are $\epsilon$-close, \\ 
\noindent (ii) for each prime $w\!\neq\! v$ with $w|2$, the C2D4 curve $C/K_w$ lies in $\mathcal{F}_{C2D4}$,\\
\noindent (iii) for each prime $w\!\neq\! v$  with $w\nmid 2$ the curve $C/K_w$ is semistable with cluster picture

\begin{itemize}[leftmargin=*]
\item \GRND[D][D][D][D][D][D] or 
\GRUNND[D][D][D][D][D][D], or
\item \CPTCV[A][B][C][A][B][C] with $\ord_w(\l_1)=1$, or
\item \CPONNDZERO[D][D][D][D][D][D] \\with $\ord_w(\ell_1)=\ord_w(\ell_2)=\ord_w(\ell_3)=\ord_w(\eta_2)=\ord_w(\eta_3)=0$, or
\item \CPTNNDZERO[B][B][C][C][A][A] 
with $\ord_w(\ell_1)=\ord_w(\ell_2)=\ord_w(\ell_3)=0$, or
\item \CPTNNMZERO[B][C][B][C][A][A] \\with $\ord_w(\ell_1)\!=\!\min(n,m)$, $\ord_w(\ell_2)\!=\!\ord_w(\ell_3)\!=\!\ord_w(\eta_2)\!=\!\ord_w(\eta_3)\!=\!0$.
\end{itemize}
\noindent (iv) for each real place $w\neq v, v'$, the curve $C/K_w$ has picture \smash{\raise4pt\hbox{\CPRNF[A][A][B][B][C][C]}}
.
\end{theorem}

\begin{proof}
Note that being $\epsilon$-close to $C_v/K_v$ for sufficiently small $\epsilon$ guarantees that $\cP, \Delta, \eta_1\neq 0$ (Lemma \ref{continuity}), so this condition will be automatic.

Write $C_v/K_v$ as $y^2 = c_v \cdot h_v(x)q_v(x)$ with $h_v(x)$ a monic quadratic and $q_v(x)$ a monic quartic with $\Gal(q_v(x))\le D_4$ given by the C2D4 structure.
In the case that $v$ is non-archimedean, we may assume that $h_v(x), q_v(x)\in\cO_{K_v}[x]$: otherwise scale $x$ by a suitable totally positive element of $K$ whose only prime factor is $v$ (this exists as $v$ has finite order in the class group).

Let $S$ be the set consisting of $v$, 
all real places other than $v'$, 
primes above 2 and
all primes with residue field of size $<23$.
For $w\in  S\setminus\{v\}$ define C2D4 curves $C_w^{spc}: y^2=c_w h_w(x)q_w(x)$ over $K_w$ for quadratic $h_w(x)$ and quartic $q_w(x)$ as follows:
\begin{itemize}[leftmargin=*]
\item For $w|\infty$ let $c_w=-1$ and $h_w(x)=x(x-1)$, $q_w(x)=\prod_{i=2}^5 (x-i)$, so $C_w^{spc}$ has picture  \smash{\raise4pt\hbox{\CPRNF[A][A][B][B][C][C]}};

\item For $w| 2$ let $c_w=1$ and $h_w(x)=r(x)$, $q_w(x)=s(x)t(x)$ for $r, s, t$ given by the roots in Definition \ref{def:familyf}, so $C_w^{spc}\in\mathcal{F}_{C2D4}$;

\item For $w\nmid 2\infty$ let $c_w=1$ and $h_w(x)=r(x)$, $q_w(x)=s(x)t(x)$ for $r, s, t$ given by Lemma~\ref{lem:smallrf}(i); in particular $C_w^{spc}$ has picture \GR[D][D][D][D][D][D]
.
\end{itemize}

Pick $c\! \in\! K$ such that $c\! =\! c_w\!\cdot\! \square \in K_w^\times$ for all $w\! \in\! S$. 
Let $S'$ be the set of primes $w\! \notin\! S$ such that $\ord_w(c)$ is odd.
For these primes define $C_w^{spc}: y^2\!=\!c h_w(x)q_w(x)$ by setting $h_w(x)\!=\!r(x)$ and $q_w(x)\!=\!s(x)t(x)$ for $r, s, t$ given in Lemma \ref{lem:smallrf}(ii); these curves have cluster picture \CPTCV[A][B][C][A][B][C] with $\ord_w(\l_1)=1$.

Using Lemmata \ref{le:D4Lift} and \ref{le:gluequadratic}, construct a monic $D_4$ quartic $q(x) \in \cO_K[x]$ and monic $C_2$ quadratic $h(x)\in \cO_K[x]$ such that the C2D4 curve $C:y^2 = ch(x)q(x)$ 
\begin{enumerate}[leftmargin=*]
\item is $\frac{\epsilon}{2}$-close to $C_v/K_v$,
\item is close to $C^{spc}_w/K_w$ for $w \in S\cup S'\setminus\{v\}$, 
\item has picture \GRNDZERO[D][D][D][D][D][D], \CPONNDZERO[D][D][D][D][D][D], \CPTNNDZERO[B][B][C][C][A][A] or \CPTNNDZERO[B][C][B][C][A][A] 
 for $w\!\notin\!S\!\cup\!S'$.
\end{enumerate}
By the semistability criterion (Theorem \ref{th:ss}), $C/K_w$ is semistable at all $w\neq v$ with $w\nmid 2$.

It remains to ensure that the invariants listed under (iii) are units. Center the curve $C$ by a substitution $x\mapsto x\!+\!\lambda$ to give a curve $C'$. Using strong approximation and Theorem \ref{thm:mobiusoverbalance} pick $t\in K$ such that (a) for $w \in S$, $t$ is $w$-adically close to 0, and (b) for finite $w\notin S$ where $C$ has cluster picture \CPONND[D][D][D][D][D][D] 
or \CPTNND[B][B][C][C][A][A] 
or \CPTNND[B][C][B][C][A][A] 
or \CPTCND[A][B][C][A][B][C]
, $C_t$ has the same cluster picture and the invariants of $C_t$ specified in (iii) are units. Note that at all other places $C_t$ will have cluster picture \GRND[D][D][D][D][D][D] or 
\GRUNND[D][D][D][D][D][D] by Theorem \ref{th:genus2bible}. Finally, shifting $C_t$ back by $x\mapsto x\!-\!\lambda$ gives the required curve~$C''$.
\end{proof}

\subsection{Approximating a curve with two isogenies}

\begin{theorem}\label{thm:doublelift}
Let $K$ be a number field and $v$ a place of $K$. Let $C_v:y^2 =c_v f_v(x)$ be a  curve over $K_v$ that admits two C2D4 structures, $C^{(1)}_v$ and~ $C^{(2)}_v$, each of which has $\mathcal{P}, \Delta, \eta_1\neq 0$.
Let $v'\neq v$ be an archimedean place of $K$.
Suppose that one of the following two conditions holds:

\noindent (1)
$\bullet$ the second colouring is obtained from the first by relabelling colours,~and\\
\indent \, $\bullet$ $\Gal(f_v)$ preserves colours; or\\
\noindent (2)
$\bullet$ both colourings have the same ruby roots, and \\
\indent \, $\bullet$  $\Gal(f_v)$ acts on the sapphire and turquoise roots as a subgroup of $V_4$.

Then for every $\epsilon>0$ there is a curve $C:y^2 =c f(x)$ over $K$ that admits two C2D4 structures, $C^{(1)}$ and~ $C^{(2)}$, such that

\noindent 
(i) $C^{(1)}/K_v$  and $C^{(1)}_v/K_v$ are $\epsilon$-close, and  $C^{(2)}/K_v$  and $C^{(2)}_v/K_v$ are $\epsilon$-close,

\noindent (ii) for each prime $w\neq v$ with $w|2$, the curve  $C/K_w$ lies in $\mathcal{F}$,

\noindent (iii) for each prime $w\neq v$  with $w\nmid 2$ the curve  $C/K_w$ is semistable and, with respect to each of the two colourings, has cluster picture
\GRND[D][D][D][D][D][D]
, 
\GRUNND[D][D][D][D][D][D]
, 
\CPONNDZERO[D][D][D][D][D][D]
, 
\CPTNNDZERO[D][D][D][D][A][A] 
or 
\CPTCNDZERO[A][D][D][A][D][D]
.

\noindent (iv) for each real place $w\neq v, v'$, the curve $C/K_w$ has picture
 \smash{\raise4pt\hbox{\CPRNF[A][A][B][B][C][C]}}
 ,
 \smash{\raise4pt\hbox{\CPRNF[A][A][B][C][B][C]}}
 ,
 \smash{\raise4pt\hbox{\CPRNF[A][A][B][C][C][B]}}
 ,
 \smash{\raise4pt\hbox{\CPRNF[B][B][A][A][C][C]}} 
 or
 \smash{\raise4pt\hbox{\CPRNF[C][C][B][B][A][A]}}
 , 
with respect to both colourings.
\end{theorem}

\begin{proof}

The condition on the colourings and the Galois action in (1) is equivalent to saying that $f_v(x)$ factorises into three quadratics $a_v(x) b_v(x) d_v(x)$ over $K_v$ and that (for both C2D4 structures) the roots of each quadratic are monochromatic. The condition in (2) is equivalent to a factorisation into a quadratic $r_v(x)$ with the ruby roots and a quartic $q_v(x)$ whose Galois group is contained in $V_4$. Thus to prove the theorem it will suffice to construct $C/K$ so that 

\begin{itemize}[leftmargin=*]
\item it satisfies (ii), (iv), and (iii), avoiding the cluster picture \CPTNND[D][D][D][D][A][A] 
in case (1), and
\item in case (1), $f(x)$ admits a factorisation into three quadratics $a(x)b(x)d(x)$ over $K$ and $c, a(x)$, $b(x), d(x)$ are $v$-adically close to  $c_v, a_v(x), b_v(x), d_v(x)$,
\item in case (2), $f(x)$ admits a factorisation over $K$ into a quadratic $r(x)$ and a quartic $q(x)$ with Galois group inside $V_4$, and $c, r(x), q(x)$ are $v$-adically close to  $c_v, r_v(x), q_v(x)$,
\end{itemize}
Indeed, such a curve will automatically admit two C2D4 structures that satisfy (i).

The construction of $C/K$ follows exactly as in the proof of Theorem \ref{thm:C2D4Lift}, except that the use of Lemmata \ref{le:D4Lift} and \ref{le:gluequadratic} in the penultimate paragraph is replaced by two applications of Lemma \ref{le:gluequadratic} in case (1) and by Lemmata \ref{le:V4Lift} and \ref{le:gluequadratic} in case (2), and that the step in the final paragraph is not relevant here.
\end{proof}

\subsection{Making the terms $\cP,\eta_1,\Delta\neq0$}

Recall that we will eventually need to address the special cases when $\cP, \eta_1,$ or $\Delta$ is 0. Here we record the methods to make small perturbations to the given C2D4 model to make these invariants non-zero.

\begin{lemma}\label{thm:mobiusPnonzero}
Let $K$ be a local field of characteristic 0 and $C/K$ a centered C2D4 curve. Then there is a $t_0 \in K$ arbitrarily close to 0 such that for all $t\in K$ sufficiently close to $t_0$ the model $C_t$ has $\mathcal{P}\neq0$.
\end{lemma}

\begin{proof}
It suffices to find one value of $t_0$ close to 0 such that $C_{t_0}$ has $\cP\neq 0$, since $\cP(C_t)$ is continuous as a function of $t$. By definition $\cP=\ell_1\ell_2\ell_3\eta_2\eta_3\xi(\delta_2+\delta_3)(\delta_2\eta_2+\delta_3\eta_3)(\hat{\delta}_2\eta_3+\hat{\delta}_3\eta_2)$. 
As the individual factors are rational functions in $t$, it sufficies to prove that none of them are identically zero for $t\in\bar{K}$.

Consider a M\"obius transformation of the form $m:z\mapsto \cc(\frac{1}{z-\aa}-\bb)$, where $\bb$ and $\cc$ are chosen so that $m(\alpha_1)=\alpha_1$ and $m(-\alpha_1)=-\alpha_1$. By Lemma \ref{le:Mobius}, $m=M_t$ for some $t\in\bar{K}$. If $\aa$ is chosen to be sufficiently close but not equal to $\alpha_2$, then $|M_t(\alpha_2)|>\!>|M_t(\pm\alpha_1)|, |M_t(\alpha_3)|, |M_t(\beta_2)|, |M_t(\beta_3)|$. This guarantees that $\ell_1, \ell_3,\eta_2, \delta_2+\delta_3, \delta_2\eta_2+\delta_3\eta_3\neq 0$ for $C_t$. 
Note also that taking $\aa$ close to $\alpha_2$ keeps $\frac{1}{\cc}(M_t(\beta_2)\pm M_t(\alpha_1))=\frac{\pm\alpha_1-\beta_2}{(\pm\alpha_1-\aa)(\beta_2-\aa)}$ away from 0, which ensures that $\hat{\delta}_2\eta_3+\hat{\delta}_3\eta_2 = \alpha_2^4(\beta_2-\alpha_1)(\beta_2+\alpha_1) + \mathcal{O}(\alpha_2^3)$ also becomes non-zero for $C_t$.

Similarly, picking $\aa$ close to $\alpha_3$ shows that one can make $\ell_2, \eta_3\neq 0$. 

Finally, for $\xi$ embed $K\subset \C$ and take $\aa\approx\alpha_1$:  this makes $\xi(C_t)$ arbitrarily close to $-32\alpha_1^4$ in $\C$, and hence non-zero.
\end{proof}

\begin{remark}\label{rmk:DeltaAlwaysZero}
One cannot similarly change the model to make either $\eta_1$ or $\Delta$ non-zero. A computation shows that if $\eta_1(C)=0$ then $\eta_1(C_t)=0$ for all $t$. Moreover $\Delta(C)=0$ if and only if $\Delta(C_m)=0$ for every $m\in \GL_2(K)$, so, by Remark \ref{rmk:models}, the condition ``$\Delta\!=\!0$'' is independent of the choice of model.
\end{remark}

\begin{lemma}\label{le:perturbingI22}
Let $K$ be a local field of characteristic 0, and $C:y^2=cf(x)$ a C2D4 curve over $K$ with $\cP, \Delta\neq 0$. 

\noindent (i)
For every $\epsilon >0$ there exists another $C2D4$ curve $C':y^2=c g(x)$ over $K$ whose roots satisfy $|\alpha_i\!-\!\alpha_i'|, |\beta_i\!-\!\beta_i'|<\epsilon$ and which has $\eta_1(C')\ne 0$.

\noindent (ii)
There is a $\delta>0$ such that for every $C2D4$ curve $C':y^2=c' g(x)$ over $K$ with $|c\!-\!c'|$, $|\alpha_i\!-\!\alpha_i'|$, $|\beta_i\!-\!\beta_i'|<\delta$, Conjecture \ref{conj:local} holds for $C/K$ if and only if it holds for $C'/K$.
\end{lemma}

\begin{proof}
(i) The definition of $\eta_1$ only depends on the roots of the quartic in $f(x)$. Write these as in Lemma \ref{le:D4Family} (with $\alpha_2\!=\!r_1$, $\beta_2\!=\!r_3$, $\alpha_3\!=\! r_2$, $\beta_3\!=\!r_4$ to get the Galois action right). By Remark \ref{re:I22Roots} $\eta_1 = 4(2az_2^2-bz_3^2-2adz_3z_4-abz_4^2)$. One checks that since $a \ne 0$ and at least one of $z_2, z_3, z_4$ is not 0, not all partial derivatives with respect to the parameters are simultaneously 0.
Changing the corresponding parameter by a sufficiently small amount gives the required $g(x)$. 

(ii)
If $\delta$ is sufficiently small then the two curves have the same invariants up to squares (other than $\eta_1)$, the same local root number and $\lambda_{C/K}=\lambda_{C'/K}$ (Lemma \ref{continuity}), and $\eta_1(C')$ is arbitrarily close to $\eta_1(C)$. If either $\eta_1(C)\neq 0$ or $\eta_1(C)=\eta_1(C')=0$, then the result follows directly, as each Hilbert symbol in Conjecture \ref{conj:local} is the same for $C$ and $C'$.

Suppose $\eta_1(C)\!=\!0$ and $\eta_1(C')\!\neq\! 0$.
By Lemma \ref{L}(\ref{le:I22}), the invariants of $C'$ satisfy $\dgd\dgt\!=\!-4\DG^2\dlu\!+\!\eta_1^2$. If $\delta$ is sufficiently small, then $\eta_1$ will be close to 0, so that $-\dgd\dgt\DG^2\dlu$ is a perfect square in $K$. Hence $(\eta_1, -\dgd\dgt\DG^2\dlu)=1$ for $C'/K$, which proves that all the Hilbert symbols for $C$ and $C'$ agree in this case too.
\end{proof}

\begin{lemma}\label{lem:perturbingDelta}
Let $K$ be a local field of characteristic 0, and $C:y^2=cf(x)$ a C2D4 curve over~$K$. 
For every $\epsilon >0$ there exists another $C2D4$ curve $C':y^2=c g(x)$ over $K$ whose roots satisfy $|\alpha_i-\alpha_i'|, |\beta_i-\beta_i'|<\epsilon$ and which has $\Delta(C')\neq 0$.
\end{lemma}
\begin{proof}
We may assume that the curve is centered. If $\Delta(C)\neq 0$ the result is clear, so suppose 
$$
\Delta(C)=-c (\alpha_1^2 \ell_1 -\alpha_2\beta_2(\alpha_3+\beta_3)+\alpha_3\beta_3(\alpha_2+\beta_2))= 0.
$$
If $\ell_1\neq 0$ we can obtain a suitable $C'$ by a small perturbation to $\alpha_1^2$, so suppose $\ell_1=0$, that is $\alpha_2+\beta_2=\alpha_3+\beta_3$. We cannot moreover have $\alpha_2\beta_2=\alpha_3\beta_3$ as then $\alpha_2, \beta_2$ would be the roots of the same quadratic as $\alpha_3, \beta_3$ and $C$ would be singular. Thus the above equation for $\Delta$ forces $\alpha_2=-\beta_2$ and $\alpha_3=-\beta_3$.

To obtain a suitable $C'$ take $\alpha_1'=\alpha_1$ and shift the other roots by a small $t\in K$, that is take $\alpha_i'=\alpha_i+t$, $\beta_i'=\beta_i+t$ for $i=2,3$. This preserves the Galois action on roots, so $g(x)\in K[x]$. Moreover $\ell_1'=\ell_1=0$, but now $\alpha_2'+\beta_2'\neq 0$, and so $\Delta(C')\neq 0$, as required.
\end{proof}

\begin{lemma}\label{thm:mobiusdoublePnonzero}
Let $K$ be a local field of characteristic 0 and $C/K$ a curve that admits two C2D4 structures, $C^{(1)}$ and $C^{(2)}$, such that   $\cP, \Delta\neq0$ for $C^{(1)}$. 
Then for every $\epsilon>0$ there is a C2D4 curve $C_2/K$ that admits two C2D4 structures, 
$C_2^{(1)}$ and $C_2^{(2)}$, such that $C_2^{(1)}$ is arithmetically and $\epsilon$-close to $C^{(1)}$ and such that $C_2^{(2)}$ has $\cP,\Delta\neq 0$.
\end{lemma}

\begin{proof}
Note that a shift of the $x$-coordinate $x\mapsto x+\lambda$ does not change $\cP$ for either C2D4 structure. Thus we can take $m=s^{-1}\circ M \circ s$, where $s$ is the shift that centers $C^{(2)}$ and $M\in\GL_2(K)$ close to the identity that then makes $\cP(C^{(2)}_{Ms})\neq 0$ from Lemma \ref{thm:mobiusPnonzero}. Now apply Lemma \ref{lem:perturbingDelta} to $C_{s^{-1}Ms}$ to obtain $C_2$.
\end{proof}

\section{Main local theorem: general case}\label{s:locconj2}

We now return to the proof of Conjecture \ref{conj:local}.

\subsection{Changing the model}

\begin{theorem}\label{thm:anothermodel}
Let $K$ be a local field of characteristic 0.
Let $C/K$ be a C2D4 curve with $\cP, \Delta \ne 0$. 
Suppose that $C/K$ admits another C2D4 model $C'/K$ for which Conjecture \ref{conj:local} holds.
Then Conjecture \ref{conj:local} holds for $C/K$.
\end{theorem}

\begin{proof}
Write $K=F_v$ as the completion of some number field $F$ at a place~$v$, which also has a complex place $v'\neq v$.

We may change the C2D4 model by scaling the $y$-coordinate (this changes the leading term $c$ by a square), as this does not affect any of the Hilbert symbols in $E_{C/K}$ and hence the validity of Conjecture \ref{conj:local} for $C/K$. By Remark \ref{rmk:models}, we may therefore assume that the model $C'$ is $C_m$ for some $m\in\GL_2(K)$. Since $m$ and $m^{-1}$ are continuous, by Lemma \ref{le:perturbingI22} we may moreover assume that $\eta_1\neq 0$ for both $C$ and $C'$.

By Theorem \ref{thm:C2D4Lift} and Theorem \ref{thm:localconjprelim} there is a C2D4 curve $\tilde C_m$ defined over $F$ which is arithmetically close to $C_m$ over $F_v$ and for which Conjecture \ref{conj:local} holds at all places of $F$.
Moreover, by continuity of $m^{-1}$, we can pick it to be $v$-adically sufficiently close to $C_m$ so that $(\tilde{C}_m)_{m^{-1}}$ is arithmetically close to $C$.

Now use continuity (Lemmata \ref{lem:cts2}, \ref{lem:cts3}) and strong approximation to pick 
$m'\in\GL_2(F)$
such that (i) $m'$ is $v$-adically close to $m^{-1}$, so that $\tilde{C}=(\tilde{C}_m)_{m'}$ is arithmetically close to $C$, and (ii) $m'$ is $w$-adically close to the identity at all places $w\neq v, v'$ that are either archimedean or where $\tilde{C}_m$ has bad reduction, so that $\tilde{C}$ is arithmetically close to  $(\tilde{C}_m)$ at these places.

To summarize, we have now replaced the pair of curves $C, C_m$  defined over $K$ by a pair $\tilde{C}, \tilde {C}_m$ defined over $F$ such that
\begin{itemize}
\item $C$ and $\tilde C$ are arithmetically close over $F_v$,
\item $C_m$ and $\tilde C_m$ are arithmetically close over $F_v$,
\item Conjecture \ref{conj:local} holds for $\tilde C_m$ at all places of $F$,
\item Conjecture \ref{conj:local} holds for $\tilde C$ at all places $w\neq v,v'$ of $F$ that are archimedean or where $C_m$ has bad reduction, and hence by Theorem~\ref{thm:localconjprelim} at all places $w\neq v$ of $F$.
\end{itemize}

By Theorem \ref{thm:introlocalglobal} the 2-parity conjecture holds for $\tilde{C}_m/F$. 
Since $\tilde{C}$ is another model for $\tilde{C}_{m}$, the 2-parity conjecture also holds for $\tilde{C}/F$. By Lemma \ref{lem:lastplace} it follows that Conjecture \ref{conj:local} must also hold for $\tilde{C}$ at the remaining place~$v$. Since this curve is arithmetically close to $C$ over $K=F_v$, the conjecture also holds for $C/K$, as required.
\end{proof}

\subsection{Finite places}

\begin{theorem}\label{thm:localconjI22zero}
Let $K$ be a finite extension of $\Q_p$ and $C/K$ a C2D4 curve  with $\cP, \Delta \neq 0$.
Then Conjecture \ref{conj:local} holds for $C/K$ if either

\noindent (1)
$p$ is odd and $C/K$ is semistable with cluster picture
\GRND[D][D][D][D][D][D], 
\GRUNND[D][D][D][D][D][D], 
\CPONND[D][D][D][D][D][D], 
\CPTNND[D][D][D][D][A][A], 
\CPTCND[A][D][D][A][D][D], 
\CPU[A][A][B][B][C][C], 
\CPTCTN[B][B][A][C][C][A] 
or
\CPTCON[A][A][B][B][C][C], 
or

\noindent (2) $p=2$, $C/K$ has good ordinary reduction and the kernel of the associated Richelot isogeny on the Jacobian is precisely the kernel of the reduction map on 2-torsion points.
\end{theorem}

\begin{proof} We consider the cases of odd and even residue characteristic independently. 
By Lemma \ref{le:perturbingI22}, we may assume that $\eta_1\neq 0$.
By Lemma \ref{lem:odddegext}, we may assume that the residue field of $K$ is sufficiently large.
The result now follows from Theorems \ref{thm:localconjprelim}, \ref{thm:mobiusoverbalance} and \ref{thm:anothermodel} for $p$ odd, and from Proposition \ref{prop:2ordinary} and Theorems \ref{thm:2orderrorterm} and \ref{thm:anothermodel} for $p=2$.

\end{proof}


\subsection{Changing the isogeny}

\begin{theorem}\label{thm:changeisogeny}
Let $K$ be a non-archimedean local field of characteristic 0.
Let $C:y^2 =c f(x)$ be a curve over $K$ that admits two C2D4 structures, $C^{(1)}$ and $C^{(2)}$, both of which have $\cP, \Delta\neq 0$ and such that Conjecture \ref{conj:local} holds for $C^{(1)}$.
Suppose that one of the following two conditions holds:

\noindent (1)
$\bullet$ the second colouring is obtained from the first by relabelling colours,~and\\
\indent \, $\bullet$ $\Gal(f)$ preserves colours; or

\noindent (2)
$\bullet$ both colouring have the same ruby roots, and \\
\indent \, $\bullet$  $\Gal(f)$ acts on the sapphire and turquoise roots as a subgroup of $V_4$.

\noindent Then Conjecture \ref{conj:local} holds for $C^{(2)}$.
\end{theorem}

\begin{proof}[Proof for $K$ a finite extension of $\Q_2$.]
Since the validity of Conjecture \ref{conj:local} is unchanged by going to an unramified extension of odd degree (Lemma \ref{lem:odddegext}), we may assume that $[K:\Q_2]>1$.

Pick a number field $F$ that has a prime $v$ above 2 with completion $F_v\simeq K$, and such that $F$ has no other primes above 2 and has a complex place. (To see that such a field exists, pick a primitive generator $\theta$ for $K/\Q_2$ and approximate its minimal polynomial by a polynomial $m(x)\in\Q[x]$ that has at least two complex roots; then $F=\Q[x]/m(x)$ has the required property.)

Over local fields, a small perturbation to the coefficients of a separable polynomial does not change its Galois group, 
so by Lemma \ref{le:perturbingI22} we may assume that both curves have $\eta_1\neq 0$. By Theorems \ref{thm:doublelift} and \ref{thm:localconjprelim} there is a curve $\tilde{C}/F$ that admits two C2D4 structures $\tilde{C}^{(1)}$ and $\tilde{C}^{(2)}$ such that $\tilde{C}^{(i)}$ is close to $C^{(i)}$ and such that Conjecture \ref{conj:local} holds for both $\tilde{C}^{(1)}$ and $\tilde{C}^{(2)}$ at all places $w\neq v$. 
In particular, Conjecture \ref{conj:local} holds for $\tilde{C}^{(1)}$ at all places of $F$, and hence the 2-parity conjecture holds for $\tilde{C}$ (Theorem \ref{thm:introlocalglobal}). 
It thus also holds for $\tilde{C}^{(2)}$, and thus by Lemma \ref{lem:lastplace} it follows that Conjecture \ref{conj:local} must hold for $\tilde{C}^{(2)}/F_v$, and hence for $C^{(2)}/K$.
\end{proof}

\subsection{2-adic places}

\begin{theorem}\label{thm:localconjeven}
Let $K$ be a finite extension of $\Q_2$. Suppose $C/K$ lies in~$\mathcal{F}$.
Then Conjecture~\ref{conj:local} holds for every C2D4 structure on $C\!/\!K$ with \hbox{$\cP,\Delta\!\neq\! 0$}.
\end{theorem}

\begin{proof}
The curve $C:y^2=cf(x)$ has all its Weierstrass points defined over $K$.
As $\Gal(f)$ is trivial, we can repeatedly apply Theorem \ref{thm:changeisogeny} (and Lemma \ref{thm:mobiusdoublePnonzero}) to the given C2D4 structure $C^{(1)}$ to change it to the standard C2D4 structure $C^{(0)}\in\mathcal{F}_{C2D4}$. Conjecture \ref{conj:local} holds for $C^{(0)}$ (Theorem \ref{thm:errorterm2}), and hence for $C^{(1)}$ as well. 
\end{proof}

\subsection{Changing the isogeny (continued)}

\begin{proof}[Proof of Theorem \ref{thm:changeisogeny} for $K$ archimedean or a finite extension of $\Q_p$, $p$ odd.]
Let $F$ be a number field with a prime $v$ such that $F_v\simeq K$ and $F$ has a complex place $v'\neq v$. By Theorem \ref{thm:localconjeven}, Conjecture \ref{conj:local} holds for all curves over 2-adic fields that lie in $\mathcal{F}$, irrespectively of the choice of the C2D4 structure. The proof now follows verbatim as the third paragraph of the proof of the case when $K$ is an extension of $\Q_2$.
\end{proof}

\subsection{Proof of Theorem \ref{thm:introlocal}}

\begin{theorem}[=Theorem\ \ref{thm:introlocal}]\label{thm:mainlocal}\label{th:localtheoremI22nonzero1}
Conjecture \ref{conj:local} holds for all C2D4 curves with \hbox{$\cP, \Delta\neq 0$} over archimedean local fields, all semistable C2D4 curves with $\cP, \Delta\neq 0$ over finite extensions of $\Q_p$ for odd primes $p$, and all C2D4 curves with $\cP, \Delta\neq 0$ and good ordinary reduction over finite extensions of $\Q_2$.
\end{theorem}

\begin{proof}
Write the curve as $C: y^2=cf(x)$ and consider the colouring of the roots of $f(x)$ given by the C2D4 structure, $C^{(1)}$. Observe that
\begin{enumerate}
\item if $\Gal(f)$ preserves colours, then $C$ admits two other C2D4 structures obtained from the original one by relabelling the colours; and
\item if $\Gal(f)$ acts as a subgroup of $V_4$ on the sapphire and turquoise roots, then $C$ admits two other C2D4 structures obtained from the original one by changing the colouring of sapphire and turquoise roots.
\end{enumerate}
Let $C^{(2)}$ be any one of these structures.
By Lemma \ref{thm:mobiusdoublePnonzero} we may assume that $C^{(2)}$ also has $\cP, \Delta\neq 0$. By Theorem \ref{thm:changeisogeny} it then suffices to prove the result for $C^{(2)}$.

We now show that through repeated use of (1) and (2) we can reduce the problem to one already covered by Theorems \ref{thm:localconjprelim} and \ref{thm:localconjI22zero}.

\underline{Complex places:}
The result is covered by Theorem \ref{thm:localconjprelim}.

\underline{Real places:}
We may assume that if $f(x)$ has a real root then $c<0$:  indeed, by Theorem \ref{thm:anothermodel} we can use a change of model given by $x\mapsto\frac{1}{x-t}$ for a suitable $t\in\R$ to make the leading term negative.

If $f(x)$ has 6 real roots then a repeated use of (1) and (2) brings it to the picture \raise3pt\hbox{\CPRNF[A][A][B][B][C][C]}
.

If $f(x)$ has 4 real roots, then by (1) we may assume that the complex roots are ruby, and then by (2) that the picture is \raise3pt\hbox{\clusterpicture            
  \Root[B] {} {first} {r1};
  \Root[B] {} {r1} {r2};
  \frob (r1)(r2);
   \Root[C] {} {r2} {r3};
  \Root[C] {} {r3} {r4};
  \frob (r3)(r4);
 \endclusterpicture} with  $\a_1=\bar{\b}_1$
 .

If $f(x)$ has 2 real roots, then by (1) we may assume that these are ruby, and then by (2)  that the picture is \raise3pt\hbox{\clusterpicture            
  \Root[A] {} {first} {r1};
  \Root[A] {} {r1} {r2};
  \frob (r1)(r2);
 \endclusterpicture} with $\a_2=\bar{\b}_2$ and $\a_3=\bar{\b}_3$,
and then by (1) again that the picture is instead
\raise3pt\hbox{\clusterpicture         
  \Root[B] {} {first} {r1};
  \Root[B] {} {r1} {r2};
  \frob (r1)(r2);
 \endclusterpicture} with $\a_1=\bar{\b}_1$ and $\a_3=\bar{\b}_3$.

If $f(x)$ has no real roots, then the ruby roots are necessarily complex conjugate and by (2) we can make 
$\a_i=\bar{\b}_i$ for $i=1,2,3$.

\underline{Odd primes:}
By Theorems \ref{thm:mobiusbalance} and \ref{thm:anothermodel} we may assume that the cluster picture of $C/K$ is balanced (using Lemma \ref{lem:odddegext}(5) to enlarge $|k|$ if necessary). 
If the reduction has type $\It$ or $\Io_{n}$ (in the sense of Theorem \ref{th:genus2bible}), the result follows from Theorem \ref{thm:localconjI22zero}. Otherwise, its cluster picture is one of the ones given below. Applying steps (1) and (2) as indicated above the arrows reduces the problem to one covered by Theorem \ref{thm:localconjI22zero}.

\noindent
\resizebox{\textwidth}{!}{
$\begin{CD}
\text{\underline{Type $\UU_{n,m,l}$}}@.\CPU[A][B][A][C][B][C]@>2>>\CPU[A][B][A][B][C][C] @>1>>\CPU[C][B][C][B][A][A] @>2>>\CPU[C][C][B][B][A][A] \\
@.@.@.@.@.\\
\text{\underline{Type $\Io\!\times_t\!\Io$}
}
@{.}\CPTCTCND[A][A][B][B][C][C]@>1>>\CPTCTCND[B][B][A][A][C][C] @{.}\phantom{hahah} \text{OR}\phantom{hahah} @{.}\CPTCTCND[A][B][C][A][B][C]\\
@.@.@.@.@.\\
\text{\underline{Type $\II_{n,m}$}\phantom{ha}}@{.}\CPTN[A][B][C][B][A][C]@>1>>\CPTN[B][A][C][A][B][C] @>2>>\CPTN[B][A][B][A][C][C] @>1>>\CPTN[B][C][B][C][A][A]  \\
@.\CPTN[B][B][A][C][A][C]@>1>>\CPTN[A][A][B][C][B][C] @>2>>\CPTN[A][A][C][C][B][B] @>1>>\CPTN[B][B][C][C][A][A]\\
@.@.@.@.@.\\
\text{\underline{Type $\Io\!\times_t\! \II_{n}$}
}@{.}\CPTCON[A][B][A][B][C][C]@>1>>\CPTCON[B][A][B][A][C][C]@>1>>\CPTCON[B][C][B][C][A][A] @>2>>\CPTCON[C][C][B][B][A][A] \\
@.@.@.@.@VV1V\\
@.\CPTCON[B][A][C][C][B][A] @>1>>\CPTCON[B][C][A][A][B][C] @>2>>\CPTCON[B][B][A][A][C][C]@>1>>\CPTCON[A][A][B][B][C][C]\\
@.@.@.@.@.\\
\text{\underline{Type $\II_{n}\!\times_t\!\II_{m}$} }@{.}\CPTCTN[A][A][B][B][C][C]@<1<< \CPTCTN[B][B][A][A][C][C]@<1<<\CPTCTN[B][B][C][C][A][A]@.\CPTCTN[A][B][C][C][B][A]  \\
@.@VV2V@.@AA2A@VV1V\\
@.\CPTCTN[A][A][B][C][C][B]@. @.\CPTCTN[B][C][C][B][A][A]  @.\CPTCTN[B][A][C][C][A][B]\\
@.@VV1V@.@.@VV2V\\
@.\CPTCTN[B][B][A][C][C][A]@<2<< \CPTCTN[B][C][A][B][C][A]@<1<<\CPTCTN[B][A][C][B][A][C]@<2<<\CPTCTN[B][A][B][C][A][C]
\end{CD}$
}

\medskip

\underline{2-adic primes:}
By Proposition \ref{prop:2ordinary}(i) and Theorem \ref{thm:anothermodel} we may assume that the cluster picture of $C$ has three twins of depth $v(4)$, and that the depth of the cluster containing all six roots is 0 (using Lemma \ref{lem:odddegext}(5) to first enlarge $|k|$ if necessary). The result follows as for the case of Type~$\UU_{n,m,l}$ above.
\end{proof}


\section{Global results}\label{s:final}

We now complete the proofs of the theorems given in the introduction.

\begin{proposition}\label{prop:2parityforproducts}
The 2-parity conjecture  holds for abelian varieties over number fields $A/K$ with $\Gal(K(A[2])/K)$ a 2-group that are either products of elliptic curves, or the Weil restriction of an elliptic curve from a field extension.
\end{proposition}
\begin{proof}
If $A\simeq \prod E_i$ or $\simeq\Res_{F/K}E$ then the condition on the 2-torsion field ensures that the elliptic curves $E_i/K$  or $E/F$ all admit a 2-isogeny. By \cite{kurast} Thm.\ 5.8, the 2-parity conjecture holds for $E_i/K$ (respectively, $E/F$). As the 2-parity conjecture is compatible with products and with Weil restriction of scalars (as both root numbers and $p^\infty$-Selmer ranks are), it also holds for $A/K$.
\end{proof}

\begin{theorem}\label{thm:ParityDelta0}
The 2-parity conjecture  holds for all C2D4 curves over number fields $C/K$ with $\Delta=0$.
\end{theorem}

\begin{proof}
By \cite{Smith} Def.\ 8.2.4 and Prop.\ 8.3.1,
$A=\Jac C$ has an isogeny of degree 4 to an abelian variety $B$ that is either a product of two elliptic curves 
or the Weil restriction 
of an elliptic curve from a quadratic extension. By hypothesis, $\Gal(K(A[2])/K)$ is a 2-group, and hence so is $\Gal(K(A[2^n])/K)$.
It follows that $\Gal(K(B[2])/K)$ is also a 2-group. 
The result now follows by Proposition \ref{prop:2parityforproducts}, since the 2-parity conjecture is compatible with isogenies (as both root numbers and  $p^\infty$-Selmer ranks are invariant under isogenies).
\end{proof}

\begin{theorem}\label{thm:2paritymain}
The 2-parity conjecture holds for all principally polarised abelian surfaces over number fields $A/K$ such that $\Gal(K(A[2])/K)$ is a 2-group that are either
\begin{itemize}[leftmargin=*]
\item the Jacobian of a semistable genus 2 curve with good ordinary reduction at primes above 2, or
\item not isomorphic to the Jacobian of a genus 2 curve.
\end{itemize}
\end{theorem}

\begin{proof}
By \cite{GL2Type} Thm.\ 3.1 $A$ is either a product of two elliptic curves, the Weil restriction of an elliptic curve from a quadratic field extension or is the Jacobian of a genus 2 curve $C/K$. By Proposition \ref{prop:2parityforproducts} and the hypothesis on the 2-torsion field, we may assume that $A=\Jac C$ for a C2D4 curve $C/K$. By Theorem \ref{thm:ParityDelta0} and Lemma \ref{thm:mobiusPnonzero} we may also assume that $\cP, \Delta\neq 0$.
The result now follows from Theorems \ref{thm:introlocalglobal} and \ref{thm:mainlocal}.
\end{proof}

\begin{corollary}
The 2-parity conjecture holds for all semistable C2D4 curves over number fields $C/K$ that have good ordinary reduction at the primes above 2.
\end{corollary}

\begin{theorem}\label{thm:paritymain}
The parity conjecture holds for all principally polarised abelian surfaces over number fields $A/K$ such that $\sha_{A/K(A[2])}$ has finite 2-, 3- and 5-primary part that are either
\begin{itemize}[leftmargin=*]
\item the Jacobian of a semistable genus 2 curve with good ordinary reduction at primes above 2, or
\item semistable, and not isomorphic to the Jacobian of a genus 2 curve.
\end{itemize}
\end{theorem}

\begin{proof}
This now follows from Theorems \ref{thm:2paritymain} and \ref{appmain1}. 
\end{proof}

\begin{theorem}
\label{thm:converse}
If the 2-parity conjecture is true for all Jacobians of C2D4 curves over number fields, then Conjecture \ref{conj:local} holds for all C2D4 curves over local fields of characteristic 0 with $\cP, \Delta\neq 0$.
\end{theorem}

\begin{proof}
Let $C/K$ be a C2D4 curve over a local field of characteristic 0 with $\cP\neq 0$. 
Let $F$ be a number field with a place $v$, such that $F_v\iso K$, and some other place $v'$ that is archimedean.

By Lemma \ref{le:perturbingI22} we may assume that $\eta_1\neq 0$.
By Theorems \ref{thm:C2D4Lift} and \ref{th:localtheoremI22nonzero1} we can find a C2D4 curve $C'$ over $F$ such that $C$ and $C'$ are arithmetically close over $F_v$ and such that Conjecture~\ref{conj:local} holds for $C/F_w$ for all places $w\neq v$.
By assumption, the 2-parity conjecture holds for $C'/F$, so by Lemma \ref{lem:lastplace} Conjecture~\ref{conj:local} holds for $C/F_v$. It follows by Lemma \ref{lem:cts3} that Conjecture \ref{conj:local} also holds for $C/K$.
\end{proof}
\newpage

\def\sectionname{Appendix}
\def\G{\ensuremath\widehat{\mathbb{G}}_m}

\def\thesection{A}

\section{Isogenies between abelian varieties with good ordinary reduction}

\begin{center}
 by ADAM MORGAN
\end{center}

\subsection{Statement of the result}\label{ss:adamapp1}

Let $K$ be a finite extension of $\mathbb{Q}_2$ and $A/K$ a principally polarised abelian variety of dimension $g$, with good ordinary reduction. Let $W$ be a maximal isotropic subspace of $A[2]$ (for the Weil pairing associated to the principal polarisation), stable under the action of the absolute Galois group $\text{Gal}(\bar{K}/K)$. Let $\phi:A\rightarrow B$ be the $K$-isogeny with kernel $W$, so that $B$ is principally polarised also, has good ordinary reduction, and (after identifying $A$ and $B$ with their duals) the dual isogeny $\hat{\phi}:B\rightarrow A$ satisfies \[\phi\circ \hat{\phi}= \hat{\phi} \circ \phi =[2].\]  Let $A_1(\bar{K})$ denote the kernel of reduction on $A$. 

The aim of the appendix is to prove the following result, whose proof we give in \S\ref{ss:adamapp3} after reviewing endomorphisms of the formal multiplicative group.

\begin{theorem}\label{first cokernel theorem}
With the notation above we have
\[\frac{\left|B(K)/\phi(A(K))\right|}{\left|A(K)[\phi]\right|}=2^{[K:\mathbb{Q}_2]\dim_{\mathbb{F}_2}\left(A_1(\bar{K})[2]\cap A(\bar{K})[\phi]\right)}.\]
\end{theorem}
 
\subsection{Endomorphisms of the formal multiplicative group}\label{ss:adamapp2}

Again, let $K$ be a finite extension of $\mathbb{Q}_2$, and let $T$ denote the completion of the maximal unramified extension of $K$. Let $\mathcal{O}$ be the ring of integers of $T$, so that $\mathcal{O}$ is a complete discrete valuation ring, whose normalised valuation restricts to that of $K$. Let $\G$ denote the formal multiplicative group over $\mathcal{O}$. In general, given formal group laws $\mathcal{F}$ and $\mathcal{G}$ over $\mathcal{O}$ of dimension $g$, and a homomorphism $\phi$ from $\mathcal{F}$ to $\mathcal{G}$, we denote by $D(\phi)$ the \textit{Jacobian} of $\phi$. That is, $\phi$ is a $n$-tuple of power series in $g$ variables $X=(x_1,...,x_g)$, coefficients in $\mathcal{O}$, and $D(\phi)\in \text{M}_g(\mathcal{O})$ is an $g \times g$ matrix such that 
\[\phi(X) \equiv D(\phi)X ~~\text{ (mod deg 2)}.\]
The homomorphism $\phi$ is an isomorphism if and only if $D(\phi)$ is invertible in $\text{M}_g(\mathcal{O})$.

\begin{lemma}\label{endomorphism lemma}
For each $g\geq 1$, the map $\phi \mapsto D(\phi)$ gives an isomorphism of rings 
\[\textup{End}_\mathcal{O}(\G^g) \stackrel{\sim}{\longrightarrow} \mathrm{M}_g(\mathbb{Z}_2).\]
\end{lemma} 

\begin{proof}
The result for general $g$ follows formally from the case $g=1$, which is standard, although we provide the proof for convenience. The formal logarithm gives an isomorphism from $\G$ to the formal additive group $\widehat{\mathbb{G}}_a$ over $T$, and the endomorphisms of the latter are given by $\phi(X)=aX$ for $a\in T$. Thus one sees that the endomorphisms of $\G$ over $\mathcal{O}$ are exactly those of the form
\[\phi_a(X)=aX+a(a-1)X^2/2+a(a-1)(a-2)x^3/3!+...\]
for those $a \in T$  such that each coefficient of $\phi_a(X)$ is in $\mathcal{O}$. Considering the coefficients of $x^{2^n}$ for varying $n$, one sees easily that this is equivalent to $a\in \mathbb{Z}_2$, from which the result follows.
\end{proof}

Now let $\mathfrak{m}$ denote the maximal ideal in $\mathcal{O}$. Letting $U^1(T)$ denote the group of units in $\mathcal{O}$ reducing to $1$ in the residue field $\mathcal{O}/\mathfrak{m}$, the map \[(x_1,...,x_g)\mapsto (x_1-1,...,x_g-1)\] gives an isomorphism form $U^1(T)^g$ to $\mathfrak{m}^g$ with the group structure on the latter coming from the formal group law $\G^g$. Any endomorphism $\phi \in \textup{End}_\mathcal{O}(\G^g)$ induces via this isomorphism an endomorphism of $U^1(T)^g$. We denote by $U^1(T)^g[\phi]$ the kernel of this map. 

\begin{lemma}\label{kernel lemma}
Let $\phi \in \mathrm{End}_\mathcal{O}(\G^g)$ and suppose there is $\psi \in \mathrm{End}_\mathcal{O}(\G^g)$ such that $\phi \circ \psi=[2]$ (here $[2]$ denotes the multiplication-by-2 map on $\G^g$). Then  $U^1(T)^g[\phi]$, being contained in $U^1(T)^g[2]=\{ \pm 1\}^g$, is a finite dimensional $\mathbb{F}_2$-vector space and we have
\[\dim_{\mathbb{F}_2}U^1(T)^g[\phi]= \textup{ord}_2 ~\textup{det} D(\phi).\]  
\end{lemma}

\begin{proof}
Let $M=D(\phi) \in M_g(\mathbb{Z}_2)$. By properties of Smith Normal Form, we can find invertible matrices $V$ and $W$ in $M_g(\mathbb{Z}_2)$ such that $M=VNW$ where $N$ is a diagonal matrix whose entries are powers of $2$. On the other hand, $D(\phi)D(\psi)$ is twice the identity matrix. Thus 
\[2 V^{-1}=NWD(\psi).\] 
In particular, each coefficient of $NWD(\psi)$ is divisible by $2$, yet $\frac{1}{2}NWD(\psi)$ has determinant a $2$-adic unit. If one of the entries of $N$ were divisible by $4$ then $2$ would divide each entry of some row of $\frac{1}{2}NWD(\psi)$, and hence its determinant, a contradiction. We deduce that each entry of $N$ is either $1$ or $2$. Moreover, the matrices $V$ and $W$ correspond to automorphisms of $\G^g$ under Lemma \ref{endomorphism lemma} and since we are only interested in the size of the kernel of $\phi$, we may replace $\phi$ with the endomorphism corresponding to $N$ (by construction, we also have $\text{ord}_2~\text{det}D(\phi)=\text{ord}_2~\text{det}N)$. However, as $N$ is diagonal with entries either $1$ or $2$, the endomorphism of $U^1(T)^g$ induced by $N$ is just the identity on each factor  where $N$ has a $1$ on the diagonal, and the map $x\mapsto x^2$ on each factor  where $N$ has a $2$ on the diagonal. The $\mathbb{F}_2$-dimension of the kernel of this map is just the number of diagonal entries of $N$ equal to $2$, which is  equal to $\text{ord}_2~ \text{det} N$.
\end{proof} 

\subsection{Proof of Theorem \ref{first cokernel theorem}}\label{ss:adamapp3}

We keep the notation of $\S\ref{ss:adamapp1}$ and $\S\ref{ss:adamapp2}$, so that in particular  let $A/K$ be a principally polarised abelian variety of dimension $g$ with good ordinary reduction, let $B/K$ be isogenous to $A$ via $\phi$, and consider the auxiliary isogeny $\psi:B\rightarrow A$ such that  $\phi\circ \psi=[2]$.
Let $v_T$ be the normalised valuation on $T$ (which extends that on $K$), $k$ denote the residue field of $K$, and let $e(K/\mathbb{Q}_2)$ be the ramification index of $K$ over $\mathbb{Q}_2$. Let $\mathcal{F}_A$ and $\mathcal{F}_B$ be the dimension $g$ formal group laws over the ring of integers $\mathcal{O}_K$ of $K$ associated to $A$ and $B$ respectively. Then $\phi$ induces an element of $\text{Hom}_{\mathcal{O}_K}(\mathcal{F}_A,\mathcal{F}_B)$ which, by an abuse of notation, we also denote by $\phi$. Similarly, we obtain $\psi\in \text{Hom}_{\mathcal{O}_K}(\mathcal{F}_B,\mathcal{F}_A)$ and we have $\phi\circ \psi=[2]$. Since $A$ and $B$ have good ordinary reduction, over $\mathcal{O}$ (the ring of integers of $T$), there is an isomorphism $\alpha$ from $\mathcal{F}_A$ to $\G^g$, and similarly an isomorphism $\beta$ from $\mathcal{F}_B$ to $\G^g$ (see \cite[Lemma 4.27]{MR0444670} for more details). We thus obtain elements $\phi':=\beta \phi \alpha^{-1}$ and $\psi':=\alpha \psi \beta^{-1}$ of $\text{End}_\mathcal{O}(\G^g)$ whose composition is multiplication by $2$. Moreover, since $\alpha$ and $\beta$ are isomorphisms, $D(\alpha)$ and $D(\beta)$ are invertible matrices in $M_n(\mathcal{O})$. In particular, the determinants of $D(\alpha)$ and $D(\beta)$ are units in $\mathcal{O}$. Thus \[v_K(\text{det}D(\phi))= v_T(\text{det}D(\phi'))=e(K/\mathbb{Q}_2)~\text{ord}_2~ \text{det}D(\phi').\]
Applying Lemma \ref{kernel lemma} we obtain
\[v_K(\text{det}D(\phi))=e(K/\mathbb{Q}_2) \dim_{\mathbb{F}_2}U^1(T)^g[\phi'].\]

Let $\bar{K}$ be an algebraic closure of $K$ and let $A_1(\bar{K})[2]$ denote the points in $A(\bar{K})[2]$ reducing to the identity under the reduction map into $\bar{k}$ (the algebraic closure of $k$). Then since $A$ has good ordinary reduction, $A_1(\bar{K})[2]$ has size $2^g$. On the other hand, the points in $A_1(T)$ correspond isomorphically under $\alpha$ to the points of $U^1(T)^g$, and there are $2^g$ such $2$-torsion points in this latter group, namely the points $\{\pm1\}^g$. In particular, we deduce that
\[\dim_{\mathbb{F}_2} U^1(T)[\phi']=\dim_{\mathbb{F}_2}\left(A_1(\bar{K})[2]\cap A(\bar{K})[\phi]\right).\]
We thus have
\[\text{ord}_2\left(|\text{det} D(\phi)|_K^{-1}\right)=[K:\mathbb{Q}_2]\dim_{\mathbb{F}_2}\left(A_1(\bar{K})[2]\cap A(\bar{K})[\phi]\right).\]

On the other hand, since the N\'{e}ron component groups of $A$ and $B$ are both trivial, the left hand side of the above equality is equal to
\[\text{ord}_2 \frac{\left|B(K)/\phi(A(K))\right|}{\left|A(K)[\phi]\right|}, \]
as follows from \cite{Schaefer} Lemma 3.8.

\def\thesection{B}

\section{Parity reduction to 2-Sylow}\label{s:appTV}

\begin{center}
by T. DOKCHITSER and V.DOKCHITSER
\end{center}

Recall that the parity conjecture (Conjecture \ref{conj:parity})
relates the parity of the Mordell-Weil rank of a principally polarised abelian variety $A$ 
over a number field $K$ to the global root number,
$$
  (-1)^{\rk A/K} = w(A/K).
$$
In this appendix we show that, assuming finiteness of $\sha$ (and a mild restriction on the reduction of $A$), 
the conjecture follows from its special case when $A$ admits a suitable $2$-power isogeny. 
This special case was proved for elliptic curves in \cite{isogroot} and this reduction approach was 
proved in \cite{squarity} to deduce the general case for elliptic curves. The present article deals with the 
case of abelian surfaces. Ironically, \cite{CFKS} establish the analogous result for a $p$-power isogeny when
$p$ is odd.

\begin{theorem}
\label{appmain1}
Let $F/K$ be a Galois extension of number fields with Galois group $G$, 
and $A/K$ a principally polarised abelian variety.
Suppose 
\begin{enumerate}
\item
$\sha(A/F)$ has finite $p$-primary part for every odd prime $p$ that divides $|G|$,
\item 
all primes of unstable reduction 
of $A$ have cyclic decomposition groups in $G$.
\end{enumerate}
If the parity conjecture holds for $A/F^H$ for all $H\le G$ of 2-power order, then it holds for $A/K$.
\end{theorem}

\begin{corollary}
\label{appmain2}
Let $K$ be a number field, $A/K$ a principally polarised abelian variety, and 
write $F=K(A[2])$. 
Suppose 
\begin{enumerate}
\item
$\sha(A/F)$ has finite $p$-primary part for every odd prime $p\,|\,[F\!:\!K]$,
\item
all primes of unstable reduction 
of $A$ have cyclic decomposition groups in $F/K$,
\item 
The parity conjecture holds for $A/L$ for subfields $K\subset L\subset F$ over which $A$ admits an isogeny
$\phi:A\to A'$ with $\phi\phi^t=[2]$.
\end{enumerate}
Then the parity conjecture holds for $A/K$.
\end{corollary}

Note the assumption (2) in \ref{appmain1} and \ref{appmain2} holds when $A/K$ is semistable.
We now turn to the proofs.

\begin{lemma}
\label{applem1}
Let $G$ be a finite group. Then
$$
  \triv_G = \sum n_i \Ind_{H_i}^G \triv_{H_i}
$$
for some $n_i\in\Z$, and where each $H_i$ is one of the following types of subgroups of $G$:
\begin{enumerate}
\item 
$H_i$ has 2-power order; or
\item
There is $U\normal H$ with $H/U\iso C_p\rtimes C_{2^k}$
for $p$ an odd prime, $k\ge 0$, and $C_{2^k}$ acting faithfully on $C_p$.
\end{enumerate}
\end{lemma}

\begin{proof}
We proceed by induction on $|G|$.
Solomon's induction theorem expresses $\triv_G$ as an integral linear
combination of $\Ind_{H_i}^G\triv_{H_i}$ for some hyperelementary $H_i<G$. As induction
is transitive, we may assume $G$ is hyperelementary. (Recall that a group $G$ is hyperelementary if $G \simeq C \rtimes P$ for a $p$-group $P$ and a cyclic group $C$ of order prime to $p$.)

If $G$ has a non-trivial odd order quotient, then it has a $C_p$-quotient for
some odd prime $p$, and we are done by (2). Otherwise, $G=C_n\rtimes P$ for 
some odd $n$ and a 2-group $P$. If $n=1$, we are done by (1).
By passing to a quotient if necessary, we may 
assume that $n$ is prime and, moreover, that $P$ acts faithfully.
Then we are done by (2).
\end{proof}

Recall that for a prime $l$, we define the dual $l^\infty$-Selmer group,
$$
  X_l(A/K)=(\text{Pontryagin dual of the $l^\infty$-Selmer group of } A/K)\>\tensor\Q_l.
$$
This is a $\Q_l$-vector space whose dimension is 
the Mordell-Weil rank of $A/K$ plus the number of copies of $\Q_l/\Z_l$
in $\sha(A/K)$.

\begin{conjecture}[$l$-Parity conjecture = Conjecture \ref{conj:pparity}]
$$
  (-1)^{\dim X_l(A/K)}=w(A/K).
$$
\end{conjecture}

If $\sha$ is finite, this is equivalent to the parity conjecture.

\begin{lemma}
\label{applem2}
Let $F/K$ be a Galois extension of number fields with Galois group
$G=C_p\rtimes C_{2^k}$ with $p$ an odd prime, $k\ge 0$, 
and $C_{2^k}$ acting faithfully on $C_p$. Let $A/K$ be a principally polarised
abelian variety and $l$ a prime. Suppose either
\begin{enumerate}
\item
$k=0$; or
\item
$0\!<\!k\!<\!\ord_2(p-1)$, and either $l=p$ or the 2-part of the 
order of $l\in \F_p^\times$ is $>k$; or
\item 
$k=\ord_2(p-1)$, $l=p$ and all primes of unstable reduction 
of $A$ have cyclic decomposition groups in $G$.
\end{enumerate}
If the $l$-parity conjecture holds for $A/L$ for all $K\subsetneq L\subset F$,
then it holds for $A/K$. 
\end{lemma}

\begin{proof}
(1) $l$-parity is invariant under odd degree Galois extensions, see e.g. \cite[Cor.\ A.3(3)]{tamroot}.

(2) The (absolutely) irreducible representations of $G$ are 1-dimensionals that factor through $C_{2^k}$ and $2^k$-dimensionals of the form $\rho_\psi=\Ind_{C_p}^G\psi$ for faithful 1-dimensional $\psi$. The field generated by the character of $\psi$ is $\Q(\zeta_p)$, so the number of $\Gal(\bar{\Q}_l/\Q_l)$-conjugates of $\psi$ is $[\Q_l(\zeta_p)\colon\!\Q_l]$,
which is
the order of $l$ in $\F^\times_p$. In particular, under the assumption on $l$, $\rho_\psi$ has an even number of images under $\Gal(\bar{\Q}_l/\Q_l)$, so that each irreducible $\Q_l$-representation has an even number of absolutely irreducible constituents. Thus 
$$
 \rk_l {A/F^{C_{2^k}}} = \dim X_l^{C_{2^k}} = \langle X_l, \Ind_{C_{2^k}}^G\triv \rangle= \langle X_l, \triv \oplus \bigoplus_{\rho_\psi}\rho_\psi \rangle \equiv
 $$
 $$
  \equiv \langle X_l, \triv \rangle = \rk_l {A/K}\mod 2,
$$
where the direct sum ranges over absolutely irreducible non-1-dimensional 
representations of $G$. Also, by \cite[Thm.\ 1]{RohI} 
(or \cite[Prop. A.2(5)]{tamroot}), twists of $A$ by Galois conjugate orthogonal 
characters have the same root number, so $w(A, \rho_\psi) = w(A, \rho_\psi')$ for 
all $\psi,\psi'\neq \triv$, and 
$$
 w(A/F^{C_{2^k}}) = w(A, \triv \oplus \bigoplus_{\rho_\psi}\rho_\psi) = w(A/K).
$$
Therefore $l$-parity over $F^{C_{2^k}}$ implies $l$-parity over $K$, as claimed.

(3) We invoke the regulator constant machinery of \cite{squarity, tamroot} that proves
special cases of the parity conjecture by exploiting Brauer relations in Galois groups.
There is a Brauer relation in $G$ (see e.g. \cite[Ex.\ 2.3]{brauer})
$$
  \Theta = C_1 - C_p - 2^k C_{2^k} + 2^k G.
$$
Write $\epsilon$ for the order 2 character of $C_{2^k}$, and $\psi_i$ for
its other non-trivial characters. Then 
the irreducible self-dual $\Q_p G$-representations are $\triv$, $\epsilon$,
$\psi_i\oplus\psi_i^{-1}$ and $\rho=\Ind_{C_{2^k}}^G\ominus\triv$.
The regulator constants (over $\Q_p$) $\cC_\Theta(\psi_i\oplus\psi_i^{-1})$ 
are 1 by \cite[Cor.\ 2.25(3)]{tamroot}. 
We can compute the others using e.g. \cite[Ex.\ 2.19]{tamroot}, and we find
$$
  \cC_\Theta(\triv)=\cC_\Theta(\epsilon)=\cC_\Theta(\rho)=p.
$$
(This is done in \cite[Ex.\ 2.20]{tamroot} when $k=1$.)

By \cite[Thm.\ 1.6b]{tamroot}, the $l$-parity conjecture holds for the twist 
of $A/K$ by any self-dual $\bar\Q_p G$-representation $\tau$ such that
$$
  \langle\tau,\triv\rangle=\langle\tau,\epsilon\rangle=\langle\tau,\rho\rangle\equiv 1\mod 2.
$$
Since $\langle\rho,\rho\rangle=\frac{p-1}{2^k}$ is odd, we can take
$$
  \tau = \Ind_{C_{2^k}}^G\triv + \Ind_{C_p}^G\triv + \triv. 
$$
Since $l$-parity holds over $F^{C_{2^k}}$ and over $F^{C_p}$ by assumption,
it holds for the twists of $A$ by $\Ind_{C_{2^k}}^G\triv$ and by 
$\Ind_{C_p}^G\triv$ (see \cite[Cor.\ A.3(2)]{tamroot}), and 
therefore for $A/K$.
\end{proof}

\begin{proof}[Proof of Theorem \ref{appmain1}]
We proceed by induction on $|G|$. 
Write $\triv=\sum n_i \Ind_{H_i}^G\triv$ as in Lemma \ref{applem1}. Then
$$
\begin{array}{llllllllllll} 
  \rk A/K 
  &=& \displaystyle \langle A(F)\!\tensor_\Z\C, \triv\rangle 
  = \sum_i \langle A(F)\!\tensor_\Z\C, n_i\Ind_{H_i}^G\triv\rangle \cr
  &=& \displaystyle \sum_i n_i \langle A(F)\!\tensor_\Z\C, \triv\rangle_{H_i} 
  = \sum_i n_i \rk A/F^{H_i}.
\end{array}
$$
Similarly, by Artin formalism for root numbers,
$$
  w(A/K) = \prod_i w(A/F^{H_i})^{n_i}.
$$
Therefore it suffices to prove parity for $A/F^{H_i}$ for all $i$.
If $H_i$ has 2-power order, then parity holds by assumption. If 
$H_i$ is as in Lemma \ref{applem1}(2), it holds by Lemma \ref{applem2} and the inductive hypothesis,
noting that the assumptions (1) and (2) hold in all intermediate Galois extensions $F'/K'$ inside $F/K$;
for (1) see e.g. \cite{squarity} Remark 2.10. 
\end{proof}

\begin{lemma}
\label{applem3}
Let $V$ be a symplectic $\F_2$-vector space of dimension $2n$. 
Then the 2-Sylow subgroup $H$ of $\Sp(V)$ stabilises an $n$-dimensional totally isotropic subspace.
\end{lemma}

\begin{proof}
We construct a totally isotropic $H$-invariant subspace $W_m$ of dimension $m$ for $0\le m\le n$ by induction. 
Take $W_0=\{0\}$. For $m>0$, let $X=\frac{W_{m-1}^\perp}{W_{m-1}}\setminus\{0\}$, 
pick an $H$-invariant vector $v\in X$ and set $W_m$ to be the span of $W_{m-1}$ and any lift of $v$ to $V$.
Such a $v$ exists because $X\ne\emptyset$ (as $m\le n$), and 
a 2-group acting on a set of odd order has a fixed point.
\end{proof}

\begin{proof}[Proof of Theorem \ref{appmain2}]
Write $G=\Gal(F/K)$. Suppose $H<G$ has 2-power order. By the above lemma, $A[2]$ has an $H$-invariant 
totally isotropic subspace. It is a standard fact that $A/F^H$ therefore 
admits an isogeny $\phi:A\to A'$ with $\phi\phi^t=[2]$.
By hypothesis, the parity conjecture holds for $A/F^H$. The result follows by Theorem \ref{appmain1}.
\end{proof}


\end{document}